\documentclass{conm-p-l}
\usepackage{graphicx,amssymb,stmaryrd,xspace,tikz-cd,color}

\RequirePackage[hidelinks]{hyperref}

\newcommand{\suchthat}{\ \text{\rm s.t.}\ }
\newcommand{\cl}{{}}
\newcommand{\ikeo}{IKEO\xspace}

\newcommand{\AAA}{A}
\newcommand{\aaa}{a}
\newcommand{\Sym}{\operatorname{Sym}}

\newcommand{\ds}[1]{\mathsf{#1}}

\usetikzlibrary{matrix,arrows}
\tikzset{
    onto/.style={/tikz/commutative diagrams/twoheadrightarrow}
}

\tikzset{
    into/.style={/tikz/commutative diagrams/hook}
}
\tikzset{
    intoL/.style={/tikz/commutative diagrams/hook'}
}

\newcommand{\pbk}[2]{|[label={
    [label distance=#1]#2-45:{\rotatebox[origin=c]{#2}{$\lrcorner$}}}]|}
\newcommand{\PBK}[1]{\pbk{-1.5mm}{#1}}
\newcommand{\SEpbk}{\PBK0}

\usepackage[v2,all]{xy}

\newcommand{\dlpullback}[1][dl]{\save*!/#1-4ex/#1:(-1,1)@^{|-}\restore}
\newcommand{\urpullback}[1][ur]{\save*!/#1-4ex/#1:(-1,1)@^{|-}\restore}
\newcommand{\drpullback}[1][dr]{\save*!/#1-4ex/#1:(-1,1)@^{|-}\restore}
\newcommand{\dpullback}[1][d]{\save*!/#1-1.8pc/#1:(-1,1)@^{|-}\restore}
\newdir{ >}{{}*!/-7pt/@{>}}

\usepackage{ stmaryrd }

\newcommand{\homomorphism}{map\xspace}
\newcommand{\homomorphisms}{maps\xspace}

\newcommand{\isoclass}{isomorphism class\xspace}
\newcommand{\isoclasses}{isomorphism classes\xspace}

\newcommand{\shortsur}{\nolinebreak[4]\!\onto\nolinebreak[4]\!}

\newcommand{\iso}{{\mathrm{iso}}}
\newcommand{\Map}{{\mathrm{Map}}}

\newcommand{\fatnerve}{\mathbf{N}}
\newcommand{\M}{M\"obius\xspace}

\newcommand{\overarrow}[1]{{\vec{#1}}}
\newcommand{\nondeg}{\overarrow}

\newcommand{\Phieven}{\Phi_{\text{\rm even}}}
\newcommand{\Phiodd}{\Phi_{\text{\rm odd}}}
\newcommand{\Beven}{\B_{\text{\rm even}}}
\newcommand{\Bodd}{\B_{\text{\rm odd}}}

\newcommand{\un}{\underline}

\newcommand{\finsup}{\mathrm{\,fin.sup.}}

\newcommand{\Bij}{\operatorname{Bij}}
\newcommand{\Vect}{\kat{Vect}}
\newcommand{\vect}{\kat{vect}}

\newcommand{\Id}{\text{Id}}

\newcommand{\fdb}{Fa{\`a} di Bruno\xspace}
\newcommand{\culf}{CULF\xspace}

\newcommand{\Decbot}[1]{\operatorname{Dec}_\bot{}\kern-2pt{#1}}
\newcommand{\Dectop}[1]{\operatorname{Dec}_\top{}\kern-2pt{#1}}

\providecommand{\norm}[1]{\left| {#1}\right|}
\providecommand{\normnorm}[1]{\left| {#1}\right|}

\newcommand{\onto}{\twoheadrightarrow}
\newcommand{\into}{\hookrightarrow}
\newcommand{\shortsetminus}{\,\raisebox{1pt}{\ensuremath{\mathbb r}\,}}

\newcommand{\ind}[1]{\underrightarrow{#1}}
\providecommand{\kat}[1]{\text{\textbf{\textsl{#1}}}}

\newcommand{\LIN}{\kat{LIN}}
\newcommand{\lin}{\kat{lin}}

\newcommand{\Inc}{\mathcal{I}}

\newcommand{\upperstar}{^{\raisebox{-0.25ex}[0ex][0ex]{\(\ast\)}}}

\newcommand{\lowershriek}{_!}

\newcommand{\isopil}{\stackrel{\raisebox{0.1ex}[0ex][0ex]{\(\sim\)}}%
			{\raisebox{-0.15ex}[0.28ex]{\(\rightarrow\)}}}

\newcommand{\tensor}{\otimes}
\newcommand{\op}{^{\text{{\rm{op}}}}}

\renewcommand{\epsilon}{\varepsilon}

\newcommand{\Set}{\kat{Set}}
\newcommand{\Grpd}{\kat{Grpd}}

\newcommand{\grpd}{\kat{grpd}}

\newcommand{\N}{\mathbb{N}}

\newcommand{\F}{\mathbb{F}}
\newcommand{\A}{\mathbb{A}}
\newcommand{\B}{\mathbb{B}}
\newcommand{\Q}{\mathbb{Q}}
\newcommand{\G}{\mathbb{G}}

\newcommand{\C}{\mathbb{C}}

\newcommand{\I}{\mathbb{I}}

\usepackage{mathrsfs}

\newcommand{\CC}{\mathscr{C}}
\newcommand{\DD}{\mathscr{D}}
\newcommand{\EE}{\mathscr{E}}
\newcommand{\FF}{\mathscr{F}}

\newcommand{\name}[1]{\ulcorner #1\urcorner}

\newcommand{\AutOmega}[2]{%
                          \Aut(#1)
           }

\newcommand{\Hom}{\operatorname{Hom}}
\newcommand{\Ext}{\operatorname{Ext}}
\newcommand{\Fun}{\operatorname{Fun}}
\newcommand{\Aut}{\operatorname{Aut}}
\newcommand{\id}{\operatorname{id}}

\newtheorem{lemma}{Lemma}[subsection]
\newtheorem{prop}[lemma]{Proposition}
\newtheorem{thm}[lemma]{Theorem}
\newtheorem{theorem}[lemma]{Theorem}
\newtheorem{cor}[lemma]{Corollary}
\theoremstyle{definition}
\newtheorem{eks}[lemma]{Example}
\newtheorem{BM}[lemma]{Remark}

\newtheorem{taller}[lemma]{$\!\!$}

\newenvironment{blanko}[1]%
{\begin{taller}{\normalfont\bfseries  #1}\normalfont}%
{\end{taller}}

\newenvironment{blanko*}[1]{\begin{list}{\bf {#1} }%
{\setlength{\labelsep}{0mm}\setlength{\leftmargin}{0mm}%
\setlength{\labelwidth}{0mm}\setlength{\listparindent}{\parindent}%
\setlength{\parsep}{\parskip}\setlength{\partopsep}{0mm}}%
\item%
}{\end{list}}

{\begin{list}{\em Definition. }%
{\setlength{\labelsep}{0mm}\setlength{\leftmargin}{0mm}%
\setlength{\labelwidth}{0mm}\setlength{\listparindent}{\parindent}%
\setlength{\parsep}{\parskip}\setlength{\partopsep}{0mm}}%
\item}{\end{list}}

\newenvironment{proof*}[1]{\begin{list}{\em #1 }%
{\setlength{\labelsep}{0mm}\setlength{\leftmargin}{0mm}%
\setlength{\labelwidth}{0mm}\setlength{\listparindent}{\parindent}%
\setlength{\parsep}{\parskip}\setlength{\partopsep}{0mm}}%
\item}{\qed\end{list}}

\thanks{%
    This work has received funding from 
	grant 10.46540/3103-00099B from the Independent Research Fund Denmark
	and from a Spanish university requalification and mobility grant
    (UP2021-034, UNI/551/2021)
    with NextGenerationEU funds.
    It was supported also by research grants
    PID2019-103849GB-I00,
    PID2020-116481GB-I00, and
    PID2020-117971GB-C22
    (AEI/FEDER, UE) of Spain,
    grants 2021-SGR-0603
    and    2021-SGR-1015
    of Catalonia,
    and through the the Severo Ochoa and Mar\'ia de Maeztu Program
    for Centers and Units of Excellence in R\&D grant number CEX2020-001084-M}

\author{Imma G\'alvez-Carrillo}
\address{Universidad de M\'alaga, IMTECH-UPC,
    and Centre de Recerca Matemàtica}
\email{imma.galvez@uma.es}

\author{Joachim Kock}
\address{University of Copenhagen, Universitat Aut\`onoma de Barcelona,
    and Centre de Recerca Matem\`atica}
\email{kock@math.ku.dk}

\author{Andrew Tonks}
\address{Universidad de Málaga}
\email{at@uma.es}

\title{Decomposition spaces in Combinatorics}

\providecommand{\simplexcategory}{%
\begin{tikzpicture}
	\begin{scope}[scale=0.87]
	  \draw (0.0, 0.0) -- (0.142, 0.284) -- (0.284, 0.0) -- (0.0, 0.0);
	  \draw (0.057, 0.0) -- +(0.113, 0.227);
	\end{scope}
\end{tikzpicture}%
}

\begin{document}

\begin{abstract}
  A decomposition space (also called $2$-Segal space) is a simplicial
  object satisfying an exactness condition weaker than the Segal condition: just
  as the Segal condition expresses composition, the new
  condition expresses decomposition.  It is a general framework for incidence
  (co)algebras.  In the present contribution, after establishing a formula for
  the section coefficients, we survey a large supply of examples, emphasising
  the notion's firm roots in classical combinatorics.  The first batch of
  examples, similar to binomial posets, serves to illustrate two key points: (1)
  the incidence algebra in question is realised directly from a decomposition
  space, without a reduction step, and reductions are often given by \culf
  functors; (2) at the objective level, the convolution algebra is a monoidal
  structure of species.  Specifically, we encounter the usual Cauchy product of
  species, the shuffle product of $\mathbb L$-species, the Dirichlet product of
  arithmetic species, the Joyal--Street external product of $q$-species and the
  Morrison `Cauchy' product of $q$-species, and in each case a power series
  representation results from taking cardinality.  The external product of
  $q$-species exemplifies the fact that Waldhausen's $S_\bullet$-construction on
  an abelian category is a decomposition space, yielding Hall algebras.  The
  next class of examples includes Schmitt's chromatic Hopf algebra, the \fdb
  bialgebra, the Butcher--Connes--Kreimer Hopf algebra of trees and several
  variations from operad theory.  Similar structures on posets and directed
  graphs exemplify a general construction of decomposition spaces from directed
  restriction species. A short appetiser on decomposition spaces of
  symmetric functions is included, featuring the base change from
  elementary symmetric functions to monomial symmetric functions, modelled
  as a span of decomposition spaces.
  We finish by computing the \M function in a few cases, exhibiting a few
  techniques,
  and commenting on certain cancellations that occur in the process of taking
  cardinality, substantiating that these cancellations are not
  always possible at the objective level.
\end{abstract}


\maketitle

\newpage

\tableofcontents

\setcounter{section}{-1}

\section{Introduction}

\begin{blanko*}{Decomposition spaces.}
  The notion of decomposition space was introduced by the authors
  \cite{GKT:DSIAMI-1,GKT:DSIAMI-2,GKT:MI} as a general setting for incidence
  algebras and \M inversion, and independently under the name $2$-Segal
  space by Dyckerhoff and
  Kapranov~\cite{Dyckerhoff-Kapranov:1212.3563}, who were motivated by
  homological algebra, representation theory and geometry.  The inherent
  simplicial nature and the broad scope of
  applications of the notion prompted a rather abstract categorical and
  homotopical treatment, with the possible side effect of obscuring its firm
  roots in combinatorics and its attractive elementary aspects.

  The purpose of the present paper is to rectify this possible shortcoming by
  explaining the combinatorial aspects of the basic theory through many
  illustrative and natural examples from classical combinatorics.  From a
  theoretical viewpoint, the natural setting for the theory of decomposition
  spaces is that of simplicial $\infty$-groupoids, but in fact the notion of
  decomposition space is interesting even for simplicial sets: there are plenty
  of natural `decomposition sets' which are not categories (or posets); some
  examples can be found in \cite{Dyckerhoff-Kapranov:1212.3563},
  \cite{Bergner-et.al:1609.02853},
  \cite{Kock-Spivak:1807.06000}, \cite{Hackney-Kock:2210.11192}.
  However, it is
  our contention that the natural level of generality for
  decomposition spaces in combinatorics is that
  of simplicial {\em groupoids}, simply because many combinatorial objects have
  symmetries, and these are taken care of elegantly by the groupoid formalism.
\end{blanko*}

\begin{blanko*}{From locally finite posets to \M categories.}
  To motivate the notion of decomposition space, let us start with incidence
  coalgebras. Since the work of Joni and Rota~\cite{JoniRotaMR544721} we know
  well that coalgebras in combinatorics arise from the ability to decompose
  structures. Very often that ability comes from something fancier, namely the
  ability to actually {\em compose} structures. A paradigmatic notion of
  composition is composition of arrows in a category, such as in particular a
  poset or a monoid. From any locally finite poset, form the free vector space
  on its intervals, and endow this with a coalgebra structure by defining the
  comultiplication as
  $$
  \Delta([x,y]) = \sum_{x \leq m \leq y} [x,m] \tensor [m,y] .
  $$
  The same construction works for elements in a monoid (with the finite
  decomposition property~\cite{Cartier-Foata}). In an appendix to
  \cite{Cartier-Foata}, Foata explains how any (reduced) incidence coalgebra of
  a poset can also be realised as the incidence coalgebra of a monoid, and
  conversely. However, it seems to be more fruitful to observe as
  Leroux~\cite{Leroux:1975}, that both are examples of incidence coalgebras of
  categories. Recall that a poset can be regarded as a category in which there
  is at most one arrow between any two given objects. To have an interval
  $[x,y]$ thus means simply that $x\leq y$, and in categorical terms this means
  that there is an arrow from $x$ to $y$. The role of elements in the interval
  $[x,y]$ is played by the possible two-step factorisations of the arrow
  $x\to y$. Recall also that a monoid is a category with only one object. 
  Leroux showed that the notions of incidence coalgebras of posets and monoids
  have a common generalisation, namely to
  locally finite categories, meaning categories in which any given arrow admits
  only finitely many $2$-step factorisations: the incidence coalgebra of such a
  category is the free vector space on its arrows, with comultiplication given
  by
  \begin{equation}\label{eq:Delta(f)}
    \Delta(f) = \sum_{b\circ a=f} a \tensor b .
  \end{equation}
  The coassociativity is a consequence of the associativity of composition
  of arrows.
\end{blanko*}

\begin{blanko*}{Functoriality.}
  One important point made by Leroux (with Content and
  Lemay~\cite{Content-Lemay-Leroux}) is that certain functors induce coalgebra
  homomorphisms. In modern language, these are the {\em \culf} functors, which
  stands for {\em conservative} and {\em unique lifting of factorisations}. That
  a functor $F:\CC\to \DD$ is conservative means that if $F(a)$ is an identity
  arrow then $a$ was already an identity arrow (see \ref{sec:bialg} below for
  more precision and discussion). Unique lifting of factorisations means that
  for an arrow $a$, there is a one-to-one correspondence between the
  factorisations of $a$ in $\CC$ and the factorisations of $F(a)$ in $\DD$.

  In the classical theory of posets, often it is not the raw incidence coalgebra
  that is most interesting, but rather a {\em reduced} incidence coalgebra,
  where two intervals are identified if they are equivalent in some specific
  sense (e.g.~isomorphic as abstract posets).  As observed in
  \cite{Content-Lemay-Leroux}, these reductions can quite often be realised by
  \culf functors.  For example, the obvious functor from the poset $(\N,\leq)$
  to the monoid $(\N,+)$, sending an `arrow' $x\leq y$ to the monoid element
  $y-x$, is \culf and realises a classical reduction: the reduced incidence
  coalgebra of the poset $(\N,\leq)$ is precisely the raw incidence coalgebra of
  the monoid $(\N,+)$.

  In the general setting of decomposition spaces, virtually {\em all} reduction
  procedures become instances of \culf functors, and furthermore, many of them
  are revealed to be instances of decalage (cf.~\ref{Dec} below),
  a general construction in simplicial homotopy theory.
\end{blanko*}

\begin{blanko*}{\M inversion.}
  \M inversion amounts to establishing the convolution invertibility of the zeta
  function; the inverse is then defined to be the \M function
  \cite{Rota:Moebius}.  Leroux~\cite{Leroux:1975} established a \M inversion
  formula for any {\em \M category}.  A category is \M when it is locally finite
  and when for each arrow there are only finitely many ways to write it as a
  composite of a chain of non-identity arrows.  This notion covers both locally
  finite posets and monoids with the finite-decomposition property.
  The formula is
  $$
  \mu = \Phieven - \Phiodd  .
  $$
  Here $\Phieven = \sum _{k \text{ even}} \Phi_k$, where $\Phi_k(f)$ is the set
  of decompositions of $f$ into a chain of $k$ composable
  non-identity arrows.  (Similarly for $k$ odd.)
\end{blanko*}

\begin{blanko*}{Simplicial viewpoints.}
  The importance of sequences of composable arrows suggests a simplicial
  viewpoint (see glossary in Appendix~\ref{sec:simplicial}), which is
  fundamental to the theory of decomposition spaces (and one of the reasons the
  theory tends to drift into homotopy theory). Recall (see~\ref{app:nerve}) that
  the nerve of a category $\CC$ is the simplicial set
  $$
  N\CC : \simplexcategory\op\to\Set
  $$
  whose set of $n$-simplices is the set of sequences of $n$ composable
  arrows in $\CC$ (allowing identity arrows).  The face maps are given by
  composing arrows (for the inner face maps) and by discarding arrows at the
  beginning or the end of the sequence (outer face maps).  The
  degeneracy maps are given by inserting an identity map in the
  sequence.

  Leroux's theory can be formulated in terms of simplicial sets, as already
  exploited by D\"ur~\cite{Dur:1986}, and many of the arguments then rely on
  certain simple pullback conditions, the first being the Segal condition which
  characterises categories among simplicial sets (cf.~\ref{Segal}).  Most
  importantly in our exploitation of this simplicial viewpoint, the
  comultiplication \eqref{eq:Delta(f)} can be written in terms of the nerve
  $N\CC$ as a push-pull formula, $\Delta=(d_2,d_0)\lowershriek\circ
  d_1\upperstar$, to be explained below.
\end{blanko*}

\begin{blanko*}{Objective method.}
  \M inversion is a versatile algebraic counting device. The fact that the
  formula is always given by an alternating sum illustrates one of the great
  features of algebra over bijective combinatorics: the existence of additive
  inverses. On the other hand, it is well appreciated that bijective proofs in
  general represent deeper insight than purely algebraic proofs.

  There is a rather general method for lifting algebraic identities to
  bijections of sets, which one may try to apply whenever the identity takes
  place in the vector space with basis the set of
  \isoclasses of objects. This is the
  so-called objective method, pioneered in this context by Lawvere and
  Menni~\cite{LawvereMenniMR2720184}, working directly with the combinatorial
  objects rather than their numbers, using linear algebra with coefficients in
  $\Set$ rather than a ring or field.

  To illustrate this, observe that a vector in the free vector space on a set
  $S$ is just a collection of scalars indexed by (a finite subset of) $S$. The
  objective counterpart is a family of sets indexed by $S$, i.e.~an object in
  the slice category $\Set_{/S}$. The notion of cardinality has a natural
  extension to families of finite sets: the cardinality of a family of finite
  sets indexed by some set $B$ is a $B$-indexed family of natural numbers, and
  is in particular an element in the vector space with basis $B$. Finiteness
  issues enter the picture now and should be taken proper care of, see below.

  Linear maps at this level are given by spans $S \leftarrow M \to T$, which
  are, in more abstract terms, the {\em linear functors}, i.e.~functors between
  slices preserving sums and certain other colimits. Indeed, the pullback
  formula for composition of spans turns out to correspond precisely to matrix
  multiplication. Spans have cardinalities, which are linear maps.

  The \M inversion principle states an equality between certain linear maps
  (elements in the incidence algebra). At the objective level, such an equality
  can be expressed as a levelwise bijection of the spans of sets that represents
  those linear functors. In this way, the algebraic identity is revealed to be
  the cardinality of a bijection of sets, which carry much more structural
  information.

  Lawvere and Menni~\cite{LawvereMenniMR2720184} established an objective
  version of the \M inversion principle for \M categories in the sense of
  Leroux~\cite{Leroux:1975}.
  A trick is needed to account for the signs:
  where the algebraic identity states that $\zeta$ is convolution
  invertible with inverse $\mu = \Phieven-\Phiodd$:
  $$
  \zeta * (\Phieven-\Phiodd) = \epsilon ,
  $$
  to avoid the minus sign, that term has to be moved
  to the other side of the equation, and the equivalent statement
  $$
  \zeta * \Phieven
  \;\;=\;\; \epsilon\;\; +\;\; \zeta * \Phiodd
  $$
  can be realised as an explicit bijection of sets \cite{LawvereMenniMR2720184}.
\end{blanko*}

\begin{blanko*}{From sets to groupoids.}
  It is useful now to generalise from sets to groupoids, in order to get a
  better treatment of symmetries.  A prominent example illustrating this is the
  \fdb coalgebra (treated in detail in \ref{sec:fdb}): it ought to be the
  incidence coalgebra of (a skeleton of) the category of finite sets and
  surjections, but since finite sets have symmetries, there are too many
  factorisations, even of identity arrows.  This is solved by passing to {\em
  fat nerves} (cf.~\ref{fatnerve}).  The fat nerve of a category is the
  simplicial groupoid
  $$
  \fatnerve \CC : \simplexcategory\op\to \Grpd
  $$
  whose groupoid of $n$-simplices is the groupoid whose objects are
  sequences of $n$ composable arrows,
  and whose arrows are isomorphisms at each level, as pictured here:
  $$
  \xymatrix{\cdot \ar[r] \ar[d]^*-[@]=0+!L{\scriptstyle \sim}
    & \cdot \ar[d]^*-[@]=0+!L{\scriptstyle \sim} \ar[r]
    & \cdot \ar[r] \ar[d]^*-[@]=0+!L{\scriptstyle \sim} &\cdots\ar[r]
    & \cdot \ar[d]^*-[@]=0+!L{\scriptstyle \sim} \\
  \cdot \ar[r] & \cdot\ar[r] & \cdot\ar[r] &\cdots\ar[r]& \cdot
  }$$

  The slice categories now have to be groupoid slices $\Grpd_{/X}$ instead of
  set slices. Linear algebra works well at this level of generality too (see
  Appendix~\ref{sec:groupoids}), and there is a notion of homotopy cardinality
  which is invariant under homotopy equivalence. This approach was initiated by
  Baez and Dolan~\cite{Baez-Dolan:finset-feynman} and further developed by Baez,
  Hoffnung and Walker~\cite{Baez-Hoffnung-Walker:0908.4305}. A cleaner homotopy
  version of their formalism was introduced in \cite{GKT:HLA}, where in
  particular the notion of homotopy sum is exploited. The upgrade from sets to
  groupoids is essentially straightforward, as long as the notions involved are
  taken in a correct homotopy sense, as recalled in
  Appendix~\ref{sec:groupoids}: bijections of sets are replaced by equivalences
  of groupoids; the slices playing the role of vector spaces are homotopy
  slices, the pullbacks and fibres involved in the functors are homotopy
  pullbacks and homotopy fibres, and the sums are homotopy sums (i.e.~colimits
  indexed by groupoids, just as ordinary sums are colimits indexed by sets).
\end{blanko*}


\begin{blanko*}{Decomposition spaces and their incidence (co)algebras.}
  The final abstraction step, which became the starting point for our
  work~\cite{GKT:DSIAMI-1,GKT:DSIAMI-2,GKT:MI}, and which is where the present
  paper starts, is to notice that coassociative coalgebras and a \M inversion
  principle can be obtained from simplicial groupoids more general than those
  satisfying the Segal condition. We call these {\em decomposition spaces};
  Dyckerhoff and Kapranov~\cite{Dyckerhoff-Kapranov:1212.3563} call them
  $2$-Segal spaces.\footnote{Originally they were called {\em
  unital} $2$-Segal spaces, but it was later
  shown~\cite{Feller-Garner-Kock-Proulx-Weber:1905.09580} that the
  unitality condition is automatic. See Hackney's
  contribution~\cite{Hackney:thisvolume} to this volume for lucid
  exposition of that issue.
  The term $2$-Segal space
    may well be the most practical in the broader picture,
    in particular in view of the generalisations
  to $k$-Segal spaces for $k>2$ (see Dyckerhoff's
  contribution~\cite{Dyckerhoff:thisvolume} in this volume), but from the
  viewpoint of combinatorics we feel the decomposition space terminology
  has its merits.} Whereas the Segal condition is the expression of the
  ability to compose morphisms, the new condition is about the ability to
  decompose, which of course in general is easier to achieve than
  composability---indeed every Segal space is a decomposition space
  (Proposition~\ref{prop:segalisdecomp}).

  The decomposition-space axiom on a simplicial groupoid $X:
  \simplexcategory\op\to\Grpd$ is expressly the condition needed for a canonical
  coalgebra structure to be induced on the slice category $\Grpd_{/X_1}$.  The
  comultiplication is the linear functor
  $$
  \Delta : \Grpd_{/X_1} \to \Grpd_{/X_1} \tensor  \Grpd_{/X_1}
  $$
  given by the span
  $$
  X_1 \stackrel{d_1}\longleftarrow X_2 \stackrel{(d_2,d_0)}\longrightarrow 
  X_1\times X_1
  $$
  (with reference to general simplicial notation, reviewed in 
  Appendix \ref{sec:simplicial}).
  This can be read as saying that the comultiplication of an edge $f\in X_1$
  returns the sum
  of all pairs of edges $(a,b)$ that are the short edges of a triangle with long
  edge $f$.  In the case that $X$ is the fat nerve of a category, this is the 
  homotopy sum of
  all pairs $(a,b)$ of arrows with composite $b\circ a=f$, just as 
  in~\eqref{eq:Delta(f)}.
\end{blanko*}

\begin{blanko*}{Incidence coalgebras, without the need of reduction.}
  It is likely that all incidence (co)algebras can be realised directly (without
  imposing a reduction) as incidence (co)algebras of decomposition spaces.  The
  decomposition space is found by analysing the reduction step.  For example,
  D\"ur~\cite{Dur:1986} realises the $q$-binomial coalgebra as the reduced
  incidence coalgebra of the category $\vect^{\text{inj}}$ of finite dimensional
  vector spaces over a finite field and linear injections, by imposing the
  equivalence relation identifying two linear injections if their quotients are
  isomorphic.  Trying to realise the reduced incidence coalgebra directly as a
  decomposition space immediately leads to Waldhausen's
  $S_\bullet$-construction, a basic construction in $K$-theory: the $q$-binomial
  coalgebra is directly the incidence coalgebra of $S_\bullet(\vect)$.
\end{blanko*}

\begin{blanko*}{Hall algebras.}
  The $q$-binomial coalgebra fits into a general class of examples: for any
  abelian category (or even stable $\infty$-category~\cite{GKT:DSIAMI-1}), the
  Waldhausen $S_\bullet$-construction is a decomposition space (which is not
  Segal).  Under the appropriate finiteness conditions, the resulting incidence
  algebras include the Hall algebras, as well as the derived Hall algebras first
  constructed by To\"en~\cite{Toen:0501343}.
  This class of examples plays a key role in the work of Dyckerhoff and
  Kapranov~\cite{Dyckerhoff:1505.06940,Dyckerhoff-Kapranov:1403.5799,
        Dyckerhoff-Kapranov:1306.2545,Dyckerhoff-Kapranov:1212.3563};
  we refer to their work for the remarkable richness of the Hall algebra aspects
  of the theory. See also Bergner et.~al~\cite{Bergner-et.al:1609.02853},
  Walde~\cite{Walde:1611.08241}, Young~\cite{Young:1611.09234},
  Poguntke~\cite{Poguntke:1808.04165}, and the contribution of Cooper and
  Young~\cite{Cooper-Young:thisvolume} for further pointers in this direction.
\end{blanko*}

\begin{blanko*}{Organisation of the paper.}
  In Section~\ref{sec:decomp} we start out with a short, self-contained summary
  of the basic notions and results of the theory of decomposition spaces,
  emphasising combinatorial aspects: the definition in Subsection
  \ref{sec:def-decomp}, their incidence coalgebras in \ref{sec:coalg}, and the
  convolution product in \ref{sec:conv}. In~\ref{sec:tioncoeff} we introduce
  techniques for computing section coefficients, under suitable finiteness
  conditions, with a closed formula for the case of Segal spaces.
  In \ref{sec:bialg} we briefly review the notion of \culf functor, relevant
  because these induce coalgebra homomorphisms. We exploit decalage (a key
  example of \culf functor) to establish a criterion for local discreteness,
  essentially the situation in which the section coefficients are integral. We
  introduce monoidal decomposition spaces as \culf monoidal structures. These
  induce bialgebras instead of just coalgebras. A running example in this
  section is Schmitt's Hopf algebra of graphs~\cite{Schmitt:1994} (called the
  chromatic Hopf algebra by Aguiar, Bergeron and
  Sottile~\cite{Aguiar-Bergeron-Sottile}), an archetypical example of a
  coalgebra which cannot be the (raw) incidence coalgebra of a category, but is
  readily obtained as the incidence coalgebra of a decomposition space. It
  illustrates well the combinatorial meaning of the decomposition space axiom
  (Example \ref{ex:graphs-decomp}), the mechanism by which the coalgebra
  structure arises (\ref{ex:graphs-coalg}), and the \culf monoidal structure
  that makes it a bialgebra (\ref{ex:graphs-bialg}).

In Section~\ref{sec:ex}, we first go through some very basic examples,
which correspond closely to power series representations of the binomial
posets of Doubilet--Rota--Stanley~\cite{Doubilet-Rota-Stanley}, and show
how the objective version of these classical incidence algebras
amounts to
monoidal structures on various kinds of species. We emphasise decalage as a
general principle behind classical reduction procedures. The case of the
Joyal--Street external product of $q$-species leads to the general
treatment of the Waldhausen $S_\bullet$-construction as a decomposition
space in \ref{sec:Wald}. In~\ref{sec:fdb} we revisit the Fa\`a di Bruno
bialgebra. Classically it is the reduced incidence bialgebra of the poset
of set partitions (reduction modulo type equivalence), but can also be
obtained directly from the category of surjections. This suggests that
again the reduction step is a decalage, but the relationship turns out to
be more subtle: it is a \culf functor but not directly a decalage.
In~\ref{sec:trees} we treat examples related to trees and graphs, starting
with the Butcher--Connes--Kreimer Hopf algebra of
trees~\cite{Connes-Kreimer:9808042}, another example of an incidence
coalgebra which cannot be the (raw) incidence coalgebra of a category. We
proceed to treat operadic variations, including incidence bialgebras of
general operads, as well as related constructions with directed graphs
(cf.~Manchon~\cite{Manchon:MR2921530} and Manin~\cite{Manin:MR2562767}). We
briefly explain how most of the examples treated in this subsection are
subsumed in the notion of decomposition spaces from restriction species
of Schmitt~\cite{Schmitt:hacs} as well as directed restriction
species, treated in detail elsewhere~\cite{GKT:restriction},
and comment also briefly on hereditary species (also
from~\cite{Schmitt:hacs}) and directed hereditary species from the
viewpoint of decomposition spaces. To finish Section~\ref{sec:ex} we give a
brief introduction to symmetric functions from the viewpoint of
decomposition spaces. The highlight in this short account is the base
change from elementary symmetric functions to monomial symmetric
functions,
modelled at the objective level by means of an \ikeo-\culf span of
decomposition spaces.

In Section~\ref{sec:M} we come to M\"obius inversion, and need first to recall a
few notions from \cite{GKT:DSIAMI-2}: complete decomposition spaces and
nondegeneracy in~\ref{sec:compl}, and the notion of locally finite length and
the general M\"obius inversion formula in~\ref{sec:length}.
In~\ref{sec:cancellation} we compute the M\"obius function in a few easy cases,
and comment on certain cancellations that occur in the process of taking
cardinality, substantiating that these cancellations are not possible at the
objective level. This is related to the distinction between bijections and
natural bijections.

In Appendix~\ref{sec:groupoids} we provide background on groupoids necessary
to understand groupoid slices as the objective analogue of vector spaces, and
linear functors and spans as the objective analogue of linear maps.  We also
explain how to recover the vector space level via taking homotopy cardinality.

In Appendix~\ref{sec:simplicial} we briefly recall the simplicial
machinery,
which is an essential tool in our undertakings, with special emphasis on the
relationship with simplicial complexes.  In particular we explain the nerve
and the fat nerve of a small category, whereby the simplicial setting covers
the cases of categories, and in particular posets and monoids.
\end{blanko*}

\begin{blanko*}{Note.}
  Most of the material in this paper was originally Section 5 of the
  large single manuscript {\em Decomposition spaces, incidence algebras and \M
  inversion} \cite{GKT:1404.3202}. That manuscript was split into five papers
  that were published
  \cite{GKT:DSIAMI-1,GKT:DSIAMI-2,GKT:MI,GKT:HLA,GKT:restriction} plus the
  present paper, which was posted to the arXiv in 2016. It was not submitted for
  publication at the time, since it was felt that it ought to be expanded with
  an account of symmetric functions from the decomposition-space viewpoint. The
  development of that theory suffered substantial delays, but now, on the
  occasion of the Banff workshop proceedings, we have finally included a very
  brief account (just an appetiser) of the decomposition-space viewpoint on
  symmetric functions as Subsection~\ref{sec:Sym}. We have also taken the
  opportunity to update the exposition with some remarks and pointers to some
  other developments that have taken place on decomposition spaces in
  combinatorics since 2016.
\end{blanko*}

\begin{blanko*}{Acknowledgements.}
  We are grateful to Andr\'e Joyal, Kurusch Ebrahimi-Fard, Mark Weber,
  Louis Carlier, Alex Cebrian, Wilson Forero, Philip Hackney,
  and Darij Grinberg, for very useful feedback and help.
\end{blanko*}

\section{Decomposition spaces and incidence coalgebras}

\label{sec:decomp}

\subsection{Segal spaces and decomposition spaces}

\label{sec:def-decomp}

Segal spaces and decomposition spaces are simplicial groupoids
$X:\simplexcategory\op\to\Grpd$ satisfying
certain exactness properties. We refer to Appendix~\ref{sec:simplicial} for
a glossary on simplicial groupoids.

\begin{blanko}{Segal spaces (Segal groupoids).}
  A simplicial groupoid $X$ is called a {\em Segal space}, or a {\em Segal
    groupoid}, when all squares of the form
  $$\xymatrix{
     X_{n+1} \ar[r]^{d_0}\ar[d]_{d_{n+1}} & X_n \ar[d]^{d_n} \\
     X_n \ar[r]_{d_0} & X_{n-1}
  }$$
  are (homotopy) pullbacks (see Appendix \ref{pbk}).

  The most important such square is
  \begin{equation}\label{eq:Segal-sq}
    \vcenter{\xymatrix{
     X_2 \drpullback \ar[r]^{d_0}\ar[d]_{d_2} & X_1 \ar[d]^{d_1} \\
     X_1 \ar[r]_{d_0} & X_0
    }}
  \end{equation}
  which says that $X_2$ can be identified with the groupoid
  $X_1 \times_{X_0} X_1$ of composable pairs of `arrows'. This is satisfied by
  the nerve or the fat nerve of a small category.

  For a Segal space $X$, the vector space with basis $\pi_0 X_1$ has a coalgebra
  structure analogous to \eqref{eq:Delta(f)}.
\end{blanko}

It turns out~\cite{GKT:DSIAMI-1} that simplicial groupoids other than Segal
spaces induce coalgebras. These are the decomposition spaces, which are
characterised by a weaker exactness condition than the Segal condition. To give
the explicit definitions we need first some simplicial terminology. We refer to
Appendix~B for notation (which is standard). \medskip

\begin{blanko}{Face and degeneracy maps, active and inert
    maps.}\label{generic-and-free}
  The simplex category $\simplexcategory$ (see Appendix B) has an active-inert
  factorisation system (an example of the general categorical notion of
  generic-free factorisation system, important in monad theory
  \cite{Weber:TAC13,Weber:TAC18}). An arrow $a: [m]\to [n]$ in
  $\simplexcategory$ is \emph{active} (also called {\em generic}) when it
  preserves end-points, $a(0)=0$ and $a(m)=n$; and it is \emph{inert}
  (also
  called \emph{free}) if it is distance preserving, $a(i+1)=a(i)+1$ for
  $0\leq i\leq m-1$. The active maps are generated by the codegeneracy maps
  $s^i : [n+1] \to [n]$ and by the {\em inner} coface maps $d^i : [n-1]\to [n]$,
  $0 < i < n$, while the inert maps are generated by the {\em outer} coface maps
  $d^\bot := d^0$ and $d^\top:= d^n$. Every morphism in $\simplexcategory$
  factors uniquely as an active map followed by an inert map. Furthermore, it is
  a basic fact \cite{GKT:DSIAMI-1} that active and inert maps in
  $\simplexcategory$ admit pushouts along each other, and the resulting maps are
  again active and inert.
  For a simplicial groupoid $X:\simplexcategory\op\to\Grpd$, the images of
  active and inert maps in $\simplexcategory$ are again called active and inert.
\end{blanko}

\begin{blanko}{Decomposition spaces \cite{GKT:DSIAMI-1}.}\label{decomp}
  A simplicial groupoid $X: \simplexcategory\op\to\Grpd$ is called a {\em
    decomposition space} when it takes active-inert pushouts to
  pullbacks.

  One can break this down to checking that the following simplicial-identity
  squares are pullbacks. In the diagrams, the indices are
  $n\geq 0$ and $0 \leq k \leq n$, so that all horizontal arrows are 
  active maps (and the vertical arrows are inert maps):
  \begin{equation}
    \vcenter{
    \xymatrix@R-1ex{
      X_{n+1} \drpullback \ar[r]^{s_{k+1}} \ar[d]_{d_\bot} 
      & X_{n+2} \ar[d]_{d_\bot} 
      & \ar[l]_{d_{k+2}} \dlpullback X_{n+3} \ar[d]^{d_\bot} 
      \\
       X_n \ar[r]_{s_k} & X_{n+1}  & \ar[l]^{d_{k+1}} X_{n+2}
    }}
    \qquad
    \vcenter{\xymatrix@R-1ex{
       X_{n+1} \drpullback \ar[r]^{s_k} \ar[d]_{d_\top} & X_{n+2} 
       \ar[d]_{d_\top} & \ar[l]_{d_{k+1}} \dlpullback X_{n+3} \ar[d]^{d_\top} \\
      X_n \ar[r]_{s_k} &  X_{n+1}  & \ar[l]^{d_{k+1}} X_{n+2}
    }}
  \end{equation}

  The most important cases are the four squares that involve $d_1:X_2 \to X_1$
  (corresponding to composition of arrows in a category) and $s_0: X_0 \to X_1$
  (corresponding to the identity arrows in a category):
  \begin{equation}\label{eq:firstfew}
    \vcenter{
    \xymatrix@R-1ex{
       X_2 \ar[d]_{d_\bot} & \ar[l]_{d_2} \dlpullback X_3 \ar[d]^{d_\bot} \\
       X_1  & \ar[l]^{d_1} X_2
    }}
    \;\;
    \vcenter{
    \xymatrix@R-1ex{
       X_2 \ar[d]_{d_\top} & \ar[l]_{d_1} \dlpullback X_3 \ar[d]^{d_\top} \\
       X_1  & \ar[l]^{d_1} X_2
    }}
    \quad
    \vcenter{\xymatrix@R-1ex{
       X_1 \drpullback \ar[r]^{s_1} \ar[d]_{d_\bot} &  X_2 \ar[d]^{d_\bot} \\
       X_0 \ar[r]_{s_0} &  X_1
    }}
    \;\;
    \vcenter{\xymatrix@R-1ex{
       X_1 \drpullback \ar[r]^{s_0} \ar[d]_{d_\top} &  X_2 \ar[d]^{d_\top} \\
       X_0 \ar[r]_{s_0} &  X_1
    }}
  \end{equation}
  We shall see shortly that the first two pullback squares are essential
  ingredients in getting coassociativity of the {\em incidence coalgebra} of
  $X$. The last two pullback squares express counitality, but it
    has turned out, by a theorem of Feller
    et al.~\cite{Feller-Garner-Kock-Proulx-Weber:1905.09580}
  that they are automatically pullbacks if just the two first squares are
  pullbacks.
\end{blanko}

Although the Segal axiom squares are quite different from the decomposition
space axioms, it is not difficult to prove the following, which shows that the
new setting of decomposition spaces does cover the cases of nerves and fat
nerves of categories.

\begin{prop}[{\cite[Proposition 3.5]{GKT:DSIAMI-1},
    \cite[Proposition 5.2.6]{Dyckerhoff-Kapranov:1212.3563}}]
  \label{prop:segalisdecomp}
Every Segal space is a decomposition space.
\end{prop}

\begin{blanko}{Example (Schmitt's Hopf algebra of
    graphs).}\label{ex:graphs-decomp}
  We give an example of a decomposition space which is not a Segal space, to
  illustrate the combinatorial meaning of the pullback condition: it is about
  structures that can be decomposed but not always composed. We shall continue
  this example in \ref{ex:graphs-coalg}, and see that it corresponds to the Hopf
  algebra of graphs of Schmitt~\cite{Schmitt:1994}.

  We define a simplicial groupoid $X$ by taking $X_1$ to be the groupoid of
  graphs (admitting multiple edges and loops), and more generally letting $X_k$
  be the groupoid of graphs with an ordered partition of the vertex set into $k$
  parts (possibly empty). In particular we have $X_0=1$, the contractible
  groupoid, consisting only of the empty graph.

  These groupoids form a simplicial object: the
  outer face maps delete the first or last part of the graph, and the inner face
  maps join adjacent parts.  The degeneracy maps insert an empty part.
  The simplicial identities are readily checked.

  It is clear that $X$ is not a Segal space: for the Segal
  square~\eqref{eq:Segal-sq}
    $$\xymatrix{
      X_2 \ar[r]^{d_0}\ar[d]_{d_2} & X_1 \ar[d]^{d_1} \\
      X_1 \ar[r]_{d_0} & X_0 }$$ to be a pullback would mean that a graph with a
    two-part partition could be reconstructed uniquely from knowing the two
    parts individually. But this is not true, because the two parts individually
    contain no information about the edges going between them.

  One can check that it {\em is} a decomposition space: that the square
  $$\xymatrix{
    X_2 \ar[d]_{d_0} & X_3 \dlpullback \ar[l]_{d_2} \ar[d]^{d_0} \\
    X_1 & X_2 \ar[l]^{d_1} }$$
  is a pullback is to say that a graph with a
  three-part partition ($\in X_3$) can be reconstructed uniquely from a pair of
  elements in $X_2$ with common image in $X_1$ (under the indicated face maps).
  The following picture represents elements corresponding to each other in the
  four groupoids.

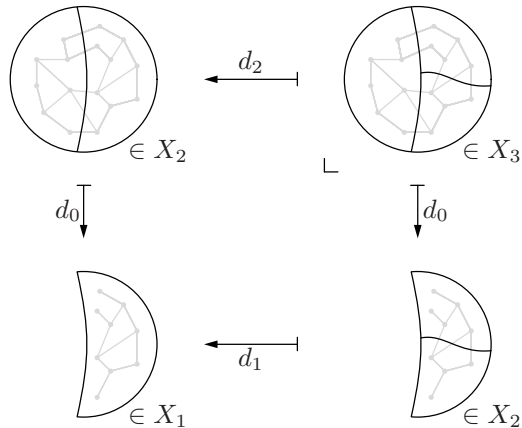
\colorlet{ourgrey}{black!18}

\tikzset{
  greydot/.pic={
	\fill (0,0) circle[radius=0.04];
  }
}

\tikzset{
  biggraph/.pic={
	\begin{scope}[ourgrey, line width=0.5pt]
	  \draw (0.35, 0.263) pic {greydot}
	  -- (0.175, 0.438) pic {greydot}
	  -- (-0.21, 0.263) pic {greydot}
	  -- (-0.263, 0.525) pic {greydot}
	  -- (0.21, 0.7) pic {greydot}
	  -- (0.525, 0.525) pic {greydot}
	  -- (0.7, 0.175) pic {greydot}
	  -- (0.7, -0.263) pic {greydot}
	  -- (0.35, -0.35) pic {greydot}
	  -- (0.175, -0.7) pic {greydot}
	  -- (-0.175, -0.7) pic {greydot}
	  -- (-0.525, -0.438) pic {greydot}
	  -- (-0.613, -0.088) pic {greydot}
	  -- (-0.613, 0.263) pic {greydot}
	  -- (-0.175, -0.14) pic {greydot}
	  -- (0.175, -0.175) pic {greydot};	
	  \draw (0.7, 0.175) -- (0.175, -0.175);	
	  \draw (0.35, -0.35) -- (0.175, -0.175)  -- (0.35, 0.263) -- (0.525, 0.525);	
	  \draw (-0.525, -0.438) -- (-0.175, -0.14) -- (0.175, -0.7);	
	  \draw  (-0.613, 0.263) -- (-0.21, 0.263);
    \end{scope}
  }
}

\tikzset{
  smallgraph/.pic={
	\begin{scope}[ourgrey, line width=0.5pt]
	  \draw (0.35, 0.263) pic {greydot}
	  -- (0.175, 0.438) pic {greydot};
	  \draw (0.21, 0.7) pic {greydot}
	  -- (0.525, 0.525) pic {greydot}
	  -- (0.7, 0.175) pic {greydot}
	  -- (0.7, -0.263) pic {greydot}
	  -- (0.35, -0.35) pic {greydot}
	  -- (0.175, -0.7) pic {greydot};
	  \draw (0.7, 0.175) -- (0.175, -0.175);
	  \draw (0.35, -0.35) -- (0.175, -0.175) pic {greydot}
	  -- (0.35, 0.263) -- (0.525, 0.525);
     \end{scope}
  }
}

\begin{center}
  \vspace*{12pt}
  \begin{tikzpicture}[line width=0.5pt]
	\footnotesize
      
	\begin{scope}[shift={(0.0, 0.0)}]
	  \draw (0.963, -0.963) node {$\in X_1$};
	  \draw (0.0, 0.0) pic {smallgraph};      
	  \draw (0.0, 0.0)+(-95:0.963) arc[start angle=-95, end angle=95, radius=0.963];
	  \draw (-0.088, 0.954)  .. controls (0.088, 0.175) and (0.088, -0.175) .. (-0.088, -0.954);
	\end{scope}

	\begin{scope}[shift={(4.375, 0.0)}]
	  \draw (0.963, -0.963) node {$\in X_2$};
	  \draw (0.0, 0.0) pic {smallgraph};
	  \draw (0.0, 0.0)+(-95:0.963) arc[start angle=-95, end angle=95, radius=0.963];
	  \draw (-0.088, 0.954) .. controls (0.088, 0.175) and (0.088, -0.175) .. (-0.088, -0.954);
	  \draw (0.044, 0.088) .. controls (0.35, 0.14) and (0.525, -0.123) .. (0.963, -0.088);
	\end{scope}

	\begin{scope}[shift={(0.0, 3.5)}]
	  \draw (0.963, -0.963) node {$\in X_2$};
	  \draw (0.0, 0.0) pic {biggraph};
	  \draw (0.0, 0.0) circle (0.963);
	  \draw (-0.088, 0.954) .. controls (0.088, 0.175) and (0.088, -0.175) .. (-0.088, -0.954);
	\end{scope}

	\begin{scope}[shift={(4.375, 3.5)}]
	  \draw (0.963, -0.963) node {$\in X_3$};
	  \draw (0.0, 0.0) pic {biggraph};
	  \draw (0.0, 0.0) circle (0.963);
	  	  \draw (-0.088, 0.954) .. controls (0.088, 0.175) and (0.088, -0.175) .. (-0.088, -0.954);
	  \draw (0.044, 0.088) .. controls (0.35, 0.14) and (0.525, -0.123) .. (0.963, -0.088);
	\end{scope}

	\begin{scope}[shift={(3.15, 2.275)}]
	  \draw (0.0, 0.0) -- (0.0, 0.175);
	  \draw (0.0, 0.0) -- (0.175, 0.0);
	\end{scope}
      
	\draw (2.8, 3.413) -- + (0.0, 0.175);
	\draw[->] (2.8, 3.5) -- + (-1.225, 0.0);
	\draw (2.188, 3.71) node {$d_2$};
	\draw (2.8, -0.088) -- + (0.0, 0.175);
	\draw[->] (2.8, 0.0) -- + (-1.225, 0.0);
	\draw (2.188, -0.21) node {$d_1$};
	\draw (-0.088, 2.1) -- + (0.175, 0.0);
	\draw[->] (0.0, 2.1) -- + (0.0, -0.7);
	\draw (-0.21, 1.75) node {$d_0$};
	\draw (4.288, 2.1) -- + (0.175, 0.0);
	\draw[->] (4.375, 2.1) -- + (0.0, -0.7);
	\draw (4.603, 1.75) node {$d_0$};
  \end{tikzpicture}
  \vspace*{12pt}
\end{center}
  The horizontal maps join the last two parts of the partition.  The vertical
  maps forget the first part.  Clearly the diagram commutes.  To reconstruct the
  graph with a three-part partition (upper right-hand corner), most of the
  information is already available in the upper left-hand corner, namely the
  underlying graph and all the subdivisions except the one between part 2 and
  part 3.  But this information is precisely available in the lower right-hand
  corner, and their common image in $X_1$ says precisely how this missing piece
  of information is to be implanted.
\end{blanko}

\subsection{Incidence coalgebras of decomposition spaces}

\label{sec:coalg}

We now turn to the incidence coalgebra (with groupoid coefficients)
associated to any decomposition space, explaining
the origin of the decomposition space axioms.

The incidence coalgebra associated to a decomposition space $X$ will be a
comonoid object in the symmetric monoidal $2$-category $\LIN$
(whose objects are groupoid slices and whose morphisms are linear functors---see 
\ref{sec:LIN}),
and the underlying object is $\Grpd_{/X_1}$.
Since $\Grpd_{/X_1} \tensor \Grpd_{/X_1} = \Grpd_{/X_1\times X_1}$,
and since linear functors are given by spans, to define a comultiplication
functor is to give a span
$$
X_1 \leftarrow M \to X_1 \times X_1.
$$

\begin{blanko}{Comultiplication and counit.}\label{comult}
  For $X$ a decomposition space, we can consider the following structure maps on
  $\Grpd_{/X_1}$.  The span
  \begin{equation}\label{comult-span}
  \xymatrix{
  X_1  & \ar[l]_{d_1}  X_2\ar[r]^-{(d_2,d_0)} &  X_1\times X_1 
  }
  \end{equation}
  defines a linear functor, the {\em comultiplication}
  \begin{eqnarray*}
    \Delta : \Grpd_{/ X_1} & \longrightarrow & \Grpd_{/( X_1\times X_1)}  \\
    (T\stackrel t\to X_1) & \longmapsto & (d_2,d_0) \lowershriek \circ d_1{} 
    \upperstar(t) .
  \end{eqnarray*}
  Likewise, the span
  \begin{equation}\label{counit-span}
  \xymatrix{
  X_1  & \ar[l]_{s_0}  X_0\ar[r]^z &  1 
  }
  \end{equation}
  defines a linear functor, the {\em counit}
  \begin{eqnarray*}
    \epsilon : \Grpd_{/ X_1} & \longrightarrow & \Grpd  \\
    (T\stackrel t\to X_1) & \longmapsto & z \lowershriek \circ s_0{}\upperstar(t) .
  \end{eqnarray*}
  
  We proceed to explain that coassociativity follows from the
  decomposition space axiom.  The coalgebra $(\Grpd_{/X_1},\Delta,\epsilon)$ is
  called the {\em incidence coalgebra} of the decomposition space $X$.  (Note
  that in the classical incidence-algebra literature (e.g.~\cite{Rota:Moebius},
  \cite{Leroux:1975}), the counit is often denoted $\delta$.)
\end{blanko}

\begin{blanko}{Coassociativity.}
  The comultiplication and counit maps on $\Grpd_{/X_1}$, defined
  in~\ref{comult} for any simplicial groupoid $X$, become coassociative and
  counital when the decomposition space axioms hold for $X$.  The desired
  coassociativity diagram (which should commute up to equivalence)
  $$
  \xymatrix@!C=25ex@R-1.5ex{
  \Grpd_{/X_1}\dto_\Delta\rto^-\Delta&\Grpd_{/X_1\times X_1}\dto^{\Delta\tensor\id}\\
  \Grpd_{/X_1\times X_1}\rto_-{\id\tensor\Delta}&\Grpd_{/X_1\times X_1\times X_1}}
  $$
  is induced by the solid spans in the diagram
  $$\xymatrix@C+6ex@R+0ex{
  X_1                  & X_2 \ar[l]_{d_1}\ar[r]^{(d_2,d_0)}   & X_1 \times X_1 \\
  X_2 \ar[u]^{d_1} \ar[d]_{(d_2,d_0)}&
  X_3\dlpullback\urpullback
  \ar@{-->}[l]_{d_2}
  \ar@{-->}[u]^{d_1}
  \ar@{-->}[d]^{(d_2 d_2,d_0)}
  \ar@{-->}[r]_{(d_3,d_0 d_0)}
          & X_2 \times X_1 \ar[u]_{d_1\times \id}\ar[d]^{(d_2,d_0)\times \id}\\
  X_1\times X_1 & X_1 \times X_2 \ar[l]^-{\id\times d_1}
  \ar[r]_-{\id\times (d_2,d_0)} &X_1 \times X_1 \times X_1.
  }
  $$
  Coassociativity will follow from the Beck--Chevalley Lemma \ref{lem:beck-chevalley} if the dashed
  part of the diagram can be established with pullbacks as indicated.
  Consider the upper right-hand square: it will be a pullback if and only if
  its composite with the first projection is a pullback:
  $$\xymatrix@C+6ex@R+0ex{
  X_2 \ar[r]^-{(d_\top,d_0)}   & X_1 \times X_1  \ar[r]^-{\text{pr}_1} & X_1 \\
  X_3\urpullback\ar[u]^{d_1}\ar[r]_-{(d_\top,d_0 d_0)}
    & X_2 \times X_1 \urpullback \ar[u]^{d_1\times \id} \ar[r]_-{\text{pr}_1} & 
     \ar[u]_{d_1} X_2   .
  }
  $$
  Saying that this composite outer square $d_\top d_1=d_1d_\top$ is a pullback
  is precisely one of the first decomposition space axioms~\eqref{eq:firstfew}.

  If one is just interested in coassociativity at the level of $\pi_0$, this
  pullback and its twin, $d_\bot d_2=d_1d_\bot$, are all that are needed, as was
  the case in the work of To\"en~\cite{Toen:0501343} who dealt with the case
  where $X$ is the Waldhausen $S_\bullet$ construction of a dg category.  On the
  other hand, it is interesting to analyse when the coassociativity is actually
  homotopy coherent at the level of groupoid slices.  It is proved in
  \cite[Theorem 7.3]{GKT:DSIAMI-1} that this is true when all the decomposition
  space axioms hold:
\end{blanko}

\begin{thm}
  If $X$ is a decomposition space then $\Grpd_{/X_1}$ has the structure of
  strong homotopy comonoid in the symmetric monoidal $2$-category $\LIN$, with the
  comultiplication and counit defined by the spans \eqref{comult-span} and
  \eqref{counit-span}.
\end{thm}

\begin{blanko}{Example: Schmitt's Hopf algebra of graphs,
    continued.}\label{ex:graphs-coalg}
  The following coalgebra is due to Schmitt~\cite{Schmitt:1994}.
  For a graph $G$ with vertex set $V$
  (admitting multiple edges and loops), and a subset $U \subset V$, define $G|U$
  to be the graph whose vertex set is $U$, and whose graph structure is induced
  by restriction (that is, the edges of $G|U$ are those edges of $G$ both of
  whose incident vertices belong to $U$).  On the vector space
  with basis the set of
  \isoclasses of graphs, define a comultiplication by the rule
  $$
  \Delta(G) = \sum_{A+B=V} G|A \tensor G|B .
  $$

  This coalgebra is obtained from the decomposition space in
  Example~\ref{ex:graphs-decomp}.  Indeed, we have to take $X_1$ the groupoid of
  graphs, because the coalgebra has linear basis the set of
  \isoclasses of graphs. Since the
  comultiplication sums over all ways to partition the vertex set into two parts
  (possibly empty), we must take $X_2$ to be the groupoid of graphs with a
  two-part partition of the vertex set.  (More generally, $X_k$ is the
  groupoid of graphs with an ordered partition of the vertex set into $k$
  parts (possibly empty).)

  Taking pullback along $d_1:X_2 \to X_1$ is to consider all possible
  two-part partitions of a given graph, and taking lowershriek along
  $(d_2,d_0): X_2 \to X_1 \times X_1$ is to return the graphs induced
  by the two parts.  In conclusion, this is precisely Schmitt's 
  comultiplication.
\end{blanko}

\begin{blanko}{Comultiplication of basis elements.}\label{sec:comult-f} 
  We proceed to spell out the effect of the comultiplication on basis elements.
  The slice $\Grpd_{/X_1}$ has a canonical basis $\{ \name f:1\to X_1
  \}_{f\in\pi_0X_1}$.  Here $\name f : 1 \to X_1$ denotes the map that singles
  out the element $f\in X_1$, in category theory called the {\em name} of $f$.
  The notion of basis for slices means that every object $T \to X_1$ can be
  written uniquely as a homotopy sum of names (cf.~Lemma~\ref{lem:fashosum}).
  Giving $\name f$ as input to the comultiplication, and expanding the result
  into a homotopy sum of names, we get:
  \begin{align}
  \nonumber\Delta(\name f)&:=\bigl((X_2)_f\xrightarrow{
  d_1^*\name f
  } X_2\xrightarrow{(d_2,d_0)} X_1\times X_1\bigr)
  \\\label{Delta(f)}&=
  \int^{\sigma\in (X_2)_f}\name{d_2\sigma}\otimes\name{d_0\sigma}
  \\\nonumber&=
  \int^{(a,b)\in X_1\times X_1}(X_2)_{f,a,b}\;\name{a}\otimes\name{b}\quad
  \in\; \Grpd_{/X_1}\otimes \Grpd_{/X_1} 
  .
  \end{align}
  Here $(X_2)_f$ is the fibre of $d_1:X_2\to X_1$ over $f$, and similarly
  $(X_2)_{f,a,b}$ is the fibre of $(d_1,d_2,d_0):X_2\to X_1\times X_1 \times
  X_1$ over $(f,a,b)$.  Here and throughout, `fibre' means `homotopy fibre', 
  cf.~\ref{fibres}.

  If $X$ is the strict nerve of a category then $X_2$ is
  the set of all composable pairs of arrows and $(X_2)_f$ is the subset of those
  pairs with composite $f$.  
  In particular, $(X_2)_{f,a,b}$ is then either empty or a singleton, and
  the comultiplication reduces to the
  formula~\eqref{eq:Delta(f)} from the introduction,
  $$
  \Delta(f)\:=\sum_{ab=f}a\otimes b.
  $$

  If $X$ is the fat nerve of a category, or more generally if $X$ is a
  Segal space (that is, $X_2 \simeq X_1 \times_{X_0} X_1$), then as in the case
  of the ordinary nerve we see that if $f$ is not the composite up to
  isomorphism of $a$ and $b$ then $X_{f,a,b}$ will be empty. In the case of a
  fat nerve, it is non-empty if and only if one can write
  \[
    f\;\;=\;(
    \begin{tikzcd}x\ar[r,"\cong"]&x'\ar[r,"a"]
      &y\ar[r,"\varphi","\cong"']&y'\ar[r,"b"]&z'\ar[r,"\cong"]&z
    \end{tikzcd}
    )\;\in\;\;X_{1}.
  \]
  In the more general case of a Segal space,
  it is non-empty if and only if there exists $\sigma\in X_{2}$ and
  $\varphi\in\Map_{X_0}(d_{0}a,d_{1}b)$ such that $f\cong d_{1}\sigma$ and
  $\sigma\in X_{2}$ corresponds to
  $(a,b,\varphi)\in X_1 \times_{X_0}X_{1}$ in the notation of \ref{pbk}.

    \begin{lemma}\label{X2ab}
If $X$ is a Segal space and $a,b\in X_{1}$ then the groupoid
 $(X_2)_{a,b}$ is discrete, naturally equivalent to the set of isomorphisms
$\Map_{X_0}(d_{0}a,d_{1}b)$.
\end{lemma}

\begin{proof}
  Since $X_2\simeq X_1 \times_{X_0} X_1$ we can compute $(X_2)_{a,b}$ as the pullback
    \[\begin{tikzcd}
\SEpbk (X_1 \times_{X_0} X_1)_{a,b}\ar[r]\ar[d]&\SEpbk X_1 \times_{X_0} X_1\ar[r]\ar[d]&X_0\ar[d,"\Delta"]\\
1\ar[r,"\name{a,b}"]&X_1\times X_1\ar[r,"d_0\times d_1"]             &X_0\times X_0.
\end{tikzcd}\]
But the homotopy fibres of the diagonal $\Delta\colon X_0\to X_0\times X_0$ are (naturally equivalent to) the mapping spaces.
\end{proof}

\begin{prop}\label{prop:Delta(f)}
  If $X$ is a Segal space and $f\in X_{1}$ then
  \[
  \Delta(\name f)=
  \int^{(a,b)\in X_1\times X_1}\!\!\!\!\!\!\!\!\!\!\!\!\!
  \sum_{\varphi\in\Map_{{X_{0}}}(d_{0}a,d_{1}b)}
  \!\!\!\!\!\!\!\!\!\!\!\!\!
  \Map_{X_{1}}(f,\ell_{a,b}(\varphi))
  \;\;
  \name{a}\otimes\name{b}
  \;\;\;
  \in
  \;
  \Grpd_{/X_1}\otimes \Grpd_{/X_1} .
  \]
  where $\ell_{a,b}: \Map_{X_0}(d_{0}a,d_{1}b)
  \simeq(X_2)_{a,b}\stackrel{d_1}\longrightarrow X_{1}$.
\end{prop}
Observe that each set $\Map_{X_1}(f,\ell_{a,b}(\varphi))$ in the sum is either
empty (if $f\not\simeq\ell_{a,b}(\varphi)$) or non-canonically in bijection with
the set $\AutOmega f{X_1} =\Map_{X_1}(f,f)$.

\begin{proof}
  By the previous lemma, $(X_2)_{f,a,b}$ is the homotopy fibre of
    $\ell_{a,b}$ over $f$, and as the domain of $\ell_{a,b}$ is discrete this
    fibre is the sum of the fibres $(X_2)_{f,a,b,\varphi}$,
  \[
    \begin{tikzcd}
      (X_2)_{f,a,b,\varphi}\ar[r,""]\ar[d,""]&(X_2)_{f,a,b}\ar[r,""]\ar[d,""]
                                                         &1\ar[d,"\name f"]\\
      1\ar[r,"\name\varphi"]&\Map_{X_0}(d_{0}a,d_{1}b)\ar[r,"\ell_{a,b}"]&X_1,
    \end{tikzcd}
  \]
  in which each
    $(X_2)_{f,a,b,\varphi}\simeq \Map_{X_{1}}(f,\ell_{a,b}(\varphi))$.
\end{proof}

\end{blanko}

\begin{blanko}{Local finiteness.}\label{finitary}
  As long as we work at the objective level, where all results and proofs are
  naturally bijective, it is not necessary to impose any finiteness conditions.
  But in order to be able to take cardinality to recover numerical results
  (i.e.~at the vector-space level), suitable finiteness conditions must be
  imposed. Intuitively, mimicking the local finiteness for categories, we should
  require that for each $n\in\N$ the active map $X_n \to X_1$ is
    finite. In the category case this
  means that, for each arrow $f\in X_1$ and $n\in\N$, there are only finitely
  many decompositions of $f$ into a sequence of $n$ arrows. Technically, the
  appropriate definition is the following (from~\cite{GKT:DSIAMI-2}).

  A decomposition space $X:\simplexcategory\op \to \Grpd$ is termed
  \emph{locally finite} if $X_1$ is locally finite (in the sense of
  groupoids~\ref{finite}) and both $s_0:X_0 \to X_1$ and $d_1:X_2 \to X_1$ are
  finite maps. Then the comultiplication and counit defined above are finite
  linear functors, and hence (by Proposition~\ref{finitetypespan}) descend to
  slices of finite groupoids
  $$
  \Delta:\grpd_{/X_1}\to\grpd_{/X_1}\tensor \grpd_{/X_1}\;, \qquad
  \epsilon:\grpd_{/X_1}\to\grpd.
  $$
  We can then take cardinality to obtain comultiplication and counit maps of
  vector spaces
  $$
  \normnorm{\Delta}:\Q_{\pi_0X_1}\to\Q_{\pi_0X_1}\tensor \Q_{\pi_0X_1}\;, \qquad
  \normnorm{\epsilon}:\Q_{\pi_0X_1}\to\Q\,.
  $$
  These are coassociative and counital, and
  $$
\Inc_{X} := (\Q_{\pi_0 X_1},\normnorm{\Delta},\normnorm{\epsilon})
  $$
  is what we call the {\em numerical incidence coalgebra} of $X$.
\end{blanko}

\begin{BM}
  If $X$ is the nerve of a poset $P$, then it is locally finite in the above
  sense if and only if all intervals $[x,y]$ are finite, which is the usual
  definition for posets~\cite{Stanley}. The points in this interval parametrise
  precisely the two-stage factorisations of the unique arrow $x\to y$, so this
  condition amounts to $X_2 \to X_1$ having finite fibre over $x \to y$. (In the
  poset case, the conditions on $X_1$ and on $s_0 : X_0 \to X_1$ are
  automatically satisfied, since everything is discrete.)

  Examples of infinite categories which are locally finite are given by free
  monoids or the free category on a directed graph.
\end{BM}

\subsection{Convolution algebras}

\label{sec:conv}

\begin{blanko}{Linear dual.}
  If $X$ is a decomposition space, we have seen there is a natural coassociative
  comultiplication on $\Grpd_{/X_1}$, the incidence coalgebra of $X$, which we
  see as an `objectification' of the vector space $\Q_{\pi_0 X_1}$ underlying
  the classical incidence coalgebra.  One may also consider the incidence (or
  convolution) {\em algebra} $\Grpd^{X_1}$, which can be obtained from the
  incidence coalgebra by taking the linear dual (\ref{duality}).  Since
  $\Grpd_{/X_1}$ is the free homotopy-sum completion of $X_1$ (just as
  $\Q_{\pi_0 X_1}$ is the `linear-combination completion' of the set $\pi_0
  X_1$), objects in $\Grpd^{X_1}$ can be regarded either as presheaves
  $X_1\to\Grpd$ or as linear functors $\Grpd_{/X_1}\to\Grpd$
  (see~\ref{duality}).  The category $\Grpd^{X_1}$ is interpreted as an
  `objectification' of the incidence algebra, denoted $\Inc^X$, which has
  underlying profinite-dimensional vector space $\Q^{\pi_0 X_1}$.
\end{blanko}

\begin{blanko}{Convolution.}\label{convolution}
  The multiplication in the incidence algebra is the convolution product, given
  as the dual of the comultiplication.  Consider two linear functors
  $$
  F,G:\Grpd_{/ X_1}\longrightarrow \Grpd
  $$
  given by spans $ X_1 \leftarrow M \to 1$ and $ X_1 \leftarrow N\to 1$.
  Their tensor product $F\otimes G$ is
  then given by the span
  $$
  X_1 \times  X_1\leftarrow M \times N\to 1
  $$
  and their {\em convolution} $F*G$ is the composite of $F\otimes G$
  with the comultiplication:
  $$
  F*G:\Grpd_{/ X_1}\stackrel{\Delta}\longrightarrow
  \Grpd_{/ X_1}\otimes \Grpd_{/ X_1} \stackrel{F\tensor G}\longrightarrow \Grpd
  \tensor \Grpd \simeq \Grpd.
  $$
  This is given by the composite span
  $$
  \xymatrix@!C=9ex{
   X_1 && \\
   X_2 \ar[u] \ar[d]&\ar[l]\ar[d] M *N\ar[lu]\ar[rd]\dlpullback &
  \\
   X_1\times  X_1 &\ar[l]M\times N \ar[r] & 1.
  }$$
  The neutral functor for the convolution product is $\epsilon$.
\end{blanko}

\begin{blanko}{The zeta functor.}\label{zeta}
  The {\em zeta functor}
  $$
  \zeta:\Grpd_{/ X_1} \to \Grpd
  $$
  is the linear functor defined by the span
  $$
  X_1 \stackrel =\leftarrow  X_1 \to 1\,.
  $$
  As an element of $\Grpd^{X_1}$, this is the terminal presheaf.

  Assuming $X_1$ locally finite then $\zeta$ is a finite linear functor and
  descends to
  $$
  \zeta:\grpd_{/X_1}\to \grpd.
  $$
  Its cardinality $\Q_{\pi_0 X_1}\to\Q$, which can be regarded as
  an element in the profinite-dimensional vector space $\Q^{\pi_0 X_1}$,
  is then the usual
  zeta function $\pi_0 X_1\to\Q$ with value $1$ on each 1-simplex of $X$.
\end{blanko}

\subsection{Section coefficients}

\label{sec:tioncoeff}

\begin{blanko}{Section coefficients.}
  If $X$ is a locally finite decomposition space then the
  homotopy cardinality of the comultiplication at the objective level
  \begin{eqnarray*}
    \grpd_{/X_1} & \longrightarrow & \grpd_{/X_1\times X_1}  \\
    \name f & \longmapsto & \big( (X_2)_f \to X_2 \to X_1 \times X_1 \big)
  \end{eqnarray*}
  yields a comultiplication in the category of vector spaces
  \begin{eqnarray*}
  \Q_{\pi_0 X_1} & \longrightarrow & \Q_{\pi_0 X_1} \tensor \Q_{\pi_0 X_1} \\
  \delta_f & \longmapsto & \int^{(a,b)\in X_1\times X_1}
  \!\!\!\!\!\!\!\!\!\!\!\!\!\!\!
   \norm{(X_2)_{f,a,b}} \ \delta_a \tensor \delta_b\;\;=\;\;
  \sum_{a,b}   c^f_{a,b} \;\delta_a \tensor \delta_b,
  \end{eqnarray*}
  which defines the (numerical) incidence coalgebra $\Inc_X$.
  It is just the cardinality of \eqref{Delta(f)},
  with the {\em section coefficients}
  \begin{equation}\label{sec coeffs}
  c^f_{a,b} :=
  \frac{\norm{(X_2)_{f,a,b}}}{\norm{\AutOmega a{X_1}}\norm{\AutOmega b{X_1}}} .
\end{equation}

  In the special case of a Segal space,
  we can take cardinality of Proposition \ref{prop:Delta(f)} to arrive at
  the following
  explicit formula for the section coefficients.
\end{blanko}

\begin{prop}[See \cite{GKT:corrigendum}]\label{sectioncoeff!}
  If $X$ is a locally finite Segal space then
  \[
  c^f_{a,b} =
  \frac{\norm{\AutOmega f{X_1}}\norm{\{\varphi \suchthat \ell_{a,b}(\varphi)\simeq f\}}}
       {\norm{\AutOmega a{X_1}}\norm{\AutOmega b{X_1}}}.
  \]
  In the case that $X$ is the fat nerve of a category,
  $\ell_{a,b}(\varphi)=a\varphi b$ and the term
  $\norm{\{\varphi \suchthat \ell_{a,b}(\varphi)\simeq f\}}$ is just the number of
  isomorphisms $\varphi : d_0 a \simeq d_1 b$ such that $a\varphi b\simeq f$.
\end{prop}

\begin{cor}\label{sec coeffs for monoids}
    If $X$ is a locally finite Segal space with $X_0=1$, then
\begin{equation}
  c^f_{a,b} =\begin{cases}\frac{\norm{\AutOmega f{X_1})}}
    {\norm{\AutOmega a{X_1}}\norm{\AutOmega b{X_1}}}& \text{if } f\simeq ab\\
    \qquad0& \text{if } f\not\simeq ab.
\end{cases}
\end{equation}
Here $ab$ denotes the image of $(a,b)$ under
$X_{1}\times X_{1}\simeq X_{2}\stackrel{d_{1}}\longrightarrow X_{1}$.
\end{cor}
\begin{proof}
  Since $X_0$ is contractible, the set $\{\varphi\}$ of the proposition is
  either singleton or empty, depending on whether $ab\simeq f$ or not.
\end{proof}

\begin{blanko}{`Zeroth section coefficients': the counit.}
  Let us also say a word about the zeroth section coefficients, i.e.~the
  computation of the counit. If $f$ is not isomorphic to a degenerate simplex
  then clearly $\norm\epsilon(\delta_f)=0$. In the case $f$ is degenerate, we
  just remark on two special cases:
  \begin{itemize}
    \item if $X$ is complete (\ref{complete}), meaning that $s_0$ is a
          monomorphism (\ref{mono}), then $\norm\epsilon(\delta_f) = 1$,
  \item  if $X_0=1$ then
          $\norm\epsilon(\delta_f)
          =\norm{\AutOmega f{X_1}}$.
  \end{itemize}
\end{blanko}

\begin{blanko}{Numerical convolution product.}
  By duality, if $X$ is locally finite, the convolution product descends to the
  profinite-dimensional vector space $\Q^{\pi_0 X_1}$ obtained by taking
  cardinality of $\grpd^{X_1}$, defining the (numerical) incidence algebra of
  $X$, denoted $\Inc^X$. It follows from the general theory of homotopy linear
  algebra (see appendix \ref{dual-card} and \cite{GKT:HLA}) that the cardinality
  of the convolution product is the linear dual of the cardinality of the
  comultiplication. Since it is the same span that defines the comultiplication
  and the convolution product, it is also the exact same matrix that defines the
  cardinalities of these two maps. It follows that the structure constants for
  the convolution product (with respect to the pro-basis $\{\delta^f\}$) are the
  same as the structure constants for the comultiplication (with respect to the
  basis $\{\delta_f\}$), i.e.~the section coefficients.
\end{blanko}

\begin{blanko}{Example.}
  The strict nerve of a category $\CC$ is a decomposition space which is
  discrete in each degree.  The resulting coalgebra at the numerical level
  (assuming local finiteness) is the coalgebra of
  Content--Lemay--Leroux~\cite{Content-Lemay-Leroux}, and if the category is
  just a poset, that of Joni and Rota~\cite{JoniRotaMR544721}.

  The objective-level incidence algebra of the strict nerve of $\CC$ has the
  convolution product
  \begin{equation}  \label{strict h*h}
  h^a * h^b = \begin{cases}
    h^{ab} & \text{ if $a$ and $b$ are composable} \\
    0 & \text{ else.}
  \end{cases}
  \end{equation}
    For the fat nerve $X$ of $\CC$, we find instead 
  \begin{equation}  \label{fat h*h}
    h^a * h^b \; = \; \sum_{\varphi\colon d_0a\simeq d_1b} \; h^{a\varphi b}
\; = \; \sum_{f\in\pi_0X_1} \;\{\varphi \suchthat a\varphi b\simeq f\}\cdot h^f
  \end{equation}
  where the first sum is over all isomorphisms
  $\varphi$ from the target of $a$ to the source of $b$,
  cf.~\cite{GKT:corrigendum}.

  To compute the cardinality of this algebra, note first that the cardinality
  of the representable $h^a$ is generally different from the canonical basis
  element $\delta^a$: the formula \eqref{representable-card}
    says
  $$
  \norm{h^a} = \norm{\AutOmega a {X_1}} \, \delta^a ,
  $$
  so that \eqref{fat h*h} becomes
$$
\delta^a * \delta^b
\;=\;
\!\!\!\!
\sum_{f\in\pi_0X_1}
\!\!\!\!
{\textstyle\frac{
\norm{\{\varphi \suchthat a\varphi b\simeq f\}} \cdot \norm{\AutOmega{f}{X_1}}
}
{
\norm{\AutOmega {a}{X_1}}\cdot\norm{\AutOmega {b}{X_1}}}
}\;
\delta^{f}
 \;=\;
 \sum_{f\in\pi_0X_1} \; c_{a,b}^f\;\delta^f
$$
 with the section coefficients as given in Proposition \ref{sectioncoeff!}.
\end{blanko}

\begin{blanko}{Finite support.}\label{fin-supp}
  The numerical incidence algebra $\Inc^X$ lives in profinite-dimensional vector
  spaces, since functions are not required to have finite support---for example,
  the zeta function does not have finite support for infinite posets or
  categories.  It is also interesting to consider the subalgebra of $\Inc^X$
  consisting of functions with finite support.  At the objective level this is
  the full subcategory $\grpd^{X_1}_{\finsup} \subset \grpd^{X_1}$, and
  numerically it is $\Q^{\pi_0 X_1}_{\finsup} \subset \Q^{\pi_0 X_1}$.  Of
  course we have canonical identifications $\grpd^{X_1}_{\finsup} \simeq
  \grpd_{/X_1}$, as well as $\Q^{\pi_0 X_1}_{\finsup} \simeq \Q_{\pi_0 X_1}$,
  but it is important to keep track of which side of duality we are on.
  
  That the decomposition space is locally finite is not the appropriate
  condition for the convolution and unit to restrict to the functions with
  finite support.  Instead the requirement is that $X_1$ be locally finite and
  the maps
  $$
  X_2 \to X_1 \times X_1, \qquad X_0 \to 1
  $$ 
  be finite.  By Lemma~\ref{X2ab} we know that the former map is finite for any
  Segal space with $X_0$ {\em locally} finite, but for the latter $X_0$ must actually
  be finite.
\end{blanko}

\begin{blanko}{Examples: category algebras.}\label{categoryalg}
  If $X$ is the {\em strict} nerve of a category $\CC$, then the finite-support
  convolution algebra is precisely the {\em category algebra} of $\CC$.  This is
  an important notion in representation theory (see~\cite{Webb}).
  
  Note that since the strict nerve is a Segal space, the formula for the section
  coefficients are the same as computed above, giving the familiar formula
  \eqref{strict h*h}.  Similarly the formula for the convolution unit is
  $$
  \epsilon = \sum_x \delta^{\id_x}
  $$
  the sum of all indicator functions of identity arrows: for this to be finite
  we need to require that the category has only finitely many objects.
  
  In the case of the {\em fat} nerve of a category $\CC$, the finiteness condition
  for having a finite-support convolution is
  implied by the condition that every object in $\CC$ has a finite automorphism
  group (a condition implied by local finiteness).
  On the other hand,
  the convolution unit has finite support precisely when there is only a finite
  number of \isoclasses of objects, already a more drastic condition.
  Compared to the usual category 
  algebra, this `fat category algebra' has (cf.~\eqref{fat h*h}):
  \[
  h^a * h^b \; = \; \sum_{f\in\pi_0X_1} \;\{\varphi \suchthat a\varphi b\simeq 
  f\}\cdot h^f .
\]

  Note that an important source of examples of category algebras are given by
  the path algebra of a quiver $Q$ (see for example \cite{Brion}): that is
  simply the category algebra on the free category on $Q$.  Since there are no
  automorphisms in a free category, in this case there is no difference between
  strict and fat nerve.
  
  It should be noted that the finite-support incidence algebras are important
  also outside the setting of category algebras, namely in the case
  of the Waldhausen $S_\bullet$-construction (cf.~\ref{sec:Wald} below):
  they are the Hall algebras (see 
  \cite{GKT:DSIAMI-1}).  The finiteness conditions are then homological, namely
  finite $\operatorname{Ext}^0$ and $\operatorname{Ext}^1$.
\end{blanko}

\begin{blanko}{Locally discrete decomposition spaces.}\label{loc-disc-ds}
  In the formula in Proposition~\ref{sectioncoeff!}
  for the section coefficients there are denominators.
  In very many examples of importance, however, the
  section coefficients are actually integral.
  This happens when the map
  $d_1 : X_2 \to X_1$
  is discrete, that is,
    has discrete homotopy fibres. Equivalently,
    the induced group homomorphisms
    $(d_1)_*:\AutOmega{\sigma}{X_{2}}\to\AutOmega{d_{1}\sigma}{X_{1}}
    $
    are injective for each $\sigma\in X_{2}$.
    For the zeroth section coefficients one
    should also require $s_0 : X_0 \to X_1$ to be discrete,
    but this always holds as $(d_1)_*(s_0)_*:
\AutOmega x {X_0}\to\AutOmega{s_0x}{X_1}\to\AutOmega x{X_0}$ is the identity.

  We define $X$ to be {\em locally discrete} when $d_1 : X_2 \to X_1$
  is a discrete map.
\end{blanko}

\begin{BM}
  In our terminology (\ref{finitary}), `locally finite' means that
  $d_1 : X_2 \to X_1$ and $s_0 : X_0 \to X_1$ are finite maps {\em and} that
  $X_1$ is locally finite. To be consistent with this definition, `locally
  discrete' should mean $d_1 : X_2 \to X_1$ and $s_0 : X_0 \to X_1$ discrete,
  {\em and} $X_1$ a locally discrete groupoid. If we define a groupoid to be
  locally discrete if all its hom sets are discrete, then every groupoid is
  locally discrete, and therefore it is not necessary to mention it in the
  definition.
\end{BM}

\begin{blanko}{Examples.}
    The fat nerve of a category $\CC$ is locally discrete if and only if,
    in any commutative diagram in $\CC$ of the form
  \[
  \xymatrix@R1ex{
&y\ar[dd]^{\simeq}_{\varphi}\ar[rd]^{b}\\x\ar[ru]^{a}\ar[rd]_{a} && z \,,\\&y\ar[ru]_{b}}
\]
the isomorphism $\varphi$ is the identity on $y$. For example, this is the case if
$\CC$ satisfies any of
  the three conditions
  \begin{itemize}
    \item All the arrows in $\CC$ are monos

    \item All the arrows in $\CC$ are epis

    \item All the automorphisms in $\CC$ are identities.
  \end{itemize}
  Starting from these three cases, many more examples can be derived by virtue
  of the following result.
\end{blanko}

\begin{lemma}\label{lem:locdiscrDec}
  The following are equivalent for a decomposition space $X$
  \begin{enumerate}
    \item $X$ is locally discrete.
    \item $\Decbot{X}$ is locally discrete.
  
    \item $\Dectop{X}$ is locally discrete.
  \end{enumerate} 
\end{lemma}
This result refers to {\em decalage} (\ref{Dec}), recalled in the next subsection
where we also prove the lemma.

As we shall see, examples coming from combinatorics
tend to be locally discrete.

\begin{blanko}{A tiny example: the `hanger category'.}
  The following category is perhaps the smallest example of a category
  whose fat nerve is {\em not} locally discrete.
  $$
  
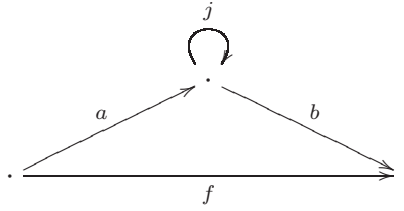
\xymatrix @C=64pt {
  &\cdot \ar@(lu,ru)[]^j \ar[rd]^b& \\
  \cdot\ar[ru]^a  \ar[rr]_f && \cdot
  }
  $$
   in which 
   $$
   ab=f \qquad jj=1 \qquad aj=a \qquad jb=b .
   $$
  It has
  $$
  \Delta(f) = 1\otimes f + f\otimes 1 + \frac{a\otimes b}{2}
  $$
  since the factorisation $ab$ admits an involution, given by $j$.
\end{blanko}

\subsection{\culf functors, coalgebra homomorphisms and bialgebras}

\label{sec:bialg}

An appropriate notion of morphism between decomposition spaces is
that of \culf functors~\cite{GKT:DSIAMI-1}, which we briefly recall.
Their importance is that they induce coalgebra homomorphisms between
the incidence coalgebras.  Two main instances of \culf functors are decalage
and monoidal structures.  As we shall see, decalage accounts for many
reduction procedures in classical theory of incidence coalgebras.  A \culf
monoidal structure on a decomposition space is precisely what makes the
incidence coalgebra into a bialgebra.

\begin{blanko}{\culf functors.}\label{culf-maps}
  A simplicial map $F:X\to Y$ is
  \begin{itemize}
  \item {\em conservative} if it is cartesian with respect to codegeneracy maps
  (\ref{culf-squares}a).
  \item {\em ULF} (for Unique Lifting of Factorisations) if it is cartesian
  with respect to inner coface maps (\ref{culf-squares}b).
  
  \item {\em \culf} if it is both conservative and ULF, that is, cartesian on
  all active maps.  We shall use the term {\em \culf functor} even
  between simplicial groupoids not assumed to be Segal.
  \end{itemize}
  \begin{equation}\vcenter{
  \xymatrix@R-1ex{
  X_n\ar@{}[rd]|{\text{(a)}}\drpullback \dto_-{F}\rto^-{s_i}
  &X_{n+1}\ar[d]^-{F}
  \\
  Y_n \ar[r]_-{s_i}&Y_{n+1},}}
  \qquad\qquad
  \vcenter{\xymatrix@R-1ex{
  X_{n+1}\ar@{}[rd]|{\text{(b)}} \dto_-{F}
  &\dlpullback X_{n+2} \ar[l]_{d_{i+1}}\ar[d]^-{F}
  \\
  Y_{n+1} & \ar[l]^{d_{i+1}} Y_{n+2}}}
  \qquad(0 \leq i \leq n).\label{culf-squares}
  \end{equation}

  If both $X$ and $Y$ are decomposition spaces, then in fact ULF implies \culf
  \cite[Proposition 4.2]{GKT:DSIAMI-1}.
  
  In many examples of decomposition spaces, $1$-simplices are thought of
  as arrows: for simplicial maps between Rezk complete Segal spaces (see
  \ref{Rezk}), conservative means not inverting any arrows, and ULF means
  inducing a one-to-one correspondence between factorisations of an arrow in $X$
  and of its image in $Y$.  

  For morphisms of posets, conservative means to preserve $<$, not just $\leq$,
  while ULF is strictly stronger: it means to induce an isomorphism $[x,x']
  \simeq [Fx,Fx']$ on each interval.  If the morphism of posets is 
  a full inclusion,
  then ULF is precisely the same as convex (cf.~\cite{GKT:restriction}): if two
  elements belong to the subposet then so do all elements between them.
  Note that 
  an ULF map of posets does not have to be injective: for example,
  if $X$ is a discrete poset then {\em any} map $X\to Y$ is ULF.

  Given a simplicial \homomorphism $F:X\to Y$ between decomposition spaces, the
  span $X_1\stackrel=\longleftarrow X_1\stackrel{F_1}\longrightarrow Y_1$
  defines a linear functor
  $$
  {F_1}\lowershriek:\Grpd_{/X_1}\to\Grpd_{/Y_1},
  $$
  which descends
  to a linear functor ${F_1}\lowershriek:\grpd_{/X_1}\to\grpd_{/Y_1}$ with
  cardinality the linear map $\Q_{\pi_0 X_1}\to\Q_{\pi_0 Y_1}$ given on the
  basis by $\delta_f\mapsto\delta_{F_1f}$.
\end{blanko}

\begin{lemma}\label{culf-hm}
  \cite{GKT:DSIAMI-1}
  If $F$ is \culf, then ${F_1}\lowershriek$ is a coalgebra homomorphism,
  meaning that it preserves the comultiplication and
  counit up to coherent homotopy
  $$
  ({{F_1}\lowershriek} \otimes {F_1}\lowershriek)\,\Delta_X\;\simeq\;\Delta_Y\,{F_1}\lowershriek,
  \qquad\quad
  \epsilon_X\;\simeq\;\epsilon_Y\,{F_1}\lowershriek.
  $$
\end{lemma}

\begin{blanko}{Decalage.}\label{Dec}
  An important source of \culf functors is given by decalage. 
  Recall that the {\em decalage} functor $\Decbot{}$ 
  on simplicial groupoids forgets the bottom 
  face and degeneracy maps, and shifts the indexing of the groupoids.  The unused
  face map $d_\bot$ provides a natural transformation from the decalage back to
  the identity functor.  We refer to this $d_\bot$ as the {\em dec map}.
  \vspace*{-6mm}

$$
\xymatrix@C+2em@R-2.5mm{ X&
X_0  
\ar[r]|(0.55){s_0} 
&
\ar[l]<+2mm>^{d_0}\ar[l]<-2mm>_{d_1} 
X_1  
\ar[r]<-2mm>|(0.6){s_0}\ar[r]<+2mm>|(0.6){s_1}  
&
\ar[l]<+4mm>^(0.6){d_0}\ar[l]|(0.6){d_1}\ar[l]<-4mm>_(0.6){d_2}
X_2 
\ar[r]<-4mm>|(0.6){s_0}\ar[r]|(0.6){s_1}\ar[r]<+4mm>|(0.6){s_2}  
&
\ar[l]<+6mm>|(0.6){d_0}\ar[l]<+2mm>|(0.6){d_1}\ar[l]<-2mm>|(0.6){d_2}\ar[l]<-6mm>|(0.6){d_3}
X_3 
\ar@{}|\cdots[r]&
\\ \\
{\Decbot{X}}\ar[uu]^{d_\bot} &
X_1  {\ar[uu]^{d_0}}
\ar[r]|(0.55){s_1} 
&
\ar[l]<+2mm>^{d_1}\ar[l]<-2mm>_{d_2} 
X_2  \ar[uu]^{d_0}
\ar[r]<-2mm>|(0.6){s_1}\ar[r]<+2mm>|(0.6){s_2}  
&
\ar[l]<+4mm>^(0.6){d_1}\ar[l]|(0.6){d_2}\ar[l]<-4mm>_(0.6){d_3}
X_3 \ar[uu]^{d_0}
\ar[r]<-4mm>|(0.6){s_1}\ar[r]|(0.6){s_2}\ar[r]<+4mm>|(0.6){s_3}  
&
\ar[l]<+6mm>|(0.6){d_1}\ar[l]<+2mm>|(0.6){d_2}\ar[l]<-2mm>|(0.6){d_3}\ar[l]<-6mm>|(0.6){d_4}
X_4 \ar[uu]^{d_0}
\ar@{}|\cdots[r]
&
}
$$
\vspace*{-2mm}

Similarly, the decalage $\Dectop{}$ forgets the top face and degeneracy maps.
\end{blanko}

Decalage also plays an important role at the theoretical level, as
exemplified by the following result.

\begin{lemma}[{\cite{Dyckerhoff-Kapranov:1212.3563}, \cite{GKT:DSIAMI-1},} 
  in conjunction with \cite{Feller-Garner-Kock-Proulx-Weber:1905.09580}]
  A simplicial groupoid $X$ is a decomposition space if and only if both
  $\Dectop{X}$ and $\Decbot{X}$ are Segal spaces. Furthermore, in this case
  the corresponding dec maps $d_\top$ and $d_\bot$ are \culf.
\end{lemma}

In particular for any decomposition space $X$ we have a canonical coalgebra
homomorphism from the incidence coalgebra of $\Decbot{X}$ to that of $X$, and
similarly for $\Dectop{}$.  This appears in many examples.

\bigskip

Lemma~\ref{lem:locdiscrDec} above refers to decalage,
and we owe the proof.

\begin{proof*}{Proof of Lemma~\ref{lem:locdiscrDec}.}
  Just note that the dec map is always essentially surjective, since it admits a
  degeneracy map as a section.  Now the result follows from the following lemma.
\end{proof*}

\begin{lemma}
  A decomposition space $X$ is locally discrete if it admits an essentially
  surjective \culf functor 
  $Y \to X$ with $Y$ a locally discrete decomposition space.
\end{lemma}
\begin{proof}
  This follows since discreteness is a local property:
  in a pullback square of groupoids
  $$\xymatrix{
     E' \drpullback \ar[r]\ar[d]_{f'} & E \ar[d]^f \\
     B' \ar[r]_e & B
  }$$
  the homotopy fibres $f'^{-1}(b')$ and  $f^{-1}(e(b'))$ are equivalent. Thus if
  $e$ is essentially surjective then $f$ is discrete if and only if
  $f'$ is discrete.  
\end{proof}

\begin{blanko}{Bialgebras.}
  Recall that a bialgebra is a coalgebra with a compatible algebra structure,
  meaning that multiplication and unit are coalgebra homomorphisms.  More
  formally it can be characterised as a monoid object in the category of
  coalgebras.  In Lemma \ref{culf-hm} we saw that a sufficient condition for a
  simplicial map $f$ between decomposition spaces to induce a coalgebra
  homomorphism on incidence coalgebras is that $f$ be \culf.  Accordingly we
  define a {\em monoidal decomposition space} \cite{GKT:DSIAMI-1} to be a
  decomposition space $Z$ equipped with an associative unital monoid structure
  given by \culf functors $m:Z \times Z \to Z$ and $e:1\to Z$.
\end{blanko}

\begin{prop}\label{prop:bialg}
  If $Z$ is a monoidal decomposition space then $\Grpd_{/Z_1}$ is naturally a
  bialgebra, termed its {\em incidence bialgebra}.  Monoidal \culf functors
  induce bialgebra homomorphisms.
\end{prop}

\begin{blanko}{Extensivity.}
  Classically, a category $\CC$ with sums is called {\em extensive} when the
  natural functor $\CC_{/A} \times \CC_{/B} \to \CC_{/A+B}$ is an
  equivalence. More generally, a monoidal category $(\CC,\tensor,I)$ is called
  {\em monoidal extensive} when the natural functor
  $\CC_{/A} \times \CC_{/B} \to \CC_{/A \tensor B}$ is an equivalence.
  The fat nerve of a monoidal extensive category is always a monoidal
  decomposition space. As an example, the category $\FF$ of finite sets and all
  maps is extensive in the classical sense. The category of finite sets and
  surjections inherits the monoidal structure $+$ from $\FF$, but it is no
  longer the categorical sum (since there are no sum
  injections). It is still monoidal extensive. We shall come back to this
  particular example in Subsection~\ref{sec:fdb}.
\end{blanko}

\begin{lemma}[{\cite[Lemma~9.3]{GKT:DSIAMI-1}}]\label{dec-monoidal}
  The Dec of a monoidal decomposition space has again a natural monoidal
  structure, and the dec map preserves this structure.
\end{lemma}

\begin{blanko}{Example: the Schmitt Hopf algebra of graphs,
    continued.}\label{ex:graphs-bialg}
  The decomposition space of Example~\ref{ex:graphs-decomp} (and
  \ref{ex:graphs-coalg}) has a canonical monoidal structure given by disjoint
  union.  Recall that $X_k$ is the groupoid of graphs equipped with an ordered
  partition of the vertex set into $k$ parts (possibly empty).  The disjoint
  union of two such structures is given by taking the disjoint union of the
  underlying graphs, with new partition given by joining the two $i$th parts,
  for each $1\leq i \leq k$.  This clearly defines a simplicial map from
  $X\times X$ to $X$.  To say that it is \culf is to establish that squares
  like this is a pullback:
  $$\xymatrix{
    X_1\times X_1 \ar[d]_{+} & \ar[l]_{d_1}\dlpullback X_2\times X_2 \ar[d]^{+}\\
    X_1 & \ar[l]^{d_1} X_2 }$$
  But this is clear: a pair of graphs with a
  $2$-partition each can be uniquely reconstructed if we know what the two
  underlying graphs are (an element in $X_1\times X_1$) and we know how the
  disjoint union is partitioned (an element in $X_2$) --- provided of course
  that we can identify the disjoint union of those two underlying graphs with
  the underlying graph of the disjoint union (which is to say that the data
  agree down in $X_1$). It follows that the resulting incidence coalgebra is
  also a bialgebra. (Furthermore, this bialgebra has a canonical grading, by the
  number of vertices, and with respect to this grading it is connected, since
  the only zero-vertex graph is the empty graph. It is well known that connected
  graded bialgebras are Hopf~\cite{Figueroa-GraciaBondia:0408145}.)
\end{blanko}

\section{Examples}
\label{sec:ex}

It is characteristic for the classical theory of incidence (co)algebras of
posets that most often it is necessary to impose an equivalence relation on the
set of intervals in order to arrive at the interesting `reduced' incidence
(co)algebras. This equivalence relation may be simply isomorphism of posets, or
equality of length of maximal chains as in binomial posets
\cite{Doubilet-Rota-Stanley}, or it may be more subtle order-compatible
relations \cite{Dur:1986}, \cite{Schmitt:1994}. Content, Lemay and
Leroux~\cite{Content-Lemay-Leroux} remarked that in some important cases the
relationship between the original incidence coalgebra and the reduced one
amounts to a \culf functor, although they did not make this notion explicit.
From our global simplicial viewpoint, we observe that very often these \culf
functors arise from decalage, often of a decomposition space which not a poset
and sometimes not even a Segal space.

Recall that for $X$ a locally finite decomposition space, we write $\Inc_X$ for
the incidence coalgebra (with underlying vector space $\Q_{\pi_0 X_1}$), and we
write $\Inc^X$ for the incidence algebra (with underlying profinite-dimensional
vector space $\Q^{\pi_0 X_1})$.

\begin{blanko}{Decomposition spaces for the classical series.}
  Classically important examples of incidence algebras are power series
  representations. From the perspective of the objective method, these
  representations appear as cardinalities of various monoidal structures on
  species, realised as incidence algebras with groupoid coefficients. We list
  six examples illustrating some of the various kinds of generating functions
  listed by Stanley~\cite{Stanley:MR513004} (see also D\"ur~\cite{Dur:1986}).
\begin{enumerate}
  \item Ordinary generating functions, the zeta function being $\zeta(z)= 
  \sum_{k\geq 0} z^k$.  This comes from ordered
  sets and ordinal sum, and the incidence algebra is that of ordered species
  with the ordinary product.

  \item Exponential generating functions, the zeta function being
        $\zeta(z)=\sum_{k\geq 0} \frac{z^k}{k!}$. Objectively, there are two
        versions of this: one coming from the standard Cauchy product of
        species, and one coming from the shuffle product of $\mathbb L$-species
        (in the sense of \cite{Bergeron-Labelle-Leroux}).

  \item Ordinary Dirichlet series, the zeta function being $\zeta(z)=\sum_{k>0} k^{-s}$.
  This comes from ordered sets with the cartesian product.

  \item `Exponential' Dirichlet series, the zeta function being
        $\zeta(z)= \sum_{k>0} \frac{k^{-s}}{k!}$. This comes from the Dirichlet
        product of arithmetic species \cite{Baez-Dolan:zeta}, also called the
        arithmetic product~\cite{Maia-Mendez:0503436}.

  \item $q$-exponential generating series, with zeta function
        $\zeta(z)= \sum_{k\geq 0} \frac{z^k}{[k]!}$. This comes from the
        Waldhausen $S_\bullet$-construction on the category of finite vector
        spaces. The incidence algebra is that of $q$-species with a version of
        the external product of Joyal--Street~\cite{Joyal-Street:GLn}.

  \item A variation with zeta function
        $\zeta(z)= \sum_{k\geq 0} \frac{z^k}{\norm{\Aut(\F_q^k)}}$, which arises
        from $q$-species with the `Cauchy' product studied by
        Morrison~\cite{Morrison:0512052}.
\end{enumerate}
Of these examples, only (1) and (3) have trivial section coefficients and come
from a \M category in the sense of Leroux. We proceed to the details.

\end{blanko}

\subsection{Additive examples}

\label{sec:add}

We start with several easy examples that serve to reiterate the importance of
having incidence algebras of posets, monoids and monoidal groupoids on the same
footing, connected by \culf functors, and in particular by decalage.

\begin{blanko}{Linear orders and the additive monoid.}\label{ex:N&L}
  Let $\mathbf L$ denote the nerve of the poset $(\N,\leq)$, and let $\mathbf N$
  be the nerve of the additive monoid $(\N,+)$.  Imposing the equivalence
  relation `isomorphism of intervals' on the incidence coalgebra of $\mathbf L$
  gives that of $\mathbf N$, and
  Content--Lemay--Leroux~\cite{Content-Lemay-Leroux} observed that this
  reduction is induced by a \culf functor $r:\mathbf L \to \mathbf N$ sending
  $a\leq b$ to $b-a$.  In fact we have:
\begin{lemma}
There is an isomorphism of simplicial sets
  \begin{align*}
    \Decbot{\mathbf N}  & \stackrel\simeq\longrightarrow  \mathbf L  \\
    \intertext{given in degree $k$ by}
    (x_0,\dots,x_k) & \longmapsto
      [x_0\leq x_0+x_1 \leq  \dots \leq x_0+\dots+x_k] ,
  \end{align*}
  and the \culf functor $r$ is isomorphic to the dec map
  $$
  d_\bot: \Decbot{\mathbf N} \to \mathbf N,\qquad (x_0,\dots,x_k)
  \mapsto (x_1,\dots,x_k).
  $$
\end{lemma}

  The comultiplication on $\Grpd_{/\mathbf N_1}$ is given by
  $$
  \Delta(\name n) = \sum_{a+b=n} \name a \tensor \name b
  $$
  and, taking cardinality, the incidence coalgebra
  $\Inc_{\mathbf N}$ is the vector space $\Q_{\N}$ with basis  
  given by the
  symbols
  $\delta_n$ and comultiplication $\Delta(\delta_n) = \sum\limits_{a+b=n}
  \delta_a \tensor \delta_b$.
  The incidence algebra $\Inc^{\mathbf{N}}$ is the
  profinite-dimensional vector space $\Q^{\N}$ on
  the symbols
  $\delta^n$ with convolution product
  $\delta^a * \delta^b = \delta^{a+b}$, and
  is isomorphic to the ring of power series in one variable,
  \begin{eqnarray*}
    \Inc^{\mathbf{N}} & \stackrel\simeq\longrightarrow & \Q[[z]]  \\
    \delta^n & \longmapsto & z^n \\
    (\N\stackrel f\to\Q) &\longmapsto& \sum f(n)\, z^n  .
  \end{eqnarray*}
\end{blanko}

\begin{blanko}{Upper dec.}\label{upperdec-op}
  In the previous example, and in most of the following, it is more convenient
  to work with lower dec.  Let us just point out what happens with upper dec.
  Let $\mathbf L\op$ denote the nerve of the opposite poset of $(\N,\leq)$,
  that is, $(\N,\geq)$.
  There is a \culf functor $r': \mathbf L \op \to \mathbf N$ sending $a\geq b$
  to $a-b$.  We have:
\begin{lemma}
There is an isomorphism of simplicial sets
  \begin{align*}
   \Dectop{\mathbf N}  & \stackrel\simeq\longrightarrow  \mathbf L \op  \\
   \intertext{given in degree $k$ by}
    (x_0,\dots,x_k) & \longmapsto  [x_0+\dots+x_k \geq x_1 + \dots + x_k \geq
     \dots \geq x_{k-1}+x_k \geq x_k] ,
  \end{align*}
  and the \culf functor $r'$ is isomorphic to the dec map
  $$
  d_\top: \Dectop{\mathbf N} \to \mathbf N,\qquad (x_0,\dots,x_k)
  \mapsto (x_0,\dots,x_{k-1}).
  $$
\end{lemma}
In the following examples, this contravariance comes in for all upper decs.
It will not play any role until Example~\ref{ex:CK}.
\end{blanko}

\begin{blanko}{Powers.}\label{ex:Nk&Lk}
  As a variation of the previous example, fix $k\in \N$ and let $\mathbf{L}^k$
  denote the (strict) nerve of the poset $(\N^k, \leq)$ and let $\mathbf{N}^k$
  denote the strict nerve of the monoid $(\N^k,+)$. Again there is a \culf
  functor $\mathbf{L}^k \to \mathbf{N}^k$, and the incidence algebra of
  $\mathbf{N}^k$ is the power series ring in $k$ variables. The functor is
  defined by coordinatewise difference, and again it is given by decalage, via a
  natural identification $\mathbf{L}^k\simeq\Decbot{\mathbf{N}^k}$. The functor
  does {\em not} divide out by isomorphism of intervals, unless $k=1$, since
  isomorphic intervals also arise by permutation of coordinates, treated next.
\end{blanko}

\begin{blanko}{Symmetric powers.}\label{sym}
  Let $M$ be a monoid.  For fixed $k\in \N$, the power $M^k$ is again a monoid,
  considered as a decomposition space via its strict nerve $X$.  The symmetric
  group $\mathfrak S_k$ acts on $X_1 = M^k$ by permutation of coordinates, and
  acts on $X_n = X_1^n = (M^k)^n$ diagonally.  There is induced a simplicial
  groupoid $X/\mathfrak S_k$ given by homotopy quotient: in degree $n$ it is the
  action groupoid $\frac{X_1\times\cdots \times X_1}{\mathfrak S_k}$.  Since
  taking homotopy quotient of a group action is a lowershriek operation, it
  preserves pullbacks, so it follows that this new simplicial groupoid again
  satisfies the Segal condition.  (It is no longer a monoid, though, since in
  degree zero we have the one-object groupoid $B\mathfrak S_k=1/\mathfrak S_k$,
  the classifying space of the group $\mathfrak S_k$).  In general, the quotient
  map $X \to X/\mathfrak S_k$ is a \culf functor which does not arise from
  decalage.

  We now return to the poset $(\N^k, \leq)$ and its nerve
  $\mathbf{L}^k$ from \ref{ex:Nk&Lk}. The reduced incidence algebra, given by
  identifying isomorphic intervals, coincides with the incidence coalgebra of
  $\mathbf{N}^k/\mathfrak S_k = (\N^k,+)/\mathfrak S_k$. The reduction map is
  the composite \culf functor
  $$
  \mathbf{L}^k \simeq \Decbot{\mathbf{N}^k} \longrightarrow \mathbf{N}^k
  \longrightarrow \mathbf{N}^k/\mathfrak S_k .
  $$
\end{blanko}

\begin{blanko}{Injections and the monoidal groupoid of sets
  under sum.}\label{ex:I=DecB}
Let $\mathbf{I}$ be the fat nerve of the category of finite sets and injections,
and let $\mathbf{B}$ be the monoidal nerve of the monoidal groupoid $(\B, +, 0)$
of finite sets and bijections (see \ref{monoidalgroupoids}).
D\"ur~\cite{Dur:1986} noted that imposing the equivalence relation `having
isomorphic complements' on the incidence coalgebra of $\mathbf{I}$ gives the
binomial coalgebra. Again, we can see this reduction map as induced by a \culf
functor from a decalage:
\begin{lemma}\label{lem:I=DecB}
  There is an equivalence of simplicial groupoids
  \begin{align*}
   \Decbot{\mathbf B}  & \stackrel\simeq\longrightarrow  \mathbf{I}  \\
   \intertext{given in degree $k$ by}
    (x_0,\dots,x_k) & \longmapsto
              [x_0\subseteq x_0+x_1 \subseteq  \dots \subseteq x_0+\dots+x_k] ,
  \end{align*}
  and a \culf functor $\mathbf{I}\to \mathbf B$ is given by
  $$
  d_\bot: \Decbot{\mathbf B} \to \mathbf B,\qquad (x_0,\dots,x_k)
  \mapsto (x_1,\dots,x_k).
  $$
\end{lemma}

The isomorphism may also be represented diagrammatically using diagrams
reminiscent of those in Waldhausen's $S_\bullet$-construction (cf.
Subsection~\ref{sec:Wald} below). As an example, both groupoids $\mathbf{I}_3 $
and $(\Decbot{\mathbf B})_3=\mathbf B_4$ are equivalent to the groupoid of
diagrams
  $$\xymatrix@R-0.8pc{
    & & &x_3 \ar[d]
    \\
    & &x_2 \ar[d] \ar[r] & x_2+x_3 \ar[d]
    \\
    &x_1\ar[r]\ar[d] & x_1+x_2 \ar[d] \ar[r] & x_1+x_2+x_3 \ar[d]
    \\
    x_0 \ar[r] & x_0+x_1 \ar[r] & x_0+x_1+x_2\ar[r] & x_0+x_1+x_2+x_3 }$$
  The face maps $d_i:\mathbf{I}_3\to \mathbf{I}_2$ and
  $d_{i+1}:\mathbf{B}_4\to \mathbf{B}_3$ both act by deleting the column
  beginning $x_i$ and the row beginning $x_{i+1}$. In particular
  $d_\bot:\mathbf{I}\to\mathbf B$ deletes the bottom row, sending a sequence of
  injections to the sequence of successive complements $(x_1,x_2,x_3)$. We will
  revisit this theme in the treatment of the Waldhausen
  $S_\bullet$-construction.

\medskip

From Subsection~\ref{sec:bialg}
we have:
\begin{lemma}\label{lem:IB}
  Both $\mathbf{I}$ and $\mathbf{B}$ are monoidal decomposition spaces under
  disjoint union, and $\mathbf{I}\simeq \Decbot{\mathbf B} \to \mathbf{B}$ is
  a monoidal \culf functor inducing a bialgebra homomorphism
  $\Grpd_{/{\mathbf{I}}_1}\to\Grpd_{/\mathbf B_1}$.
\end{lemma}

Proposition~\ref{prop:Delta(f)} gives the comultiplication on
$\Grpd_{/\mathbf B_1}$ as
  \begin{align*}
    \Delta&(\name S)
            \;= \;
            \int^{A,B}\Bij(A+B,S)
            \cdot\name A\tensor\name B
\;=
\!\!\!\!\!\!\!\!\!\!
\sum_{
  \substack{A ,B\subset S \\  A\cup B = S,\; A\cap B = \emptyset}
  }
\!\!\!\!\!\!\!\!\!\!
 \name A \tensor \name B .
  \end{align*}
It follows that the convolution product on $\Grpd^{\B}$
is just the Cauchy product on groupoid-valued species
  $$
  (F*G)[S] = \sum_{A+B=S} F[A] \times G[B] .
  $$
For the representables, the formula says simply
$h^A*h^B=h^{A+B}$.

The decomposition space {\bf B} is locally finite, and taking cardinality gives
the classical binomial coalgebra $\Inc_{\mathbf B} = \Q_{\pi_0 \mathbf{B}_1}$,
with basis given by the symbols $\delta_n$ and
$$
\Delta(\delta_n) =
\sum\limits_{a+b=n} \frac{n!}{a!\,b!}\,\delta_a \tensor \delta_b.
$$
As a bialgebra we have $(\delta_1)^n=\delta_n$ and one recovers the
comultiplication from $\Delta(\delta_n)= \big( \delta_0 \tensor \delta_1 +
\delta_1 \tensor \delta_0 \big)^n$.

  Dually, the incidence algebra $\Inc^{\mathbf{B}}$ is the
  profinite-dimensional vector space $\Q^{\pi_0\mathbf{B}_1}$ with basis
  given by the
  symbols $\delta^n$ and with convolution product
  $$
  \delta^a* \delta^b = \frac{(a+b)!}{a!\,b!}\,\delta^{a+b} .
  $$
  This is isomorphic to the algebra $\Q[[z]]$, where $\delta^n$ corresponds to
  $z^n/n!$ and the cardinality of a species $F$ corresponds to its exponential
  generating series.
\end{blanko}

\begin{blanko}{Finite ordered sets,
    and the shuffle product of $\mathbb L$-species.}
  Let $\mathbf{OI}$ denote (the fat nerve of) the category of finite ordered
  sets and monotone injections.  The resulting incidence coalgebra can be
  reduced by identifying two monotone injections if they have isomorphic
  complements, in analogy with Example~\ref{ex:I=DecB}, yielding in this case
  the {\em shuffle coalgebra}.  Again, this reduction is an example of decalage.
  Consider the decomposition space $\mathbf{Z}$ with $\mathbf{Z}_n=
  \mathbf{OI}_{/\un n}$, the groupoid of arbitrary maps from a finite ordered
  set $S$ to $\un n$, or equivalently of $n$-shuffles of $S$.  This provides a
  direct construction of the shuffle coalgebra.  This example is subsumed in the
  theory of restriction species, developed in~\cite{GKT:restriction}.  The
  section coefficients are the binomial coefficients, but we may now note that
  on the objective level the convolution algebra is the shuffle product of
  $\mathbb L$-species (cf.~\cite{Bergeron-Labelle-Leroux}).

  There is a natural identification
  $$
  \mathbf{OI} \isopil \Decbot{\mathbf{Z}},
  $$
  which takes a sequence of monotone injections to the list of successive
  complements. There is also a \culf functor $\mathbf Z\to \mathbf B$ that takes
  an $n$-shuffle to the underlying $n$-tuple of subsets, and the decalage of
  this functor is the \culf functor $\mathbf{OI}\to\mathbf{I}$ given by
  forgetting the order (see \cite[Example 4.5]{GKT:DSIAMI-1}). Combining with
  Lemma~\ref{lem:IB}, we get altogether this commutative diagram of monoidal
  decomposition spaces and monoidal \culf functors,
$$\xymatrix{
  \mathbf{OI}\rto^-\simeq\dto&\Decbot{\mathbf Z}\dto\rto^-{d_\bot}&\mathbf{Z}\dto
  \\
  \mathbf{I}\rto^-\simeq&\Decbot{\mathbf{B}} \rto_-{d_\bot} & \mathbf{B} . }
$$
\end{blanko}

\begin{blanko}{Words.}
  Let $A$ be a fixed set, an alphabet.  The comma category $\mathbf{OI}_{/A}$ is
  the category of finite words in $A$ and subword inclusions,
  cf.~Lothaire~\cite{Lothaire:MR675953} (see also D\"ur~\cite{Dur:1986}).  Again
  it is naturally identified with the decalage of the {\em $A$-coloured shuffle
  decomposition space} $\mathbf Z_A$, which in degree $k$ is the groupoid of
  $A$-words (of arbitrary length) equipped with a
  not-necessarily-order-preserving map to $\un k$.  Precisely, the objects are
  spans of sets
  $$
  \un k \leftarrow \un n \to A .
  $$

  The dec map $\mathbf{OI}_{/A} \simeq \Decbot{\mathbf{Z}_A} \to \mathbf{Z}_A$
  takes a subword inclusion to its complement word. The incidence algebra
  $\Inc^{\mathbf{Z}_A}$ is the Lothaire shuffle algebra of words. Again, it all
  amounts to observing that $A$-words admit a forgetful monoidal \culf functor
  to $1$-words, which is just the decomposition space $\mathbf{Z}$ from before,
  and that this in turn admits a monoidal \culf functor to $\mathbf B$.

  Note the difference between $\mathbf Z_A$ and the free monoid on $A$: the
  latter is like allowing only the trivial shuffles, where the subword
  inclusions are only concatenation inclusions.  In terms of the structure maps
  $\un n \to \un k$, the free-monoid nerve allows only monotone maps, whereas
  the shuffle decomposition space allows arbitrary set maps.
\end{blanko}

\subsection{Multiplicative examples}

\label{sec:mult}

\begin{blanko}{Divisibility poset and multiplicative monoid.}\label{ex:D&M}
  In analogy with \ref{ex:N&L}, let $\mathbf D$ denote the (strict) nerve of the
  divisibility poset $(\N^\times, |)$, and let $\mathbf M$ be the strict nerve
  of the multiplicative monoid $(\N^\times, \cdot)$.  Imposing the equivalence
  relation `isomorphism of intervals' on the incidence coalgebra of $\mathbf D$
  gives that of $\mathbf M$, and
  Content--Lemay--Leroux~\cite{Content-Lemay-Leroux} observed that this
  reduction is induced by the \culf functor $r:\mathbf D \to \mathbf M$ sending
  $d| n$ to $n/d$.  In fact we have:

\begin{lemma}\label{lem:DtoM}
There is an isomorphism of simplicial sets
  \begin{align*}
   \Decbot{\mathbf M}  & \stackrel\simeq\longrightarrow  \mathbf D  \\
   \intertext{given in degree $k$ by}
    (x_0,x_1,\dots,x_k) & \longmapsto  [x_0| x_0x_1 | \dots | x_0x_1\cdots x_k],
  \end{align*}
  and the \culf functor $r$ is isomorphic to the dec map
  $$
  d_\bot:\Decbot{\mathbf M}\to\mathbf M,\qquad
  (x_0,\dots,x_k)\mapsto (x_1,\dots,x_k).
  $$
\end{lemma}

  This example can be obtained from Example \ref{ex:N&L} directly, since
  $\mathbf M= \prod_p \mathbf N$ and $\mathbf D=\prod_p \mathbf L$, where the
  (weak) product is over all primes $p$.  Now $\Decbot{}$ is a right adjoint, so
  preserves products, and Lemma \ref{lem:DtoM} follows from Lemma~\ref{ex:N&L}.

  Since $\mathbf{M}_0$ is contractible, we can use
  Corollary~\ref{sec coeffs for monoids}, and since $\mathbf{M}_1$ is 
  discrete, there are no symmetry factors, and we find that the convolution 
  product is
  $$\delta^m * \delta^n =
  \delta^{mn},
  $$ and the incidence algebra is isomorphic to the Dirichlet
  algebra:
  \begin{eqnarray*}
    \Inc^{\mathbf M} & \longrightarrow & \{ \sum_{k>0} a_k k^{-s} \}   \\
    \delta^n & \longmapsto & n^{-s} \\
    f & \longmapsto & \sum_{n>0} f(n) n^{-s} .
  \end{eqnarray*}
\end{blanko}

\begin{blanko}{Arithmetic species.}
  The Dirichlet coalgebra (\ref{ex:D&M}) also has a fatter version: consider now
  instead the monoidal groupoid $(\B^\times, \times, 1)$ of non-empty finite
  sets under the cartesian product, and its monoidal nerve $\mathbf A$ with
  $\mathbf A_k := (\B^\times)^k$, as in \ref{monoidalgroupoids}, where this time
  the inner face maps take the cartesian product of two adjacent factors, and
  the outer face maps project away an outer factor.

  The resulting coalgebra structure is
  $$
  \Delta(S) = \sum_{A\times B \simeq S} A \tensor B .
  $$
  Some care is due to interpret this correctly: the homotopy fibre of
  $d_1:\mathbf A_2 \to \mathbf A_1$ over $S$ is the groupoid whose objects are
  triples $(A,B,\phi)$ consisting of sets $A$ and $B$ equipped with a bijection
  $\phi: A \times B \isopil S$, and whose morphisms are pairs of isomorphisms
  $\alpha: A \isopil A'$, $\beta: B \isopil B'$ forming a commutative square
  with $\phi$ and $\phi'$.

  The corresponding incidence algebra $\grpd^{\B^\times}$ with the convolution
  product is the algebra of arithmetic species \cite{Baez-Dolan:zeta} under the
  Dirichlet product (called the arithmetic product of species by Maia and
  M\'endez~\cite{Maia-Mendez:0503436}).

  Clearly we are in the locally finite situation; since $\mathbf{A}_0$ is 
  contractible, the section coefficients are
  given directly by Corollary~\ref{sec coeffs for monoids}:
  $$
  \delta^m * \delta^n = \frac{(mn)!}{m!n!} \, \delta^{mn}   .
  $$
  From this we see that the incidence algebra $\Inc^{\mathbf{A}}$ is isomorphic
  to the Dirichlet algebra, namely
  \begin{eqnarray*}
    \Inc^{\mathbf A} & \longrightarrow & \{ \sum_{k>0} a_k k^{-s} \}  \\
    \delta^m & \longmapsto & \frac{m^{-s}}{m!} \\
    f & \mapsto & \sum_{k>0} f(k) \frac{k^{-s}}{k!};
  \end{eqnarray*}
  these are the `exponential' (or modified) Dirichlet series
  (cf.\ Baez--Dolan~\cite{Baez-Dolan:zeta}).
  So the incidence algebra zeta function in this setting is
  $$
  \zeta = \sum_{k>0} \delta^k \mapsto \sum_{k>0} \frac{k^{-s}}{k!}
  $$
  (which is not the usual Riemann zeta function).
\end{blanko}

\subsection{Linear examples and the Waldhausen $S_\bullet$-construction}

\label{sec:Wald}

In this subsection, we are concerned with linear versions of the additive
examples: instead of starting with finite sets and injections, we look at vector
spaces over a finite field, and their linear injections. This is a richer
setting: in particular, there is now an essential difference between quotients
and complements, which at the level of decomposition spaces is the difference
between the Waldhausen $S_\bullet$-construction and the monoidal nerve of direct
sums, as we shall see.

\begin{blanko}{$\F_q$-vector spaces.}
  Let $\mathbb F_q$ denote a finite field with $q$ elements.  Let $\mathbf W$
  denote the fat nerve of the category $\vect^{\text{inj}}$ of
  finite-dimensional $\mathbb F_q$-vector spaces and $\mathbb F_q$-linear
  injections.  From this decomposition space we immediately get a coalgebra, but
  it is not the most interesting.
\end{blanko}

\begin{blanko}{Direct sums of $\F_q$-vector spaces and
    `Cauchy' product of $q$-species.}
  A coalgebra which is the $q$-analogue of $\mathbf B$ can be obtained from the
  monoidal groupoid $(\vect^{\operatorname{iso}}, \oplus, 0)$.  Denote by
  $\mathbf M$ the monoidal nerve of $(\vect^{\operatorname{iso}}, \oplus, 0)$,
  in the sense of \ref{monoidalgroupoids}.  We compute the section coefficients
  directly from the definition~\eqref{sec coeffs}: the fibre of $d_1: \mathbf M_2 \to
  \mathbf M_1$ over a vector space $V$ is the groupoid consisting of triples
  $(A,B,\varphi)$ where $\varphi$ is a linear isomorphism $A\oplus B\isopil V$.  This
  groupoid projects to
  $\vect^{\operatorname{iso}}\times\vect^{\operatorname{iso}}$: the
  fibre over $(A,B)$ is discrete, of cardinality $\norm{\Aut(V)}$, giving
  altogether the section coefficient
  $$
  c_{k,n-k}^n =
  \frac{\norm{\Aut(\F_q^n)}}{\norm{\Aut(\F_q^k)}\norm{\Aut(\F_q^{n-k})}}
  = q^{k(n-k)} \, { n \choose k}_q .
  $$
  (We could also have invoked Corollary~\ref{sec coeffs for monoids},
  but using the definition directly has the pedagogical advantage that it also 
  works for the closely related Example~\ref{ex:q} below.)

  At the objective level, this convolution product corresponds to the `Cauchy'
  product of $q$-species in the sense of Morrison~\cite{Morrison:0512052}.

  If we let $\delta_n$ denote the cardinality of the name of an $n$-dimensional
  vector space $V$, the resulting coalgebra $\Inc_{\mathbf{M}}$ therefore has
  comultiplication:
  $$
  \Delta(\delta_n) = \sum_{k\leq n}
  q^{k(n-k)}\, {n \choose k}_q \cdot\; \delta_k \tensor \delta_{n-k} .
  $$

  In analogy with the discrete case discussed in
  \ref{ex:I=DecB}--\ref{lem:I=DecB}
  there is a canonical simplicial map
  $\Decbot{\mathbf M} \to \mathbf W$, given by sending an $(n+1)$-tuple
  of vector spaces $(V_0,\ldots,V_n)$ to the sequence of inclusions
  $$
  V_0 \into V_0 \oplus V_1 \into \cdots \into V_0 \oplus \cdots\oplus V_n.
  $$
  But in contrast to Lemma~\ref{lem:I=DecB}, this simplicial map is not an
  equivalence: the inverse, which in the discrete case was
  `taking complements', does not exist in the linear case (or if it is
  constructed artificially, for example by reference to euclidean structure, it
  will mess with the isomorphisms).
  Let us actually compute $\Decbot{\mathbf M}$.
\end{blanko}

\begin{blanko}{Complements as retractions.}
  Let $\mathbf W^{\operatorname{retr}}$ denote the fat nerve of the category
  whose objects are finite-dimensional $\F_q$-vector spaces and whose morphisms
  are retracted injections (linear of course)
  $$\xymatrix{
     V \ar@{ >->}[r]<-2pt>_-i & \ar@{->>}[l]<-4pt>_-r V'
  }$$
  Such retracted injections have canonical complements, namely $\ker(r)$.  The
  following analogue of Lemma~\ref{lem:I=DecB} is now straightforward to
  establish.
\end{blanko}

\begin{lemma}
  There is a canonical equivalence of simplicial groupoids
  \begin{align*}
    \Decbot{\mathbf M}  &
            \stackrel\simeq\longrightarrow\mathbf{W}^{\operatorname{retr}}  \\
   \intertext{given in degree $k$ by}
    (V_0,\dots,V_k) & \longmapsto
[V_0 \subseteq V_0\oplus V_1 \subseteq\dots\subseteq V_0\oplus\dots\oplus V_k]
  \end{align*}
  inducing a \culf functor $\mathbf{W}^{\operatorname{retr}}\to \mathbf M$.
\end{lemma}

The discussion shows that altogether $\mathbf M$ is not the most interesting
viewpoint.  We now change perspective from complements to quotients, getting
to the more important power series representation with factor $[n]!$ instead
of $\norm{\Aut(\F_q^n)}$, and
realise $\mathbf W$ as a decalage, in analogy with Lemma~\ref{lem:I=DecB}.

\begin{blanko}{$q$-binomials.}\label{ex:q}
  With reference to the incidence coalgebra of $\mathbf W$,
  impose the equivalence relation identifying
  two injections if their cokernels are isomorphic.  This gives the
  $q$-binomial coalgebra (see D\"ur~\cite[1.54]{Dur:1986}).

  The same coalgebra can be obtained without reduction as follows.  Put
  $\mathbf V_0 = 1$, let $\mathbf V_1$ be the maximal groupoid of
  $\vect$, and let $\mathbf V_2$ be the groupoid of short exact sequences.
  The span
  $$
   \xymatrixrowsep{4pt}
  \xymatrix {
  \mathbf V_1 & \ar[l] \mathbf V_2 \ar[r]  &\mathbf V_1\times\mathbf V_1 \\
  E & \ar@{|->}[l] [ E' \!\to\! E \!\to\! E''] \ar@{|->}[r]  & (E',E'')
  }$$
  (together with the span $\mathbf V _1 \leftarrow \mathbf V_0 \to 1$) defines a
  coalgebra structure on $\grpd_{/\mathbf V_1}$ which (after taking cardinality)
  is the $q$-binomial coalgebra, without further reduction.  The groupoids and
  maps involved are part of a simplicial groupoid $\mathbf V:
  \simplexcategory\op\to\Grpd$, namely the Waldhausen $S_\bullet$-construction
  of $\vect$, studied in more detail below, where we'll see that this is a
  decomposition space but not a Segal space.  The lower dec of $\mathbf V$ is
  naturally equivalent to the fat nerve $\mathbf W$ of the category of
  injections, and the dec map $d_\bot$ is the reduction map of D\"ur.

  We calculate the section coefficients of $\mathbf V$. Since $\mathbf V$ is
  not a Segal space, we cannot invoke Corollary~\ref{sec coeffs for monoids},
  so we use the definition of section coefficients \eqref{sec coeffs} directly,
  to find the following
  standard formula
  for the {\em Hall numbers} (cf.~also~\ref{Hall} below):
   $$
   c_{k,n-k}^n =
  \frac{\norm{\text{SES}_{k,n,n-k}}}{\norm{\Aut(\F_q^k)}\norm{\Aut(\F_q^{n-k})}}  .
  $$
  Here $\text{SES}_{k,n,n-k} = (\mathbf{V}_2)_{k,n,n-k}$ is 
  the groupoid of short exact sequences with
  specified vector spaces of dimensions $k$, $n$, and $n-k$.
  This is just a discrete groupoid, and it has
  $\norm{\Aut(\F_q^k)}\norm{\Aut(\F_q^{n-k})}{ n \choose k}_q $ elements.
  Indeed, there are ${ n \choose k}_q$
  $k$-dimensional subspaces of the $n$-dimensional space $ \F_q^n$, and hence $\norm{\Aut(\F_q^k)}{ n \choose k}_q $
  choices for the injection $\F_q^k \into \F_q^n$,  and then $\norm{\Aut(\F_q^{n-k})}$ choices for identifying the cokernel with $ \F_q^{n-k}$.
  Thus
  $$
  c^n_{n,n-k}\; =\;  { n \choose k}_q \;=\;\frac{[n]!}{[k]![n-k]!}.
  $$%
  This description gives an isomorphism of algebras
  (cf.~Goldman--Rota~\cite{Goldman-Rota:1970}, D\"ur~\cite{Dur:1986})
  \begin{eqnarray*}
    \Inc^{\mathbf V} & \longrightarrow & \Q[[z]]  \\
    \delta^k & \longmapsto & \frac{z^k}{[k]!} .
  \end{eqnarray*}

  Clearly this algebra is commutative.
  However, an important new aspect is revealed on the objective
  level: here the convolution product is the external product of $q$-species
  of Joyal-Street~\cite{Joyal-Street:GLn}.  They show (working with vector-space
  valued $q$-species), that this product has a natural
  non-trivial braiding (which of course reduces to commutativity upon taking
  cardinality).
\end{blanko}

\begin{blanko}{Waldhausen $S_\bullet$-construction of an abelian category.}
  The decomposition space with the short exact sequences
  leading to the Hall algebra is an example of
  Waldhausen's $S_\bullet$-construction \cite{Waldhausen},
  a centrepiece of modern $K$ theory.  We briefly explain this.

  The Waldhausen $S_\bullet$-construction of an abelian category $\mathscr A$ is
  a simplicial groupoid $S_\bullet\mathscr A$, with the following explicit
  description.  $S_0\mathscr A$ is a point, $S_1\mathscr A$ is the maximal
  groupoid in $\mathscr A$, and $S_2\mathscr A$ is the groupoid of short exact
  sequences in $\mathscr A$.  More generally, $S_n\mathscr A$ is the groupoid of
  staircase diagrams like the following (picturing $n=4$):
  $$\xymatrix{
  &&& A_{34} \\
  && A_{23} \ar@{ >->}[r] & A_{24} \ar@{->>}[u] \\
  & A_{12} \ar@{ >->}[r] & A_{13} \ar@{->>}[u]\ar@{ >->}[r] & A_{14} \ar@{->>}[u] \\
  A_{01} \ar@{ >->}[r] & A_{02} \ar@{->>}[u]\ar@{ >->}[r] & A_{03}
  \ar@{->>}[u]\ar@{ >->}[r] & A_{04} \ar@{->>}[u]
  }$$
  in which each sequence $A_{ij} \to A_{ik} \to A_{jk}$ is exact.
  The face map $d_i$ deletes all objects containing an $i$ index.
  The degeneracy map $s_i$ repeats the $i$th row and the $i$th column.

A sequence of composable monomorphisms
$(A_{1}\rightarrowtail A_{2}\rightarrowtail \dots \rightarrowtail A_{n})$
determines, up to canonical isomorphism,  short exact
sequences $A_{ij}\rightarrowtail A_{ik}\twoheadrightarrow
A_{jk}=A_{ij}/A_{ik}$ with $A_{0i}=A_i$.
Hence the whole diagram can be reconstructed up to isomorphism from the
bottom row.  (Similarly, since epimorphisms have uniquely determined kernels,
the whole diagram can also be reconstructed from the last column.)

We have  $s_0(*)=0$, and
\begin{align*}
d_0(A_1\rightarrowtail A_2\rightarrowtail \dots \rightarrowtail A_{n})&=
(A_2/A_1\rightarrowtail \dots \rightarrowtail A_{n}/A_1),
\\s_0(A_1\rightarrowtail A_2\rightarrowtail \dots \rightarrowtail A_{n})&=
(0\rightarrowtail
A_1\rightarrowtail
A_2\rightarrowtail \dots \rightarrowtail A_{n}) .
\end{align*}
The simplicial maps $d_i,s_i$ for $i\geq1$ are more straightforward:
the simplicial set $\Decbot{(S_\bullet\mathscr A)}$ is just the fat nerve of
$\mathscr A^{\operatorname{mono}}$.
\end{blanko}

\begin{lemma}
  The projections
  $S_{n+1} \mathscr A \to \Map([n], \mathscr A^{\operatorname{mono}})$
  and
  $S_{n+1} \mathscr A \to \Map([n],\mathscr A^{\operatorname{epi}})$
  are equivalences of groupoids.
\end{lemma}
More precisely (with reference to the fat nerve):

\begin{prop}
  These equivalences assemble into levelwise simplicial equivalences
$$
\Decbot{(S_\bullet\mathscr A)} \simeq \fatnerve ( \mathscr A^{\operatorname{mono}})
$$
$$
\Dectop{(S_\bullet\mathscr A)} \simeq \fatnerve ( \mathscr A^{\operatorname{epi}}) .
$$
\end{prop}

\begin{theorem}
  \cite[Theorem 7.3.3]{Dyckerhoff-Kapranov:1212.3563}, \cite[10.10]{GKT:DSIAMI-1}
  The Waldhausen $S_\bullet$-construction of an abelian category $\mathscr A$
  is a decomposition space.
\end{theorem}

\begin{blanko}{Generalised Waldhausen construction.}
  It was shown by Bergner, Osorno, Ozornova, Rovelli, and   Scheimbauer~\cite{Bergner-et.al:1609.02853, Bergner-Osorno-Ozornova-Rovelli-Scheimbauer:1809.10924,
  Bergner-Osorno-Ozornova-Rovelli-Scheimbauer:1901.03606}
  that in fact every decomposition space arises from an 
  $S_\bullet$-construction, but a more general one, taking as input 
  certain augmented stable double Segal spaces. The classical case is the 
  double Segal space whose horizontal morphisms are the monos and whose 
  vertical morphisms are the epis. See the contributions of 
  Rovelli~\cite{Rovelli:thisvolume} and 
  Ozornova~\cite{Ozornova:thisvolume} for an introduction and further 
  pointers.
\end{blanko}

\begin{blanko}{Hall algebras.}\label{Hall}
  The finite-support incidence algebra of a decomposition space $X$ was
  mentioned in \ref{fin-supp} (see
  \cite[7.15]{GKT:DSIAMI-2} for more details). In order for it to admit a
  cardinality, the required assumption is that $X_1$ be locally finite,
  that $X_0$ be finite, and that
  $X_2 \to X_1 \times X_1$ be a finite map.
  In the case of $X= S_\bullet(\mathscr A)$
  for an abelian category $\mathscr A$, this translates into the condition that
  $\Ext^0$ and $\Ext^1$ be finite (which in practice means `finite dimension
  over a finite field').  The finite-support incidence algebra in this case is
  the {\em Hall algebra} of $\mathscr A$ (cf.~Ringel~\cite{Ringel:Hall}; see
  also \cite{Schiffmann:0611617}, although these sources twist the
  multiplication by the so-called Euler form).

  Hall algebras were one of the main motivations for Dyckerhoff and
  Kapranov~\cite{Dyckerhoff-Kapranov:1212.3563} to
  introduce $2$-Segal spaces.  We refer to their work for development
  of this important topic, recommending the lecture notes of
  Dyckerhoff~\cite{Dyckerhoff:1505.06940} as a starting point; 
  see also the contribution of Cooper and
  Young~\cite{Cooper-Young:thisvolume} in this volume.
\end{blanko}

\subsection{\fdb bialgebra and variations}
\label{sec:fdb}

The \fdb bialgebra, originating with composition of power series, was
constructed combinatorially by let~\cite{Doubilet:1974} by imposing a {\em
type-equivalence} relation on the incidence coalgebra of the partition poset.
Joyal~\cite{JoyalMR633783} observed that it can also be realised directly from
the category of finite sets and surjections, without the need of a reduction
step.  Both constructions, and in particular the relationship between them, can
be cast elegantly in the framework of decomposition spaces, serving to
illustrate many of the characteristic aspects of the theory, such as the use of
groupoids and the role of decalage.

\begin{blanko}{\fdb from the partition poset.}
  Fix a finite set of each cardinality, denoted $\underline 0, \underline 1,
  \underline 2$, etc.  Let $\mathcal{P}(\underline m)$ denote the poset of
  partitions of the set $\underline m$; we write $\rho \leq \pi$ when partition
  $\rho$ refines partition $\pi$.  The {\em partition poset} is by definition
  the disjoint union of all these, $\mathcal{P} := \sum_{m\in \N}
  \mathcal{P}(\underline m)$.
  The nerve of $\mathcal{P}$ defines a coalgebra (which is furthermore a
  bialgebra, with multiplication given by disjoint union).
  More interesting is the reduction of this bialgebra modulo type equivalence.
  An interval $[\rho,\pi]$ in a poset $\mathcal{P}(\underline m)$ is said to
  have {\em type} $1^{\lambda_{1}} 2^{\lambda_{2}} \cdots$ if $\lambda_i$
  is the number of blocks of $\pi$ that consist of exactly $i$ blocks of $\rho$.
  Declare two intervals equivalent if they have the same type.
  The reduced incidence coalgebra of the poset $\mathcal{P}$
  (with reduction given by type equivalence) is the {\em \fdb bialgebra} 
  (we gloss over the multiplicative structure, as it will be clearer in 
  the viewpoint of surjections coming up next).
\end{blanko}

\begin{blanko}{Partitions as surjections.}
  A partition $\rho$ of $\underline m$ can be realised as a surjection
  $\underline m \onto \underline n$,
  where $\underline n$ is the set of blocks.  An interval $[\rho,\pi]$,
  from $\rho:\underline m \onto \underline n$
  to $\pi:\underline m \onto \underline k$ say, is then realised as
  a (unique) comparison surjection $\lambda$ in a commutative triangle
  \begin{equation}\label{eq:under n}
    \vcenter{\xymatrix@C=1.5ex{  & \un m \ar@{->>}[ld]_\rho\ar@{->>}[rd]^\pi \\
                                \un n    \ar@{->>}[rr]_\lambda&& \un k.}
  }
  \end{equation}
  The equivalence type of the interval $[\rho,\pi]$ is then
  precisely the isomorphism type of the surjection $\lambda:\un n\to\un 
  k$, namely
    \begin{equation}\label{eq:type}
      1^{\lambda_{1}} 2^{\lambda_{2}}\cdots n^{\lambda_{n}}
    \end{equation}
    if $\lambda$ has $\lambda_i$ fibres of cardinality $i$.
    Identifying intervals of the same type therefore amounts to forgetting
  the ambient set $\underline m$.

  Splitting of intervals, as in the formula for comultiplication in the
  incidence coalgebra,
  $$
  \Delta( [\rho,\pi] ) = \sum_{\sigma\in [\rho,\pi]} [\rho,\sigma] \tensor 
  [\sigma,\pi] ,
  $$
  is now precisely factorisation of such comparison
  surjections, so as to get
  $$
\Delta(\un n\shortsur \un k) =
\sum_{\un n\onto \un s\onto \un k} (\un n \shortsur \un s) \otimes 
(\un s \shortsur \un k) ,
$$
but some care is needed to count these factorisations correctly. 
The sum is over \isoclasses of factorisations $\un n \onto \un s 
\onto \un k$.
In detail,
consider the {\em factorisation groupoid}
$\operatorname{Fact}(\un n\shortsur \un k)$,
whose objects are factorisations of $\un n\shortsur \un k$
into two surjections $\un n\shortsur \un s\shortsur \un k$,
and whose morphisms are bijections $\un s\simeq \un s'$ making the two triangles
commute:
\begin{equation}\label{eq:fact}
  \xymatrix@R1ex{
&\un s\ar[dd]^{\simeq}\ar@{->>}[rd]& \\
\un n\ar@{->>}[ru]\ar@{->>}[rd] && \un k \,.\\
& \un s'\ar@{->>}[ru]
}
\end{equation}
Then the above sum is over $\pi_0(\operatorname{Fact}(\un n\shortsur \un k))$,
the set of connected components of the factorisation groupoid.
This is Joyal's construction of the \fdb bialgebra~\cite{JoyalMR633783}
(and now the algebra structure is simply given by disjoint union of 
surjections).
\end{blanko}

\begin{blanko}{The decomposition space of surjections.}\label{surjection-gpd}
In decomposition space language, the \fdb
bialgebra is simply the incidence bialgebra of the monoidal decomposition space
$\mathbf S$ given as the fat nerve of the category of finite sets and
surjections. Indeed, it has
as $\mathbf S_0$ the groupoid of finite sets and bijections,
and
as $\mathbf S_1$ the groupoid whose objects are surjections and
whose morphisms are squares
$$
\xymatrix{ \underline n \ar@{->>}[r]\ar[d]_{\sim}
  & \underline k \ar[d]_{\sim}  \\
  \underline n' \ar@{->>}[r] & \underline k'.\!\!\!}$$
$\mathbf S_2$ is the groupoid whose objects are composable pairs of surjections,
and whose morphisms are diagrams
$$
\xymatrix{ \underline n \ar@{->>}[r]\ar[d]_{\sim}
  & \underline s \ar@{->>}[r]\ar[d]_{\sim} &
  \underline k \ar[d]_{\sim}  \\
  \underline n' \ar@{->>}[r] & \underline s'\ar@{->>}[r]
  & \underline k' .\!\! }$$
The homotopy fibre of $d_1:\mathbf S_2\to\mathbf S_1$ over $\lambda\in\mathbf S_1$ is
equivalent to the subgroupoid whose objects are composable pairs composing to
$\lambda$ and whose morphisms are diagrams~\eqref{eq:fact}. It follows readily that
the incidence bialgebra is precisely the \fdb bialgebra \`a la Joyal.

  Note that in the groupoid setting, there is no need to restrict to a skeleton
  of the category of finite sets and surjections.  We may as well work with the
  whole groupoid, without making choices.  The homotopy equivalences take care
  of `dividing out', while keeping the correct automorphism data.
\end{blanko}

\begin{blanko}{Relationship via decalage.}
  Having interpreted partitions as surjections, and refinements as factorisations
  of surjections, one may suspect that the partition poset is the decalage of
  the surjections nerve.  This is almost correct, but not quite: there is a
  subtle difference related to symmetries (pointed out by Mark Weber), which in
  turn originates in the fact that the partition poset is based on chosen fixed
  sets $\underline m$.   Getting this straight is a nice opportunity to see some
  \culf functors:

  We first analyse the relationship between partitions and surjections. Two
  surjections $\underline m \onto \underline n$ and $\underline m
  \onto\underline k$ represent the same partition if there is a bijection
  $\underline n\isopil \underline k$ making the triangle \eqref{eq:under n}
  commute. Consider the category $\CC$ whose objects are surjections between the
  standard sets $\underline m$ and whose morphisms are triangles like
  \eqref{eq:under n}.
  Since there is at most
  one arrow $\lambda$ between any two surjections, this category is equivalent
  to a poset, and indeed equivalent to the partition poset $\mathcal{P}$. It
  does contain non-trivial isomorphisms, but all automorphisms are trivial.
  Since this category $\CC$ and the partition poset are equivalent as
  categories, their fat nerves are (levelwise) equivalent decomposition spaces,
  and therefore define equivalent coalgebras. Note that for strict posets, the
  fat nerve is the same thing as the strict nerve.

This category $\CC$ sits in a bigger
category $\DD$ with the same
objects (surjections), but where maps between two surjections are allowed to
have a non-identity bijection between the domains, instead of just an identity
arrow. The categories $\CC$ and $\DD$ are not equivalent, and their fat nerves
are not equivalent, and their bialgebras are not isomorphic---all because of the
different amount of symmetry they sport at the surjection domains. However, it
is clear that they have exactly the same type reduction, since the type
reduction precisely throws away the surjection domains.
\end{blanko}

\begin{lemma}
  The inclusion functor $\CC \to \DD$ is \culf.
\end{lemma}
\noindent
  Essentially this is for the same reason as the type reduction argument:
  \culf-ness is about factorisations of the codomain maps, and this is not
  affected by what happens at the domain.

  Now that in $\DD$ we have symmetries built in naturally, there is no reason
  to restrict to the skeleton anymore. As an equivalent $\DD'$ we can take the
  same description but allow the objects to be surjections between arbitrary
  finite sets, instead of just those chosen sets $\underline n$. Note that
  this bigger category $\DD'$ has a natural interpretation in terms of
  partitions: suppose we want a notion of partition, but do not wish to
  restrict attention to those chosen sets $\underline n$. We would then have
  to say when two partitions are considered the same, and more generally what
  should be the notion of morphism of partitions: the natural notion is to
	  have a bijection at the domain that preserves block membership,
  i.e.~a bijection $f$ such that if $t_1$ and $t_2$ belong to the same
  block, then also $ft_1$ and $ft_2$ belong to the same block. Such
  a bijection induces a unique surjection between the codomains and
  it is clear
  that this notion of morphism is precisely a refinement.
  This is a category rather
  than a poset: it mixes the poset structure with the invertible maps given by
  renaming of set elements. Note that a partition in $\DD$ or in $\DD'$ has more
  automorphisms than in $\CC$: for example a $(2,2)$-partition has an
  automorphism group of order $8$, namely $4$ possibilities to permute within
  the blocks, and $2$ possibilities of interchanging the blocks.

With these extra symmetries we have

\begin{lemma}
  The fat nerve of $\DD'$ is naturally equivalent to $\Decbot{\mathbf S}$.
\end{lemma}

Altogether:

\begin{prop}
  Type reduction, and the relationship between the partition poset and the
  surjections nerve is given by the string of \culf functors
  $$
  N\mathcal{P} \simeq \fatnerve\CC \longrightarrow \fatnerve\DD' \simeq
	\Decbot{\mathbf S} \stackrel{d_\bot} \longrightarrow \mathbf S .
  $$
  The composite sends a partition to its set of blocks, and sends a refinement
  to the corresponding surjection as in \eqref{eq:under n}.
\end{prop}
The verifications are straightforward.  It should be noted that these
\culf functors are all monoidal, and hence induce bialgebra homomorphisms (by
Proposition~\ref{prop:bialg}).

\begin{blanko}{\fdb section coefficients.}\label{fdb-sectioncoeffs}
  We work with the decomposition space $\mathbf S$. Since $\mathbf S$ is
  monoidal, to describe its section coefficients, it is enough describe the
  comultiplication of connected surjections $n\onto 1$.
  Write $\AAA_n$ for
  the cardinality of $\name{n\onto 1}$ in $\grpd_{/\mathbf S_{1}}$. Since a
  general surjection $\lambda:n\onto k$ of type $1^{\lambda_1}2^{\lambda_2}
  \cdots$ is isomorphic to a disjoint union of connected surjections, its
  cardinality is $\AAA_{1}^{\lambda_{1}}\AAA_{2}^{\lambda_{2}}\cdots$.
  We will apply the general formula from
  Proposition~\ref{sectioncoeff!}. For a given factorisation $\un n
  \stackrel{\lambda}\onto \un k \stackrel{\pi}\onto \un 1$, we have to look
  at the set $\{\varphi\}$ of those isomorphisms $\varphi:\un k\isopil \un
  k$ that make the composite surjection $\lambda\varphi\pi$ equal to the
  original surjection $\un n \onto \un 1$. But since $\un 1$ is terminal,
  all $\varphi$ will do, so the cardinality of this set is $k!$ (see
  \cite{GKT:corrigendum} for discussion of this point). With this, the
  general formula \ref{sectioncoeff!} gives
  $$
  \Delta(\AAA_{n}) =
  \sum_{
  \substack{
  \un n \onto \un k\\  \un k\onto \un 1
  }
  }
  \frac{\norm{\Aut(\un n\onto \un 1)}
   \cdot k!}
       {\norm{\Aut(\un n\onto\un k)}
         \cdot\norm{\Aut(\un k\onto\un 1)}} \;
       \AAA_{1}^{\lambda_{1}}\cdots \AAA_{n}^{{\lambda_{n}}}\;\otimes\; \AAA_{k} .
  $$
  The sum is over all distinct \isoclasses of pairs of
  surjections (for fixed $n$).
  The formula for the order of the automorphism group of 
  a surjection $\lambda:\un n\onto \un 1$ is
  \begin{equation}\label{eq:aut-surj}
  \norm{\Aut(\un n\onto \un k)}\;=\;\prod_{j=1}^\infty\lambda_j ! 
  (j!)^{\lambda_j} ,
  $$
  and in particular $\norm{\Aut(\un k\onto \un 1)} = k!$. Altogether we find
  $$
  \Delta(\AAA_{n}) =
  \sum_{
  \un n \onto \un k
  }
  \frac{n! }{\prod_{j=1}^n\lambda_j ! (j!)^{\lambda_j}}\;
       \AAA_{1}^{\lambda_{1}}\cdots \AAA_{n}^{{\lambda_{n}}}\;\otimes\; \AAA_{k} .
  \end{equation}
  For example,
  \begin{align*}
  \Delta(\AAA_3)  
  &=  
  \AAA_3 \tensor \AAA_1  + 3 \AAA_1 \AAA_2\tensor \AAA_2  + \AAA_1\AAA_1\AAA_1 \tensor \AAA_3
  \\
  \Delta(\AAA_4)
  &=
  \AAA_4 \tensor \AAA_1  + (4 \AAA_1 \AAA_3 +  3 \AAA_2 \AAA_2) \tensor \AAA_2  
  + 6 \AAA_1\AAA_1 \AAA_2 \tensor \AAA_3  + \AAA_1\AAA_1\AAA_1\AAA_1 \tensor \AAA_4   
\end{align*}

  The section coefficients, called the \fdb section coefficients,
  are the coefficients
  ${n \choose \lambda;k}$ of the Bell polynomials,
  cf.~\cite[(2.5)]{Figueroa-GraciaBondia:0408145}.

  Using a different normalisation we may choose the basis 
  $\aaa_{n}=\AAA_{n}/n!$,
on which the comultiplication formula becomes
\[
\Delta(\aaa_{n})=  \sum_{
  \un n \onto \un k
  }
  \frac{k!}{\prod_{j=1}^n\lambda_j ! }\;
       \aaa_{1}^{\lambda_{1}}\cdots \aaa_{n}^{{\lambda_{n}}}\;\otimes\; 
	   \aaa_{k} .
     \]
For example,
\begin{align*}
  \Delta(\aaa_3) & =  \aaa_3 \tensor \aaa_1 + 2 \aaa_1 \aaa_2\tensor \aaa_2  + \aaa_1\aaa_1\aaa_1 
  \tensor \aaa_3  \\
  \Delta(\aaa_4) & =  \aaa_4 \tensor \aaa_1  + (2 \aaa_1 \aaa_3 +  \aaa_2 \aaa_2) \tensor \aaa_2 
  + 3 \aaa_1 \aaa_1 \aaa_2 \tensor \aaa_3  + \aaa_1\aaa_1\aaa_1\aaa_1 \tensor \aaa_4
\end{align*}

\end{blanko}

\begin{blanko}{A decomposition space for the \fdb\ {\em Hopf} algebra.}
  The \fdb\ {\em Hopf} algebra is obtained by further reduction, classically
  stated as identifying two intervals in the partition poset if they are
  isomorphic as posets.  This is equivalent to forgetting the value of
  $\lambda_1$.  There is also a decomposition space that yields this Hopf
  algebra directly, obtained by quotienting the decomposition space $\mathbf S$
  by the same equivalence relation.  This means identifying two surjections (or
  sequences of composable surjections) if one is obtained from the other by
  taking disjoint union with a bijection.  One may think of this as `levelled
  forests modulo linear trees'.  It is straightforward to check that this
  reduction respects the simplicial identities so as to define a simplicial
  groupoid, that it is a monoidal decomposition space, and that the quotient
  map from $\mathbf S$ is monoidal and \culf.
\end{blanko}

\begin{blanko}{Ordered surjections.}\label{ex:OS}
  Let $\mathbf{OS}$ denote the fat nerve of the category of finite ordered set
  and monotone surjections.  It is a monoidal decomposition space under ordinal
  sum.  Hence to describe the resulting comultiplication, it is enough to say
  what happens to a connected ordered surjection, say $ \un n \onto \un 1$,
  which we denote simply $n$: since there are no automorphisms around, we find
  $$
  \Delta(n) = \sum_{k=1}^n \sum_a a \tensor k
  $$
  where the second sum is over the ${n-1 \choose k-1}$ possible surjections
  $\lambda:n\onto k$.
  The resulting bialgebra is essentially the (dual) Landweber--Novikov
  bialgebra in algebraic topology~\cite{MoravaFdB1993} (see also
  \cite{Baues:Hopf}), the noncommutative Fa\`a di Bruno bialgebra in
  combinatorics~\cite{BFKAdv2006}, and the Dynkin--Fa\`a di Bruno bialgebra in
  numerical analysis~\cite{MuntheKaas:BIT95}; it also comes up in number theory
  \cite{Goncharov:Hopf}.  See \cite{GaKaTo2016} and \cite{Kock-Weber:1609.03276}
  for recent perspectives.
\end{blanko}

\subsection{Trees and graphs}

\label{sec:trees}

Various bialgebras of trees and graphs can be realised as incidence bialgebras
of decomposition spaces which are not Segal. This means that one can decompose
but not compose, as already exemplified in the running example with
graphs~\ref{ex:graphs-decomp}. In each case the lack of composability is caused
by the decomposition destroying info that would have been needed to define a
composition. As we shall see (in \ref{Ptrees}), it is sometimes possible to
`remedy' this to get instead a decomposition space which is Segal, at the price
of giving up connectedness of the bialgebra. In the examples based on graphs and
trees, this involves keeping `open-ended' edges, and is intimately related to
the theory of operads and related structures (\ref{ex:monads}).

\medskip

All the examples of decomposition spaces in  this subsection are monoidal
under disjoint union, and hence the resulting coalgebras are bialgebras.

\begin{blanko}{Butcher--Connes--Kreimer Hopf algebra.}\label{ex:CK}
  A {\em rooted tree} is a connected and simply connected graph with a specified
  root vertex; a forest is a disjoint union of rooted trees.  The
  Butcher--Connes--Kreimer Hopf algebra of rooted trees
  \cite{Connes-Kreimer:9808042} is the free algebra on the set of \isoclasses of
  rooted trees, with comultiplication defined by summing over certain admissible
  cuts $c$:
  $$
  \Delta(T) = \sum_{c\in \operatorname{adm.cuts}(T)} P_c \tensor R_c\,.
  $$
  An admissible cut $c$ is a splitting of the set of nodes into two subsets,
  such that the second forms a subtree $R_c$ containing the root node (or is the
  empty forest); the first subset, the complement `crown', then forms a
  subforest $P_c$, regarded as a monomial of trees. (The order of the two
  factors is dictated by an operadic viewpoint, where leaves are `in' and the
  root is `out', and is further justified in \ref{ex:monads} below.)

  D\"ur~\cite{Dur:1986} (Ch.IV, \S3) gave an incidence-coalgebra construction of
  the Butcher--Connes--Kreimer coalgebra by starting with the category $\CC$ of
  forests and root-preserving inclusions, generating a coalgebra (in our
  language the incidence coalgebra of the fat nerve of $\CC$), and imposing the
  equivalence relation that identifies two root-preserving forest inclusions if
  their complement crowns are isomorphic forests.

  Note that to be precise, one must
use $\CC\op$ instead of $\CC$:
  $$
  \mathbf R := \fatnerve (\CC\op).
  $$
  From the viewpoint of the incidence coalgebra this `op' affects the
  comultiplication only by reversing the order of the tensor factors. We shall
  see shortly that the `op' originates in an upper-dec construction (compare
  \ref{upperdec-op}).

  We can obtain the Butcher--Connes--Kreimer coalgebra directly from a
  decomposition space: let $\mathbf H_1$ denote the groupoid of forests, and let
  $\mathbf H_2$ denote the groupoid of forests with an admissible cut.  More
  generally, $\mathbf H_0 $ is defined to be a point, and $\mathbf H_k$ is the
  groupoid of forests with $k-1$ compatible admissible cuts.  These form a
  simplicial groupoid in which the inner face maps forget a cut, and the outer
  face maps project away stuff: $d_\bot$ deletes the crown (everything above the
  top-most cut) and $d_\top$ deletes the bottom layer (the part of the forest
  below the bottom-most cut).  It is readily seen that $\mathbf H$ is not a
  Segal space: a tree with a cut cannot be reconstructed from its crown and its
  bottom tree, which is to say that $\mathbf H_2$ is not equivalent to $\mathbf
  H_1 \times_{\mathbf H_0} \mathbf H_1$.  It is straightforward to check that it
  {\em is} a decomposition space, in fact a symmetric monoidal one under
  disjoint union, and it is also clear from its construction that the resulting
  bialgebra is the Butcher--Connes--Kreimer Hopf algebra.  Note that the
  decomposition space is graded by the number of nodes (which is precisely the
  length filtration \ref{length}), and that it is connected since the empty
  forest is the only forest with zero nodes.
  
  To explain the relationship between the two constructions, note that
  admissible cuts are essentially the same thing as root-preserving forest
  inclusions: then the
  cut is interpreted as the division between the included forest and the forest 
  induced on the nodes in its
  complement.  In this way we see that $\mathbf H_k$ is the groupoid of $k-1$
  consecutive root-preserving inclusions. 
  Furthermore, there is a
  natural identification
  $$
  \Dectop{\mathbf{H}} \simeq \mathbf{R} = \fatnerve (\CC\op),
  $$
  where the `op' occurs since we are dealing with an upper dec, as in 
  \ref{upperdec-op}.
  Under this identification, the dec map $\Dectop{\mathbf H} \to \mathbf H$,
  always a
  (symmetric monoidal) \culf functor, realises
  precisely D\"ur's reduction: on $\mathbf R_1
  \to \mathbf H_1$ it sends a root-preserving forest inclusion to its crown,
  that is, its complement.  More generally, on $\mathbf R_k \to \mathbf H_k$ it
  sends a sequence
  of forest inclusions $ F_0 \subset F_1 \subset \dots \subset F_k $ to
  $$
  F_1\shortsetminus F_0 \subset \dots \subset F_k \shortsetminus F_0 .
  $$

\end{blanko}

\begin{blanko}{Restriction species and directed restriction species \cite{GKT:restriction}.}\label{restriction}\label{directedrestriction} The
  Butcher--Connes--Kreimer example shares important characteristics with the
  graph example of Schmitt, our running example in Section~\ref{sec:decomp}
  (Examples \ref{ex:graphs-decomp}, \ref{ex:graphs-coalg},
  \ref{ex:graphs-bialg}), but where in the graph example there are no
  constraints on the nature of the cuts, in the tree example, only certain
  order-respecting cuts are deemed admissible.
  
  Both
  examples can be subsumed in big families of decomposition spaces, which can be
  treated uniformly, namely decomposition spaces of restriction species, in the
  sense of Schmitt~\cite{Schmitt:hacs} (see also \cite{Aguiar-Mahajan}), and
  decomposition spaces of directed restriction species, introduced and studied
  in~\cite{GKT:restriction}.  Here we content ourselves with
  outlining the idea.

  A {\em restriction species}~\cite{Schmitt:hacs} is simply a presheaf of the category
  $\I$ of finite sets and injections.    Compared to a classical
  species~\cite{JoyalMR633783}, a restriction species is thus functorial not
  only on bijections but also on injections, meaning that a given structure on a
  set induces such structure also on any subset. 
  
  Given a restriction species
  $R: \I\op\to \Set$, a coalgebra is obtained on the set of \isoclasses of
  $R$-structures with comultiplication
  $$
  \Delta(X)=
  \sum_{A+B=S} X|A \tensor X|B , \qquad X \in R[S]
  $$
  and counit sending only the empty structures to $1$.  (This is the 
  construction of Schmitt~\cite{Schmitt:hacs}.)
   
  It is preferable to work with groupoid-valued species as in 
  \cite{Baez-Dolan:finset-feynman},
  rather than the traditional set-valued species.
  Given a (groupoid-valued) restriction species
  $R: \I\op\to \Grpd$, we construct a simplicial groupoid $\mathbf R$ where $\mathbf
  R_k$ is the groupoid of $R$-structures with an ordered partition of the
  underlying set into $k$ parts (possibly empty).  Functoriality on active maps
  is clear, by joining adjacent parts or inserting an empty part.  Functoriality
  on inert maps is about projecting away outer parts, and is possible precisely
  because $R$ is a restriction species.  This simplicial groupoid can be shown
  to be a decomposition space, and the resulting incidence coalgebra is the
  Schmitt coalgebra~\cite{GKT:restriction}.  Furthermore, morphisms of
  restriction species induce \culf functors and hence coalgebra homomorphisms.  A
  great many species are actually restriction species (such as various classes
  of graphs, matroids, and posets), providing in this way a large supply of
  decomposition spaces (which are not Segal spaces).

  The Butcher--Connes--Kreimer example is subsumed in a large class of examples
  coming from {\em directed restriction species}, a notion introduced in 
  \cite{GKT:restriction}.  Where ordinary restriction species are presheaves on
  finite sets and injections, directed restriction species are presheaves on 
  the category of finite posets and convex injections.  The definition formalises the idea of
  considering only decompositions compatible with the poset structure in a 
  certain way, as exemplified clearly by the notion of admissible cut.
\end{blanko}

\begin{blanko}{Operadic trees and $P$-trees.}\label{Ptrees}
  There is an important variation on the Butcher--Connes--Kreimer Hopf algebra
  (but it is only a bialgebra): instead of considering combinatorial trees one
  considers operadic trees (i.e.~trees with open incoming edges, and an
  open-ended root edge).  More generally one can consider $P$-trees for a
  finitary polynomial endofunctor $P$, i.e.~trees whose nodes are decorated by
  the operations of $P$.  For details on this setting, see 
  \cite{Kock:0807,Kock:MFPS28,Kock:1109.5785}, 
  \cite{GalvezCarrillo-Kock-Tonks:1207.6404}; it suffices here to note that the
  notion of $P$-tree covers many kinds of structured trees, such as planar
  trees, binary trees, effective trees, linear trees, words, and
  a large class of inductive data types (W-types).
  
  For operadic trees, when copying over the description to get a simplicial
  groupoid $X$ where $X_k$ is the groupoid of forests with $k-1$ compatible
  admissible cuts, there are two important differences, both due to the fact
  that the cuts cannot remove the edges, since this might 
  violate the local structure of the tree, (e.g.~being binary)---the cut
  leaves a trace of the edge on each side of the cut, in the form of an 
  open-ended edge.  One difference is that $X_0$ is not
  just a point: it is the groupoid of node-less forests.  The second difference
  is that
  unlike $\mathbf H$, the simplicial groupoid $X$
  is a Segal space; this follows from the Key Lemma
  of \cite{GalvezCarrillo-Kock-Tonks:1207.6404} (see
  \cite{Kock-Weber:1609.03276} for an abstract viewpoint).  The reason is that
  the `half-edges' left by the cut constitute enough data to
  reconstruct a tree with a cut from its bottom tree and crown
  by grafting.  More precisely, the Segal maps $X_k\to X_1 \times_{X_0} \cdots
  \times_{X_0} X_1$ return the layers seen in between the cuts, and they are 
  easily seen to be equivalences: given the layers separately,
  and a match of their boundaries, one can glue them together to reconstruct the
  original forest, up to isomorphism.  In this sense the operadic-forest
  decomposition space $X$ is a `category' with node-less forests as objects, 
  and arbitrary forests as morphisms, a forest being seen as a morphism from
  its leaves to its roots.  In
  this perspective, the decomposition space $\mathbf H$ of combinatorial forests is
  obtained from $X$ by throwing away the object information, i.e.~the data
  governing the composability constraints.  These two differences are crucial in
  the work on Green functions and \fdb formulae in
  \cite{GalvezCarrillo-Kock-Tonks:1207.6404,Kock:1512.03027,Kock-Weber:1609.03276}.
  
  There is a functor from operadic trees or $P$-trees to combinatorial trees
  which is {\em taking core} \cite{Kock:1109.5785}: it amounts to forgetting the
  $P$-decoration and shaving off all open-ended edges.  This defines a monoidal
  \culf functor $X \to \mathbf H$ which realises a bialgebra
  homomorphism from the bialgebra of operadic trees or $P$-trees to the
  Butcher--Connes--Kreimer Hopf algebra of combinatorial trees.

\end{blanko}

\begin{blanko}{Note about symmetries.}
  One {\em cannot} obtain the same bialgebra of trees
  (either the combinatorial or the
  operadic) by taking \isoclasses in each groupoid $X_k$: doing
  so would destroy symmetries that constitute an essential ingredient in the
  Butcher--Connes--Kreimer bialgebra.  Indeed, define a simplicial set $Y$ in
  which $Y_k = \pi_0(X_k)$, the set of \isoclasses of forests with $k$
  compatible admissible cuts.  Consider the tree $T$ 

  \tikzset{
	mytree/.pic={
      \draw (0.0, 0.0) -- (0.0, 0.28);
	  \fill (0.0, 0.28) circle[radius=0.065];
	  \draw (0.0, 0.28) -- (-0.175, 0.595);
	  \fill (-0.175, 0.595) circle[radius=0.065];
	  \draw (0.0, 0.28) -- (0.175, 0.595);
	  \fill (0.175, 0.595) circle[radius=0.065];
	  \draw (0.175, 0.595);
	}
  }
  
  \begin{center}\begin{tikzpicture}[line width=0.25mm]
    \draw (0.0, 0.0) pic {mytree};
  \end{tikzpicture}\end{center}
  belonging to $X_1$.  The fibre of $d_1:X_2\to X_1$ over $T$ is the (discrete) groupoid of
  all possible cuts in this tree:
  
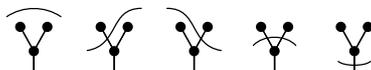
\begin{center}\begin{tikzpicture}[line width=0.25mm]
    \draw (0.0, 0.0) pic {mytree};
	\begin{scope}[shift={(0.0, 0.735)},line width=0.15mm]
	  \draw (-0.35, 0.0) .. controls (-0.14, 0.14) and (0.14, 0.14) .. (0.35, 0.0);
	\end{scope}
    \draw (1.05, 0.0) pic {mytree};
	\begin{scope}[shift={(1.05, 0.35)},line width=0.15mm]
	  \draw (-0.35, -0.07) .. controls (0.07, 0.0) and (-0.07, 0.49) .. (0.35, 0.49);
	\end{scope}
    \draw (2.1, 0.0) pic {mytree};
	\begin{scope}[shift={(2.1, 0.35)},line width=0.15mm]
	  \draw (0.35, -0.07) .. controls (-0.07, 0.0) and (0.07, 0.49) .. (-0.35, 0.49);
	\end{scope}
    \draw (3.15, 0.0) pic {mytree};
	\begin{scope}[shift={(3.15, 0.35)},line width=0.15mm]
	  \draw (-0.28, 0.0) .. controls (-0.14, 0.14) and (0.14, 0.14) .. (0.28, 0.0);
	\end{scope}
    \draw (4.2, 0.0) pic {mytree};
	\begin{scope}[shift={(4.2, 0.14)},line width=0.15mm]
	  \draw (-0.21, 0.0) .. controls (-0.14, -0.07) and (0.14, -0.07) .. (0.21, 0.0);
	\end{scope}
  \end{tikzpicture}\end{center}
  The thing to notice here is that while the second and third cuts are isomorphic
  as abstract cuts, and therefore get identified in $Y_2 = \pi_{0}(X_2)$, this
  isomorphism does not fix the underlying tree $T$.  This means that 
  in the formula for comultiplication of $T$ as an element of $X_1$ both cuts
  appear, and there is a total of $5$ terms, whereas in the formula for 
  comultiplication of $T$ as an element of $Y$ there
  will be only $4$ terms.  (Put in another way, the functor $X \to Y$ given by
  taking components is not \culf.)

  It seems that there is no way to circumvent this discrepancy directly at the
  \isoclass level: attempts involving ingenious decorations by natural numbers and
  actions by symmetric groups will almost certainly end up amounting to actually
  working at the groupoid level, and the conceptual clarity of the groupoid
  approach seems much preferable.
\end{blanko}

\begin{blanko}{Non-commutative versions.}
  The Butcher--Connes--Kreimer Hopf algebra of combinatorial trees admits a
  natural non-commutative version, first studied by Foissy~\cite{Foissy:2002I}.
  It is defined in exactly the same way, but with {\em ordered} forests of {\em
  planar} combinatorial trees.  In this case, the decomposition space is
  monoidal but not symmetric monoidal, giving naturally a non-commutative
  bialgebra.
  
  The same modification can be applied in the operadic case.  Planar operadic
  trees are precisely $M$-trees for $M$ the free-monoid monad.  More generally,
  to have planar structure on $P$-trees is to have a cartesian natural 
  transformation $P \Rightarrow M$ (see \cite{Gambino-Kock:0906.4931} for 
  details); in this situation there is a non-commutative
  bialgebra of ordered forests of $P$-trees.
\end{blanko}

\begin{blanko}{Free multicategories 
  \cite{Gambino-Kock:0906.4931}.}\label{ex:freemonad}
  Continuing the previous example, for any polynomial endofunctor $P$ cartesian
  over $M$, the groupoid of $P$-trees is (essentially) discrete, which is to
  say that it is equivalent to the set of \isoclasses of $P$-trees (because the planar
  structure encoded in the cartesian natural transformation to $M$ fixes the
  automorphisms).  This set is the set of operations of the free monad on $P$
  \cite{Gambino-Kock:0906.4931}, \cite{Kock:0807}.
  Thinking of $P$ as specifying a signature, we can equivalently think of
  $P$-trees as operations for the free (coloured) operad on that signature, or
  as the multi-arrows of the free multicategory on $P$ regarded as a multigraph.
  To a multicategory there is associated a monoidal
  category~\cite{Hermida:repr-mult}, whose object set is the free monoid on the
  set of objects (colours).
    The decomposition space of $P$-trees is naturally identified with the (fat) nerve of
  the (monoidal) category associated to the multicategory
  of $P$-trees.
  (The adjective `fat' is in parenthesis here because it could be omitted:
  the categories involved here have no invertible arrows (other than the 
  identities), because the multicategory is free.)
\end{blanko}

\begin{blanko}{Polynomial monads and operads.}\label{ex:monads}
  The decomposition space of $P$-trees for $P$ a polynomial endofunctor
  (\ref{Ptrees}) can be regarded as the decomposition space associated to the
  free monad on $P$.  In fact the construction works for any (cartesian,
  discrete-finitary) polynomial monad, not just free ones, as we now proceed to
  explain.  This construction has been generalised and subsumed in a more
  comprehensive setting of relative two-sided bar constructions in
  \cite{Kock-Weber:1609.03276}.  Presently we outline, in a more heuristic
  manner, the construction of a monoidal decomposition space from any coloured
  operad, and from it a commutative bialgebra.  For the numerical version, some
  finiteness conditions must be assumed.
  
  Coloured operads can be encoded as polynomial monads~\cite{Weber:1412.7599}.
  The combinatorial data of the endofunctor $R$ underlying the monad is a
  diagram of groupoids
  $$
  I \leftarrow E \to B \to I
  $$
  where $I$ is the set (or more generally, groupoid) of colours, $B$ is the
  groupoid of operations (more precisely the action groupoid of the action of
  the symmetric groups on the operations), and $E$ is the groupoid of operations
  with a marked input slot.  It follows that $E \to B$ is a finite map; the
  fibre over an operation is the set of its input slots.  The operad 
  substitution law then amounts to a cartesian monad structure on $R$, 
  i.e.~cartesian natural 
  transformations $R \circ R\Rightarrow R \Leftarrow \Id$ subject to axioms.

  Following the graphical
  interpretation given in \cite{Kock-Joyal-Batanin-Mascari:0706}, one can regard $I$
  as the groupoid of decorated unit trees (i.e.~trees without nodes), and $B$ as
  the groupoid of corollas (i.e.~trees with exactly one node) decorated with
  $B$ on the node and $I$ on the edges, compatibly.  The arity of a
  corolla labeled by $b\in B$ is then the cardinality of the fibre $E_b$. 
  
  We can now form a simplicial groupoid $X$ in which $X_0$ is the groupoid of
  disjoint unions of decorated unit trees, $X_1$ is the groupoid of disjoint
  unions of decorated corollas, and where more generally $X_n$ is the groupoid of
  $R$-forests of height $n$.  For example, $X_2$ is the groupoid of $R$-forests of
  height $2$, which equivalently can be described as configurations consisting
  of a disjoint unions of bottom corollas whose leaves are decorated with other
  corollas, in such a way that the roots of the decorating corollas match the
  leaves of the bottom corollas.  This groupoid can more formally be described
  as the free symmetric monoidal category on $R(B)$ (the endofunctor $R$ 
  applied to $B$).  Similarly, $X_n$ is the
  free symmetric monoidal category on $R^{n-1}(B)$.  The outer face maps project
  away the top or bottom layer in a level-$n$ forest.  For example $d_0: X_1
  \to X_0$ sends a disjoint union of corollas to the disjoint union of their
  root edges, while $d_1: X_1 \to X_0$ sends a disjoint union of corollas to the
  forest consisting of all their leaves.  The active face maps (i.e.~inner face
  maps) join two adjacent layers by means of the monad multiplication on $R$.
  The degeneracy maps insert unary corollas by the unit of the monad.
  Associativity of the monad law ensures that this simplicial groupoid is
  actually a category object and a Segal space~\cite{Kock-Weber:1609.03276}.
  The operation of disjoint union makes this a symmetric monoidal
  decomposition space, and altogether an incidence bialgebra results from the
  construction.

  The example (\ref{Ptrees}) of $P$-trees (for $P$ a polynomial endofunctor) and
  admissible cuts is an example of this construction, namely corresponding to
  the free monad on $P$: indeed, the operations of the free monad on $P$ form the
  groupoid of $P$-trees, which now plays the role of $B$.  Level-$n$ trees in
  which each node is decorated by objects in $B$ is the same thing as $P$-trees
  equipped with $n-1$ compatible admissible cuts, and grafting of $P$-trees (as
  prescribed by the active face maps in \ref{Ptrees}) is precisely the monad
  multiplication in the free monad on $P$.  
  
  It should be stressed that while the decomposition space of a free operad is
  always automatically locally finite, the case of a general operad is not
  automatically so.  This condition must be imposed separately if numerical
  examples are to be extracted.
  
  Another subexample of this is the case where the monad is the terminal 
  reduced monad $\mathrm{Comm}$, which is the free-commutative-semimonoid monad.
  In this case, the resulting category object in groupoids is equivalent to the
  fat nerve of the category of surjections (as in \ref{sec:fdb}), so the
  associated bialgebra is the classical \fdb bialgebra.  The main achievement
  of~\cite{Kock-Weber:1609.03276} is to show that the \fdb formula for the
  comultiplication in the classical \fdb bialgebra generalises to incidence
  bialgebras of arbitrary operads and polynomial monads (the free case having
  been established previously in \cite{GalvezCarrillo-Kock-Tonks:1207.6404}).
\end{blanko}

\begin{blanko}{Progressive graphs and free PROPs.}\label{ex:MM}
  The constructions in \ref{Ptrees} readily generalise from trees to
  progressive graphs (although the attractive polynomial
  interpretation does not).  By a progressive graph
  we understand a finite directed graph with a certain number of open
  input edges, a certain number of open output edges, and prohibited to contain
  an oriented cycle (see \cite{Kock:1407.3744} for a categorical formalism).  In
  particular, the set of vertices of a progressive graph has a natural poset
  structure.  The progressive graphs form a groupoid $\mathbf G_1$.  We allow
  graphs without vertices, these form a groupoid $\mathbf G_0$.  Let $\mathbf
  G_2$ denote the groupoid of progressive graphs with an {\em admissible cut}:
  by this we mean a partition of the set of vertices into two disjoint parts, a
  down-set $V_1$ and an up-set $V_2$.  This partition determines a set of edges,
  called the {\em cut}, consisting of the edges that connect a vertex in $V_1$
  with a vertex in $V_2$, the out-edges of $V_0$, the in-edges of $V_2$, and the
  edges of $G$ that are both in-edges and out-edges. The two vertex sets $V_1$ 
  and $V_2$ induce new progressive graphs $G|V_1$ and $G|V_2$, by including all
  edges incident to the given vertex set, and including in both cases also the whole cut set,
  which becomes the
  new set of output edges for $G| V_1$ and the new set of input edges for $G|
  V_2$.  Similarly, let
  $\mathbf G_k$ denote the groupoid of progressive graphs with $k-1$ compatible
  admissible cuts, just like we did for forests.  It is clear that this defines
  a simplicial groupoid $\mathbf G$, easily verified to be a decomposition space
  and in fact a Segal space.  The progressive graphs form the set of operations
  of the free PROP with one generator in each input/output degree $(m,n)$.
  Decorating data for progressive graphs are called tensor schemes in
  \cite{Joyal-Street:tensor-calculus}, and the progressive graphs decorated by a
  tensor scheme form the set of operations of the free (coloured) PROP on the
  tensor scheme. The same construction is important for the operational
  semantics of Petri nets~\cite{Kock:2005.05108}.
  In fact, the construction works for any PROP, not just free ones, in
  analogy with the passage from trees and free operads to arbitrary operads 
  (\ref{ex:monads}).  Note that disjoint union (or the monoidal structure underlying any 
  PROP) makes the resulting incidence coalgebras into bialgebras.
  
  Bialgebras of
  progressive graphs have been studied in the context of Quantum Field Theory by
  Manchon~\cite{Manchon:MR2921530}.  Certain decorated progressive graphs, and
  the resulting bialgebra have been studied by Manin~\cite{Manin:MR2562767},
  \cite{Manin:0904.4921} in the theory of computation: his graphs are decorated
  by operations on partial recursive functions and switches.
\end{blanko}

\begin{blanko}{Hereditary species and directed hereditary species.}
  In incidence coalgebras of categories or (directed) restriction species
  the left and the right tensor factor of the comultipication are on equal
  footing. As we have just seen, incidence bialgebras of operads of 
  various kinds have a monomial in the left-hand tensor factor and linear 
  terms in the right-hand tensor factor, reflecting the many-to-one nature 
  of operads. There is a large class of combinatorial bialgebras with this
  feature which do not come from operads: they are bialgebras coming from
  Schmitt's hereditary species, from \cite{Schmitt:hacs}: just like
  ordinary species are presheaves on $\B$ and restriction species are
  presheaves on $\I$, hereditary species are functors on the category of
  finite sets and partial surjections, that is spans where the backward leg
  is injective and the forward leg is surjective. We are thus talking about
  structures that can be restricted along injections and pushed forth along
  surjections. 

  A prototypical example is the hereditary species of graphs: the
  comultiplication of a graph is given by summing over partitions of the
  vertex sets; then on each block there is an induced graph structure (as in
  restriction species), and these are formally multiplied together to form
  a monomial placed in the left-hand tensor factor. The right-hand tensor
  factor receives the graph structure induced on the quotient, that is, the
  set of blocks. Carlier~\cite{Carlier:1903.07964} showed how a hereditary
  species $H$ induces a monoidal decomposition space $\ds H$ whose
  incidence bialgebra is the one constructed by
  Schmitt~\cite{Schmitt:hacs}. The construction is similar to the two-sided
  bar construction of an operad: the decomposition space $\ds H$ has in
  degree $1$ the groupoid of (families of) $H$-structures, and in degree
  $2$ the groupoid of (families of) $H$-structures with a partition on the
  underlying set. The middle face map $d_1$ forgets the partition; the
  bottom face map $d_0$ returns the $H$-structure induced on the quotient
  set, and the top face map $d_2$ returns the family of $H$-structures
  given by the blocks. Just as for operads, it is this nature of the top
  face maps that necessitates working with families of structures rather
  than single structures.

  Cebrian and Forero~\cite{Cebrian-Forero:2211.07721} gave a directed
  version of this construction, starting with a new notion of directed
  hereditary species, which is a functor on the category of posets and
  partial monotone contractions. Just as the case of graphs is the
  paradigmatic example of hereditary species, posets and trees form
  canonical examples of directed hereditary species. In particular, there
  are monoidal decomposition spaces of finite posets  (giving a certain 
  bialgebra of finite topologies first studied by 
  Fauvet--Foissy--Manchon~\cite{Fauvet-Foissy-Manchon:1503.03820}), 
  trees (giving the 
  Calaque--Ebrahimi-Fard--Manchon bialgebra of 
  trees~\cite{Calaque-EbrahimiFard-Manchon:0806.2238}), or linear trees 
  (giving a version of the Fa\`a di Bruno bialgebra as in Subsection~\ref{sec:fdb}, 
  with the alternative
  normalisation mentioned in \ref{fdb-sectioncoeffs}, in turn a 
  symmetrised version of the example in \ref{ex:OS}).

  Both for hereditary species and directed hereditary species, an important
  feature is that there is an 
  ordinary restriction species present at the same time, and
  in this way there are {\em two} different comultiplications. It is shown in
  \cite{Carlier:1903.07964} and \cite{Cebrian-Forero:2211.07721} that these
  structures form what is called a {\em comodule bialgebra}: it is about two
  bialgebras with the same underlying algebra but with two
  different comultiplications,
  and such that one comultiplication co-distributes over the other from
  one side. The interest in these structures comes mostly from their 
  appearance in various branches of analysis (see
  Manchon~\cite{Manchon:Abelsymposium}
  and Foissy~\cite{Foissy:1702.05344}).

  Both hereditary species and directed hereditary species have a strong
  operad flavour, since there is one output (the quotient) and many
  inputs (the blocks). However, they are rarely operads; instead they are
  so-called {\em operadic categories} in the sense of Batanin and
  Markl~\cite{Batanin-Markl:1404.3886}. An operadic analogue of the tree
  examples of Cebrian and Forero was given by Kock~\cite{Kock:1912.11320},
  leading to other decomposition-space constructions of comodule bialgebras. The
  difference is essentially the difference between combinatorial trees (as
  in \ref{ex:CK}) and operadic trees (as in \ref{Ptrees}).
\end{blanko}

\subsection{Symmetric functions}
\label{sec:Sym}

The ring $\Sym$ of symmetric functions is the subring of the ring of
power series in countably many variables, $\Sym\subset
\Q[[x_1,x_2,...]]$, consisting of the power
series that are of bounded degree and are invariant under
permutation of the variables (see Stanley~\cite{Stanley:volII}).

Given an integer partition $\lambda \vdash n$ let
$m^\cl_\lambda(\underline x) $ be the symmetric function
given by the
prescription: list all {\em distinct}
permutations of the vector $\lambda$, then apply these as
exponents to all $x_i \ x_j \ x_k\dots$ with
$i<j<k<\dots$, and take the sum of the resulting monomials.
For example, the integer partition $331$ has three permutations,
namely $331,313,133$, and if we restrict to the alphabet
with three variables, we obtain
\begin{equation}\label{eq:classicalm331}
m^\cl_{331} = x_1^3 x_2^3 x_3^1 + x_1^3 x_2^1 x_3^3  + x_1^1 x_2^3 x_3^3 \; .
\end{equation}
The symmetric functions $m^\cl_\lambda(\underline x)$
form an additive basis for $\Sym$, termed the monomial basis.

\begin{blanko}{Towards an objective theory: the groupoid of surjections.}
  We outline some developments of an objective theory
  of symmetric functions. Further details will be forthcoming
  in \cite{GKT:sym,GKT:QSym}.
  As in \ref{surjection-gpd}
  we take the basic object to be the groupoid of
  surjections $\mathbf S_{1}$ rather than the set of integer partitions.
\end{blanko}

The theory of symmetric functions is huge,
and
even setting up the standard bases and base changes is a big task.
In the following we content ourselves to outline the monomial basis and
the elementary basis, and the base change between the two. This already
illustrates the flavour of the objective theory, and reveals some
subtleties.

\begin{blanko}{Symmetric functions in the $M$-basis.}
The decomposition space encoding the comultiplication of symmetric
functions in the monomial basis is given by
$$
\Lambda^M_r := \left\{
\begin{tikzcd}[cramped]
\underline n \ar[d, onto,"\lambda"] &   \\
\underline k \ar[r,"\ell"] & r
\end{tikzcd}
\right\}
$$
So $\Lambda^{M}_0=1$ as only the empty surjection
splits into 0 parts,
and $\Lambda^M_1$ is the groupoid of surjections $\lambda$,
while $\Lambda^M_2$ is the
groupoid of surjections $\lambda$ with a splitting or \emph{layering} $\ell$
of the codomain into two parts,
which we can picture more explicitly as
\begin{equation}\label{eq:sum-split-m}
\qquad\qquad\begin{tikzcd}[column sep=large]
  \underline n' \ar[dashed,r] \ar[dashed,d, onto,"{\lambda'=d_0(\lambda,\ell)}"']
  & \underline n \ar[d,
  onto,"{d_1(\lambda,\ell)}"',"\lambda"] & \underline n'' \ar[dashed,l]
  \ar[dashed,d, onto,"{\lambda''=d_2(\lambda,\ell)}"]  &
  \text{\footnotesize{sum splitting \phantom{$\ell$}}} \\
\underline k' \ar[r] & \underline k & \underline k'' \ar[l]
& \text{\footnotesize{sum splitting $\ell$}}
\end{tikzcd}
\end{equation}
The bottom row is a sum-splitting diagram, meaning a pair of injective 
maps whose images are disjoint and that is jointly surjective.
Note that then the top row is a sum splitting too,
and that the data of the top row and the surjections $\lambda',
\lambda''$ are implied (taking preimages) by $\lambda$ and the bottom row.

Let $M_\lambda$ denote the cardinality of
$\name{\lambda} : 1 \to \Lambda^M_1$.
The simplicial groupoid $\Lambda^M$ is a
locally discrete decomposition space in the sense of \ref{loc-disc-ds}
and we have
$$
\Delta(M_\lambda) = \sum_{\underline k' + \underline k'' = \underline k}
M_{\lambda'} \otimes M_{\lambda''}
$$
where the summation is over all splittings
\eqref{eq:sum-split-m}.
\end{blanko}

\begin{blanko}{Example comultiplication.} In the comultiplication of $M_{331}$
  there are $8$ terms, correspondning to the $2^3=8$ splittings of the codomain.
We get
$$
\Delta(M_{331}) = M_{331} \tensor 1 + M_{33} \tensor M_1 + 2 \, M_{31}
\tensor M_3 + 2\, M_3 \tensor M_{31} + M_1 \tensor M_{33} + 1 \tensor
M_{331} .
$$
\end{blanko}

\begin{blanko}{Comparison with classical normalisation.}
Experts will notice here that this comultiplication is different from the comultiplication
of the usual monomial symmetric functions $m^\cl_\lambda$, where
$$
\Delta(m^\cl_{331}) =
m^\cl_{331} \tensor 1 + m^\cl_{33} \tensor m^\cl_1 +  m^\cl_{31}
\tensor m^\cl_3 +  m^\cl_3 \tensor m^\cl_{31} + m^\cl_1
\tensor m^\cl_{33} + 1 \tensor
m^\cl_{331} .
$$
In fact we are dealing with a different
normalisation, which was first studied by Doubilet~\cite{Doubilet:1972}.
Precisely, the relation with the classical monomial symmetric functions is
given by
$$
M_\lambda = ( \lambda_1! \lambda_2! \lambda_3!
\cdots ) \;
 m^\cl_\lambda
 $$
 if the type of the surjection is
 $1^{\lambda_1}2^{\lambda_2}\dots$ as in \eqref{eq:type}.
\end{blanko}

\begin{blanko}{Polynomial semantics.}
  Let $A$ be a set, here playing the role of an alphabet.
  A {\em polynomial functor} in $A$-many variables is a diagram
  $$
  A \stackrel s \leftarrow T \to B \to 1
  $$
  which defines a functor
  \begin{eqnarray*}
     \grpd_{/A} & \longrightarrow & \grpd  \\
    X{\to}A \ & \longmapsto & \sum_{b\in B} \prod_{e\in T_b} X_{se}\;\;=\;\;
                              \sum_{b\in B}\prod_{a\in A}(X_a)^{T_{a,b}}
                              .
  \end{eqnarray*}
  Note that with $B$ infinite, it is more like a power series.

  To the monomial symmetric function $M_\lambda$ we assign the polynomial
  functor defined by the diagram of sets
  $$
  A   \leftarrow T_\lambda \to B_\lambda \to 1
  $$
  where $A$ is our alphabet which could be finite or could be
  $\N$, and
  $T_\lambda$ and $B_\lambda$ are the sets of possible diagrams
$$
T_\lambda=\left\{\begin{tikzcd}[cramped]
  1 \ar[r] & \un n \ar[d, onto, "\lambda"] &
  \\
  & \un k \ar[r, into] & A
  \end{tikzcd}\right\}\,,
\qquad
B_\lambda=\left\{\begin{tikzcd}[cramped]
  \un n \ar[d, onto, "\lambda"] &
  \\
  \un k \ar[r, into] & A
  \end{tikzcd}
\right\}\,.
  $$
  The map from $T_\lambda$ to $A$ returns the element singled out
  by the composite, and the map to $B_\lambda$
  forgets which element of $\un n$ was singled out.

  In other words, the polynomial associated to $\lambda$
  consists of one monomial for
  each monomorphism $\un k \into A$,
  formed using the fibres
  $n_i$ (for $i\in k$) as the exponents,
  $$
M_\lambda\longmapsto \sum_{j:k\into A}x_{j_1}^{n_1}\dots x_{j_k}^{n_k}
  $$
\end{blanko}

\begin{blanko}{Example.}
  Consider $M_{331}$
    and take the alphabet
$A=\{x_1,x_2,x_3\}$.
In contrast to the recipe for classical symmetric polynomials
in \eqref{eq:classicalm331}
we keep the exponents vector fixed, and then
let the variables come in all orders, so as to get
\begin{align*}
M_{331} &= x_1^3 x_2^3 x_3^1 + x_1^3 x_3^3 x_2^1  + x_2^3 x_1^3 x_3^1
+ x_2^3 x_2^3 x_1^1 + x_3^3 x_1^3 x_2^1  + x_3^3 x_2^3 x_1^1
.\end{align*}
This is consistent with the normalisation
$M_{331}\!= \!2\, m^\cl_{331}$,
where the factor $2\!=1!\,0!\,2!$
appears since the surjection has type $1^{1}2^{0}3^{2}$.
\end{blanko}

\begin{blanko}{Elementary symmetric functions.}
The classical elementary basis for $\Sym$ is given by defining
$$
e_n^\cl:=m^\cl_{11\dots1}
$$
and then defining
$$
e^\cl_\lambda := e_{n_1}^\cl \cdots e_{n_k}^\cl.
$$

In the objective theory we define $E_\lambda$ directly in terms of
surjections, with comultiplication given by splitting the domain.
There is a locally discrete decomposition space $\Lambda^E$ with
$$
\Lambda^E_r := \left\{
\begin{tikzcd}[cramped]
\underline n \ar[d, onto] \ar[r] & r  \\
\underline k  &
\end{tikzcd}
\right\}.
$$
Thus $\Lambda^E_0=1$ and $\Lambda^E_1$ is the
groupoid of surjections as before, while now $\Lambda^E_2$ is the
groupoid of surjections with a splitting of the domain into two parts,
which we can picture more explicitly as
$$
\begin{tikzcd}
\underline n' \ar[r] \ar[d, onto] & \underline n \ar[d,
onto] & \underline n'' \ar[l] \ar[d, onto]  & \text{\footnotesize{sum splitting}}\\
\underline k' \ar[r] & \underline k & \underline k'' \ar[l] &
\text{\footnotesize{cover}\phantom{xxxxxx}}
\end{tikzcd}
$$
where the top row is a sum-splitting diagram.
The bottom row then consists of two injections that jointly cover $k$.
\end{blanko}

\begin{blanko}{Example.}
  $\Delta(E_{21})$ has $8$ terms, since there are $8$ ways to split the
  domain. If the split is inside a block, that block splits into two.
  $$
  \Delta(E_{21}) =
  1 \tensor E_{21}
  + 2 E_{11}\tensor E_1
  + E_1 \tensor E_2
  + 2 E_1 \tensor E_{11}
  + E_2 \tensor E_1
  + E_{21} \tensor 1  .
  $$
  Experts in symmetric functions will notice this differs from
  the comultiplication of
  classical elementary symmetric functions, where
  $$\Delta(e^\cl_{21}) =
  1\tensor e^\cl_{21}
  + e^\cl_{11}\tensor e^\cl_1
  + e^\cl_1 \tensor e^\cl_2
  + e^\cl_{1}\tensor e^\cl_{11}
  + e^\cl_2 \tensor e^\cl_1
  + e^\cl_{21}  \tensor 1 .
  $$
\end{blanko}

\begin{blanko}{Another normalisation.}
  For the elementary symmetric functions, as for the monomial
  basis, there is a scalar factor relating
  the classical and our objective version.
The
relationship here says $E_2 = 2\, e^\cl_2$ and
 $E_{21} = 2\, e^\cl_{21}$, and in general
$$
E_\lambda \;=\; n_1! \cdots n_k! \; e^\cl_\lambda.
$$
Again, the reason is the polynomial semantics: to $E_\lambda$ represented
by a surjection $\lambda:\un n \onto \un k$ we assign the polynomial functor $A
\leftarrow T \to B\to 1$, where $B$ is the set of diagrams
$$\begin{tikzcd}
  \un n \ar[r, "\text{loc.inj.}"] \ar[d, twoheadrightarrow] & A \\
  \un k &
\end{tikzcd}$$
The map $\un n\to A$ is required to be
locally injective, which means that it is injective on each fibre.

For example, if the alphabet  $A$ has 3 elements, then
$e_{21}^\cl=m_{11}^\cl m_{1}^\cl$ is a sum of
$3\times 3$ monomials, while
$E_{2,1}$ has $18$ terms: $6$ choices of an
injective map $\un 2\into A$ and $3$ of an injective map
$\un 1 \into A$.

This normalisation was perhaps first studied by
Doubilet~\cite{Doubilet:1972}.
\end{blanko}

\begin{blanko}{Change of basis: elementary symmetric functions in terms of
  monomial symmetric functions.}
  One can write every elementary symmetric function as a sum
of monomial symmetric functions
\begin{equation}\label{e-to-m-cl}
e_\lambda = \sum_{\lambda\wedge\pi =\id} m_\pi
\end{equation}
Traditionally the summation is over $0/1$ matrices whose row and column totals
give the partitions $\lambda$ and $\pi$.
There are several ways to rephrase this in a more objective manner.
For example, consider the sum over all injections $\un n\into \un k\times \un b$
whose components are surjections $\lambda:\un n \onto \un k$ and $\pi:\un n\onto \un b$.
Note that two surjections $\lambda$ and $\pi$
with common domain are jointly injective if and only if
they are \emph{transversal partitions}, that is,
their meet $\lambda\wedge\pi=\id$.
Alternatively, an injection $\un n\into \un k\times \un b$ can be regarded as a relation
between $\un k$ and $\un b$, and is termed an \emph{effective relation}
when its components are surjective.

The change of basis \eqref{e-to-m-cl} is best expressed
as a coalgebra homomorphism, and
at the level of decomposition spaces this means that we describe a
\ikeo-\culf span of decomposition spaces
$$
\Lambda^E \stackrel{\text{\tiny{ikeo}}}\longleftarrow W
\stackrel{\text{\tiny{culf}}}\longrightarrow \Lambda^M
$$
Recall that \ikeo maps are those that induce coalgebra homomorphisms
contravariantly, while \culf maps induce coalgebra homomorphisms covariantly.

In this case, the middle decomposition space $W$ has $W_1$ the groupoid of
effective relations, or transversal partitions,
$$
\begin{tikzcd}
  & \un n \ar[ld, onto, "\lambda"'] \ar[rd, onto, "\pi"] \ar[d, into] & \\
  \un k & \un k\times \un b  \ar[r] \ar[l]& \un b
\end{tikzcd}
$$
Given two partitions $\lambda$ and $\pi$, fitting them into an effective
relation is a yes/no question: it is the question whether the resulting map
$(\lambda,\pi) :\un  n \to \un k\times \un b$ is injective or not.

The span
$$
(\Lambda^E)_1 \leftarrow W_1 \rightarrow (\Lambda^M)_1
$$
clearly implements the map we aim for:
\begin{equation}\label{EtoM}
E_\lambda \mapsto \sum_{\lambda\wedge \pi = \id} M_\pi
\end{equation}
There is now the obvious question whether this assignment is
comultiplicative. It will be so if we can fit $W_1$ into a simplicial
object (necessarily a decomposition space) which is \ikeo over $\Lambda^E$
and \culf over $\Lambda^M$. Guessing what such a decomposition space should
be is a question of knowing what it is in degree $1$ (which we do, since
we already know what it should do), and then use the \culf condition and
the \ikeo condition to arrange the rest.
The result in this case is
$$
W_r  = \{ \un k \leftarrow \un n \to \un b \to r \}
$$
where the part $\un k \leftarrow \un n \to \un b$ is an effective relation.

Now there are obvious projection maps constituting simplicial maps to 
$\Lambda^E$ and $\Lambda^M$. It is not difficult to show that 
the simplicial
  map $W \to \Lambda^M$ is \culf, while it requires some work to show
  that the simplicial map $\Lambda^E \leftarrow W$ is \ikeo.
  The span thus defines a coalgebra homomorphism, whose cardinality is
  \eqref{EtoM}.
\end{blanko}

\begin{blanko}{Discussion of non-invertibility of the base change.}
  The span is not invertible! However, after taking cardinality, it becomes
invertible, and the inverse is given by a formula involving M\"obius
inversion in the lattice of partitions (cf.~Doubilet~\cite{Doubilet:1972}).
The upshot is that at the objective level, we can see more detail: there
is not a single decomposition space of symmetric functions, but rather one
for each basis (all of them having the groupoid of surjections in
simplicial degree $1$). These are not isomorphic,
because while there is maybe an \ikeo-\culf span in one direction to express
a change of coordinates, the inverse change of coordinates will involve
minus signs (coming from some instance of \M inversion), and it will
not exist at the objective level. (It will exist in the same way as
\M inversion exists: only in virtue of formal differences,
cf.~Section~\ref{sec:M} below.)
\end{blanko}

\begin{blanko}{Bialgebra structure.}
  More than the comultiplication of symmetric functions, there is of course
  the multiplication. At the objective level, this should be encoded by
  certain bisimplicial spaces $B : 
  (\simplexcategory\times\simplexcategory)\op \to \Grpd$ that are
  decomposition spaces in each direction, and with a certain compatibility
  condition. The idea is then that the bialgebra is $\grpd_{/B_{11}}$
  with the horizontal direction used to
  define the comultiplication and the vertical direction used to
  define the multiplication. The compatibility condition expresses the
    multiplicativity of the comultiplication (or equivalently the
  comultiplicativity of the   multiplication.) In practice, such bisimplicial
  spaces often arise as \culf-monoidal decomposition spaces, meaning that
  one of the two directions is actually a Segal monoid (rather than a more
  general decomposition space). We already saw several examples of this.
\end{blanko}

\begin{blanko}{Bialgebra structure (\culf monoidal structure) on $\Lambda^E$.}
  The two decomposition spaces of symmetric functions we have described,
  $\Lambda^M$ and $\Lambda^E$, fit into bisimplicial groupoids like this.
  In fact, they are the same bisimpicial groupoid, only read in two
  different directions. This is actually a \culf-monoidal decomposition
  space, most easily described from the viewpoint of $\Lambda^E$: this 
  decomposition space is
  \culf-monoidal under disjoint union of surjections. This means that it is
  very easy to multiply in the $E$-basis:
  $$
  E_{\lambda'} \cdot E_{\lambda''} = E_{\lambda'\sqcup \lambda''}
  $$
  where $\lambda'\sqcup \lambda''$ is the disjoint union $\un n' + \un n'' \onto
    \un k' + \un k''$.

  If we write out the monoidal nerve of this monoidal operation in the way we
  did with $\B$ in \ref{ex:I=DecB}, it is precisely a question of
  splitting the codomain. As a bisimplicial space $\Lambda^E_{\bullet\bullet}$,
  we have $\Lambda^E_{ij}$ the groupoid of surjections with the codomain split
  into $i$ parts and the domain split into $j$ parts. There are no compatibility
  requirements on the splittings of domain and codomain (and this independence
  is essentially the statement that the monoidal structure is \culf). So for
  $\Lambda^E$, the multiplication is simply given by a monoidal structure.
\end{blanko}

\begin{blanko}{Bisimplicial groupoid for $\Lambda^M$.}
  For $\Lambda^M$, in contrast, we have the same bisimplicial
  groupoid, but transposed, so that we use the monoidal direction to define the
  comultiplication, and use the `decomposition-space' direction for the
  multiplication. This means that the formula for multiplication in the
  $M$-basis is less straightforward. Given $\lambda': \un n' \onto \un k'$ and
  $\lambda'': \un n'' \onto \un k''$, we need to find all ways to fit it into a
  diagram
  $$
  \begin{tikzcd}
  \underline n' \ar[r, dotted] \ar[d, onto] & \underline n \ar[d,
  onto] & \underline n'' \ar[l, dotted] \ar[d, onto]  \\
  \underline k' \ar[r, dotted] & \underline k & \underline k'' \ar[l, dotted]
  \end{tikzcd}
  $$
  such that the top row is a sum-inclusion diagram. This forces $n=n'+n''$.
  We can list all such diagrams by first writing the sum $\lambda' \sqcup
  \lambda'' :\un  n' +\un  n'' \onto \un k' + \un k''$, and then postcomposing with all
  quotient maps $q:\un k'+\un k'' \onto \un k$
  whose components are both injective.
\end{blanko}

\begin{blanko}{Example.}
  For any positive numbers $a,b,c,d,e$ we have
  \begin{align*}
 & \!\!\! M_{a,b,c} * M_{d,e} \;\;\;=\;\;\;  M_{a,b,c,d,e} \\
  &{} + M_{a+d,b,c,e} + M_{a+e,b,c,d} + M_{b+d,a,b,e} + M_{b+e,a,c,d}
  + M_{c+d,a,b,e} + M_{c+e,a,b,d}
  \\
  &{} + M_{a,b+d,c+e} + M_{a,b+e,c+d}
  + M_{b,a+d,c+e} + M_{b,a+e,c+d}
  + M_{c,b+d,a+e} + M_{c,b+e,a+d}
  \end{align*}
    To see this, we have to look at quotient maps $q$ in
  $$\begin{tikzcd}
  \un 3 \ar[r, into] \ar[rd, into] & \un 3+\un 2 \ar[d, onto, "q"] & \un 2 \ar[l,intoL]
  \ar[ld,intoL] \\
  & \un k &
  \end{tikzcd}$$
  We can choose to join nothing; this is possible in $1$ way.
  Then we can choose to join one point in $\un 3$ with one point in $\un 2$.
  This can be done in $6$ ways.
  Finally, we can choose to join two points in $\un 3$ with two points in $\un 2$,
  and this can also be done in $6$ ways.
\end{blanko}

\begin{blanko}{Multiplicativity of the change of basis.}
  The change of basis expressed by the \ikeo-\culf span is not only
  comultiplicative but also multiplicative. This means that the
  decomposition space $W$ is also just one direction of a bisimplicial
  groupoid, with bisimplicial maps to the bisimplicial groupoids
  $\Lambda^E_{\bullet\bullet}$ and $\Lambda^M_{\bullet\bullet}$.

  This span of bisimplicial groupoids
  $$
  \Lambda^E_{\bullet\bullet} \leftarrow W \rightarrow \Lambda^M_{\bullet\bullet}
  $$
  is row-wise \ikeo-\culf, as we already saw, while column-wise it is instead
  \culf-\ikeo. This is precisely to say that the corresponding homomorphism on
  incidence coalgebras is furthermore a homomorphism of bialgebras.
\end{blanko}

\begin{blanko}{Outlook.}\label{QSym}
  We have briefly outlined two (double) decomposition spaces, one
  corresponding to the monomial basis and one corresponding to the
  elementary basis, and explained how to express elementary symmetric
  functions in terms of monomial ones by way of an intermediate double
  decomposition space and \ikeo-\culf spans. While this gives a hint at
  what the decomposition-space approach to symmetric functions looks like,
  we have only scratched the surface here. The other combinatorial bases
  (complete homogeneous, power-sum, and the forgotten basis) admit similar
  decomposition-space interpretations (with similar non-standard
  normalisations). Unfortunately, the Schur basis, which is perhaps the
  most interesting basis, for its many connections to representation theory
  and geometry, is not easy to describe as a decomposition space.

  Next, related Hopf algebras, such as that of quasi-symmetric functions,
  noncommutative symmetric functions, free quasi-symmetric functions (the
  Malvenuto--Reutenauer Hopf algebra), word quasi-symmetric functions, also
  admit decomposition space interpretations~\cite{GKT:QSym}. The
  decomposition space for the Hopf algebra of
  quasi-symmetric functions as well as the word quasi-symmetric functions
  have received some interest recently~\cite{Hackney-Kock:2210.11192}, as
  they are examples of free decomposition spaces
  (meaning that they are left Kan extended from
  $\simplexcategory_{\text{inert}}$, as briefly touched upon in 
  \ref{ex:cULF/N} below). The bialgebra of quasi-symmetric
  functions is the terminal object in the category of graded connected
  bialgebras with a zeta function, by a theorem of
  Aguiar--Bergeron--Sottile~\cite{Aguiar-Bergeron-Sottile}.
  An objective version of this result has been
  established recently by Hackney--Kock--Steinebrunner.
\end{blanko}

\section{\M inversion}

\label{sec:M}

\subsection{Completeness, and \M inversion at the objective level}

\label{sec:compl}

We are interested in the invertibility of the zeta functor (see \ref{zeta})
under the convolution product (see \ref{convolution}).  Unfortunately,
at the objective level it can practically {\em never} be convolution invertible,
because the inverse $\mu$ should always be given by an alternating sum
$$
\mu = \Phieven - \Phiodd. 
$$
We do not have minus sign available, but the sign-free equation
$$
\zeta * \Phieven = \epsilon + \zeta * \Phiodd
$$
will hold, as we proceed to recall.
In the category case (cf.~\cite{Content-Lemay-Leroux,LawvereMenniMR2720184}),
$\Phieven$ is given by the even-length
chains of non-identity arrows, that is, by the non-degenerate simplices of even dimension, and similarly for $\Phiodd$.  
To  make sense of this for more general 
decomposition spaces we need to recall, from \cite{GKT:DSIAMI-2},
the notion of completeness.

A simplex in any simplicial groupoid is degenerate when it is in the image of a
degeneracy map.  `Nondegenerate' should mean to be in the complement of the image, but this is only well behaved for
monomorphisms of groupoids,
i.e.~maps that are fully faithful as functors, see \ref{mono}.

\begin{blanko}{Completeness and non-degeneracy.}\label{complete}
 A decomposition space is {\em complete} if $s_0: X_0
  \to X_1$ is mono \cite{GKT:DSIAMI-2}.  It follows that all other degeneracy maps in $X$ are also mono (see \cite{GKT:DSIAMI-1}).

For a complete decomposition space $X$ we define $\nondeg X_n \subset X_n$ to
  be the full subgroupoid of nondegenerate $n$-simplices, i.e.~not in the
  image of any of the degeneracy maps.  More importantly, in a decomposition 
  space
  one can measure whether a simplex is nondegenerate on its principal edges: it
  is nondegenerate if and only if all its principal edges are~\cite[Corollary
  2.14]{GKT:DSIAMI-2}.  Hence it really just boils down to defining
  nondegenerate $1$-simplices: define $\nondeg X_1 \subset X_1$ to be the
  complement of the monomorphism $s_0 : X_0 \to X_1$. 
\end{blanko}

\begin{blanko}{Examples and non-example.}
  Clearly, every discrete decomposition space (such as strict nerves)
  is complete, since any map between
  sets which admits a retraction is a monomorphism.  Also every Rezk-complete 
  Segal space is complete in the sense of \ref{complete}.  In particular, fat
  nerves of categories are complete.
  
  To see an example of a non-complete decomposition space, let $G$ be a
  nontrivial group, and write $BG$ for the same group considered as a groupoid with one object. Now consider the simplicial groupoid $X$ with $X_n = (BG)^n$.
  Here $s_0: 1 \to BG$ is not a monomorphism, as the trivial subgroupoid of $BG$ is not a full subgroupoid.
\end{blanko}

\begin{blanko}{`Phi' functors.}\label{Phi}
  We define $\Phi_n$ to be the linear functor given by the span
  $$
  X_1 \longleftarrow \nondeg{X}_n \longrightarrow  1.
  $$
  If $n=0$ then $\nondeg X_0=X_0$ by convention, and $\Phi_0$ is given by the
  span
  $$
  X_1 \longleftarrow X_0 \longrightarrow  1.
  $$
  That is, $\Phi_0$ is the linear functor $\epsilon$.  Note that
  $\Phi_1=\zeta-\epsilon$, and is denoted $\eta$ in the classical literature
  \cite{Content-Lemay-Leroux,Rota:Moebius}.  The minus sign makes sense here, since $X_0$ and $\nondeg X_1$, representing $\epsilon$ and $\Phi_1$,
   define complementary full subgroupoids of $X_1$, representing
  $\zeta$.
\end{blanko}

Computing convolution with the functors $\Phi_n$
is really about knowing how the groupoids $\nondeg X_n$ behave under various 
pullbacks. 
This is carried out in detail in \cite{GKT:DSIAMI-2}, leading to
the following results.
 \begin{lemma}    \cite[Lemma 3.6]{GKT:DSIAMI-2}
   For a complete decomposition space, we have $$\Phi_n\;\; =\;\; (\Phi_1)^n\;\; =\;\; (\zeta-\epsilon)^n,$$ 
     the $n$th convolution product of $\Phi_1$ with itself.
 \end{lemma}

\begin{prop}
  For a complete decomposition space $X$, the square
$$\xymatrix{
\nondeg{X}_1 + \nondeg{X}_{2} \ar[d]\ar[r] & X_2 \ar[d] \\
X_1 \times \nondeg{X}_1 \ar[r]& X_1\times X_1
}$$
is a pullback.
\end{prop}
These are special cases of 
\cite[Lemma 3.5]{GKT:DSIAMI-2}.
The proposition can be read as saying that if a $2$-simplex $\sigma$ has its second principal 
edge nondegenerate, then there are two possibilities for the first principal 
edge: either it is degenerate and the whole simplex $\sigma$ is determined
by the second principal edge (an element of $\nondeg X_1$), or it is nondegenerate and the whole simplex $\sigma$ is nondegenerate (an element of $\nondeg X_2$).

From this lemma and its higher-dimensional analogues, it is not difficult to 
prove the following key result.

\begin{prop}
    \cite[Proposition 3.7]{GKT:DSIAMI-2}
The linear functors $\Phi_n$ satisfy the following explicit equivalences of 
linear functors
$$
\zeta*\Phi_n
\;\;=\;\;
\Phi_n+\Phi_{n+1}
\;\;=\;\;
\Phi_n*\zeta.
$$
\end{prop}
Now let
$$
\Phieven := \sum_{n \text{ even}} \Phi_n , \qquad
\Phiodd := \sum_{n \text{ odd}} \Phi_n .
$$

\begin{theorem}\label{thm:zetaPhi}
  \cite[Theorem 3.8]{GKT:DSIAMI-2}
  For a complete decomposition space, the following \M inversion 
  principle holds (explicit equivalences of linear functors):
  \begin{align*}
\zeta * \Phieven
 &\;\;=\;\; \epsilon\;\; +\;\; \zeta * \Phiodd,\\
    \Phieven *\zeta &\;\;=\;\; \epsilon \;\;+ \;\; \Phiodd*\zeta.
\end{align*}
\end{theorem}

\begin{proof}
  This follows immediately from the proposition: all four linear functors are in fact
  equivalent to $\sum_{r\geq0}\Phi_r$.
\end{proof}

For these results there is no need for finiteness conditions: there in no
problem in taking infinite sums of groupoids.  In the following subsection,
however, we must impose finiteness conditions before we can take cardinality and
recover \M inversion at the level of vector spaces and (co)algebras over $\Q$.

\subsection{Length and \M decomposition spaces}

\label{sec:length}

If $X$ is a complete and locally finite decomposition space, then by 
Proposition~\ref{finitetypespan}
the linear functors
$$
\Phi_r:\Grpd_{/X_1}\to \Grpd
$$
are finite for each $r\geq0$ and descend to linear functors
$$
\Phi_r:\grpd_{/X_1}\to \grpd.
$$

This is not enough to guarantee finiteness of the sum of all those $\Phi_r$ and
hence allow the \M inversion formula to descend to the vector-space level.  For
this we also need to assume that, for each $f\in X_1$, there is an upper bound on
the dimension of a nondegenerate $n$-simplex with long edge $f$.  This condition
is important in its own right, as it is the condition for the existence of a
length filtration \ref{length}, useful in many applications.  When $X$ is the
nerve of a category, the condition says that for each arrow $f$, there is an
upper bound on the number of non-identity arrows in a sequence
of arrows composing to $f$.  We are led to the following definition.

\begin{blanko}{Length.}\label{length}
  A complete decomposition space $X$ is {\em of locally finite length} if, for
  each $a\in X_1$, the fibres $F^{(n)}_a$ of $d_1^{n-1}:\nondeg X_n\to X_1$ over
  $a$ are empty for $n$ sufficiently large.
\end{blanko}

The {\em length} of $a$ is the greatest $n$ for which
$F^{(n)}_a\neq\varnothing$; this induces a filtration on the incidence
coalgebra.  If $X$ is a Segal space, it is the longest factorisation of $a$ into
nondegenerate $a_i\in \nondeg X_1$.  

\begin{blanko}{Example.}
  The incidence coalgebra of $(\N^2,+)/\mathfrak S_2$ (see \ref{sym})
  is the simplest example 
  we know of in which the length filtration
  does not agree with the coradical 
  filtration (see Sweedler~\cite{Sweedler} for this notion).
  The elements $(1,1)$ and $(2,0)\simeq (0,2)$ are clearly of 
  length $2$.  On the other hand, the element
  $$
  P := (1,1) - (2,0) - (0,2)
  $$
  is primitive, meaning
  $$
  \Delta(P) = (0,0) \tensor P + P \tensor (0,0)
  $$
  and is therefore of coradical filtration degree $1$.  (Note that in 
  $(\N^2,+)$ it is not true that $P$ is primitive:
  it is the symmetrisation that makes the $(0,1)$ terms cancel out in
  the computation, to make $P$ primitive.)
\end{blanko}
  
\begin{blanko}{\M condition.}\label{M}
  A complete decomposition space $X$ is {\em \M} if it is locally finite and of
  locally finite length, that is, for each $a$, $F^{(n)}_a$ is finite and
  eventually empty.
\end{blanko}

  Note that for posets, `locally finite' already implies `locally finite
  length', so the \M condition is not needed separately in the poset case.  
  If $X$ is the strict nerve of a category, then it is \M in our
  sense if and only if it is \M in the sense of Leroux~\cite{Leroux:1975}.

  Classically, it is known that a \M category in the sense of Leroux does not
  have non-identity invertible arrows~\cite[Lemma~2.4]{LawvereMenniMR2720184}.
  Similarly (cf.~\cite[Corollary~8.7]{GKT:DSIAMI-2}), if a \M decomposition
  space $X$ is a Segal space, then it is Rezk complete (meaning that all
  invertible arrows are degenerate, cf.~\ref{Rezk}).

\begin{lemma}\label{lem:oldcharacterisationofM}
  A complete decomposition space $X$ is \M if and only if $X_1$ is locally
  finite and the restricted composition map
  $$
  \sum_r{d_1}^{r-1}:\sum_r \nondeg X_r\to X_1
  $$
  is finite.
\end{lemma}
Thus, if $X$ is \M, the linear functors $\Phieven$ and $\Phiodd$ also descend to
$$
\Phieven,\Phiodd:\grpd_{/X_1}\to \grpd
$$
and their cardinalities are elements $\norm{\Phieven}, \norm{\Phiodd}:\Q_{\pi_0X_1}\to\Q$ of the incidence algebra.
We can therefore take the cardinality of the abstract \M
inversion formula of Theorem \ref{thm:zetaPhi}:

\begin{thm}\label{thm:|M|}
  If $X$ is a \M decomposition space,
  then the cardinality of the zeta functor, $\norm{\zeta}:\Q_{\pi_0 X_1}\to\Q$,
  is convolution invertible with inverse $\norm{\mu}:= \norm{\Phieven} - \norm{\Phiodd}$:
  $$
  \norm\zeta * \norm\mu = \norm\epsilon = \norm\mu * \norm\zeta .
  $$
\end{thm}

\subsection{\M functions and cancellation}

\label{sec:cancellation}

We compute the \M functions in some of our examples.  While the formula $\mu =
\Phieven-\Phiodd$ seems to be the most general and uniform expression of the \M
function, it is often not the most economical.  At the numerical level, it is
typically the case that much more practical expressions for the \M functions can
be computed with different techniques.  The formula $\Phieven-\Phiodd$ should not be
dismissed on these grounds, though: it must be remembered that it constitutes a
natural `bijective' account, valid at the objective level, in contrast to many
of the elegant cancellation-free expressions in the classical theory which are
often the result of formal algebraic manipulations, often power-series 
representations.

Comparison with the economical formulae raises the question whether these too can
be realised at the objective level.  This can be answered (in a few cases) by
exhibiting an explicit cancellation between $\Phieven$ and $\Phiodd$, which in
turn may or may not be given by a {\em natural} bijection.  

Once a more economical expression has been found for some \M decomposition space
$X$, it can be transported back along any \culf functor $f:Y \to X$ to yield also
more economical formulae for $Y$.

\begin{blanko}{Natural numbers.}
  For the decomposition space $\mathbf N$ (see~\ref{ex:N&L}),
  the incidence algebra is $\grpd^{\N}$, with basis given by
  the representables $h^n$, and with convolution product
  $$
  h^a * h^b = h^{a+b}.
  $$

  To compute the \M functor,
  we have
    $$
  \Phieven = \sum_{r \text{ even}} (\N\shortsetminus \{0\})^r ,
  $$
  hence $\Phieven(\un n)$ is
  the set of ordered compositions of the ordered 
  set $\un n$ into an even number of parts, or equivalently
  $$
  \Phieven(\un n) = \{ \un n \onto \un r \mid r \text{ even } \} ,
  $$
  the set of monotone surjections. 
  In conclusion, with an abusive sign notation,
  the \M functor is
  $$
  \mu(\un n) = \sum_{r\geq 0} (-1)^r  \{ \un n \onto \un r \} .
  $$
  
  At the numerical level, this formula simplifies to
  $$
  \mu(n)
  = \sum_{r\geq 0} (-1)^r {n-1 \choose r-1} = \begin{cases}
    1 & \text{ for } n=0 \\
    -1 & \text{ for } n=1 \\
    0 & \text{ else, }
  \end{cases}
  $$
  (remembering that ${-1 \choose -1} =  1$, and ${k \choose -1} = 0$ for $k\geq 
  0$).
  
  \bigskip
  
  On the other hand, since clearly the incidence algebra is isomorphic to the 
  power series ring under the identification $\norm {h^n} = \delta^n 
  \leftrightarrow z^n \in \Q[[z]]$, and since the zeta function corresponds to the 
  geometric 
  series $\sum_n x^n = \frac{1}{1-x}$, we find that the \M function is $1-x$.
  This corresponds to the functor $\delta^0-\delta^1$.

  At the objective level, there is indeed a cancellation of groupoids taking
  place.  It amounts to an equivalence of the Phi-groupoids restricted to
  $n\geq2$:
$$\xymatrix{
\Phieven{}_{\mid r\geq 2} \ar[rd]\ar[rr]^\sim && \Phiodd{}_{\mid r\geq 2} \ar[ld] \\
& \N_{\geq 2} & }
$$
which cancels out most of the terms,
leaving us with
the much more economical \M function
$$
\delta^0 - \delta^1
$$
supported on $\N_{\leq 1}$.
Since $\N$ is discrete, this equivalence (just a bijection)
can be established fibrewise: 
  
  {\em For each $n\geq 2$ there is a natural fibrewise
  bijection}
  $$
  \Phieven(n) \simeq \Phiodd(n) .
  $$
  To see this, encode the elements $(x_1,x_2,\ldots,x_k)$ 
  in $\Phieven(n)$ (and $\Phiodd(n)$) as binary strings 
  of length $n$ and starting with $1$ as follows: each coordinate 
  $x_i$ is represented as a string of length $x_i$ 
  whose first bit is $1$ and whose other bits are 
  $0$, and all these strings are concatenated.
  In other words, thinking of the
  element $(x_1,x_2,\ldots,x_k)$ as a ordered partition of the 
  ordered set $n$, in the binary representation the $1$-entries
  mark the beginning of each part.  
  (The binary strings must start with $1$ since the first part must
  begin at the beginning.)  
  For example, with
  $n=8$, the element $(3,2,1,1,1)\in \Phiodd(8)$, 
  is encoded as the binary string $10010111$.
  Now the bijection between
  $\Phieven(n)$ and $\Phiodd(n)$ can be taken to simply flip the
  second bit in the binary representation.  In the example,
  $10010111$ is sent to $11010111$, meaning that
  $(3,2,1,1,1)\in \Phiodd(8)$ is sent to $(1,2,2,1,1,1)\in 
  \Phieven(8)$.  Because of this cancellation which occurs for
  $n\geq 2$ (we need the second bit in order to flip), the difference
  $\Phieven - \Phiodd$ is the same as $\delta_0 -\delta_1$, which is
  the cancellation-free formula.
  
  The minimal solution $\delta^0-\delta^1$ can also be checked immediately at 
  the objective level to satisfy the defining equation for the \M function:
  $$
  \zeta * \delta^0 = \zeta * \delta^1 + \delta^0
  $$
  This equation says
  $$
  \xymatrix{\N\times \{0\} \ar[d]_{\text{add}} \\ \N}
  =
  \xymatrix{(\N\times \{1\}) + \{0\} \ar[d]_{\text{add}+\text{incl}} \\ \N}
  $$
  
  In conclusion, the classical formula lifts to the objective level.
\end{blanko}

\begin{blanko}{Finite sets and bijections.}
  Already for the next example (\ref{ex:I=DecB}),
  that of the monoidal groupoid $(\B,+,0)$,
  whose incidence algebra is the algebra of species under the Cauchy convolution
  product (cf.~\cite{Aguiar-Mahajan}), the situation is more subtle. 
  
  Similarly to the previous example, we have
  $\Phi_r(S) = \operatorname{Surj}(S, \un r)$, but this time we are 
  dealing with arbitrary surjections, as $S$ is just an abstract set.
  Hence the \M functor is given by
  $$
  \mu(S) = \sum_{r\geq 0} (-1)^r \operatorname{Surj}( S, \un r) .
  $$
  Numerically, 
  the incidence algebra
  is just the power series ring
  $\Q[[z]]$ (cf.~\ref{ex:I=DecB}).
  Since this time the zeta function is the exponential $\exp(z)$, 
  the \M function is the series $\exp(-z)$, corresponding to
  $$
  \mu(n) = (-1)^n .
  $$

The economical \M function suggests the existence of the following
equivalence at the groupoid level:
$$
\mu(S) = \int^r (-1)^r h^r(S) \ \simeq \ \Beven(S) - \Bodd(S) ,
$$
where 
$$\Beven = \sum_{r \text{ even}} \B_{[r]} \quad \text{ and }\quad
\Bodd = \sum_{r \text{ odd}} \B_{[r]}$$
are the full subgroupoids of $\B$ consisting of the even and odd sets,
respectively.  However, it seems that such an equivalence is not possible, at
least not over $\B$: while we are able to exhibit a bijective proof, this
bijection is {\em not} natural, and hence does not assemble into a groupoid
equivalence.
\end{blanko}

\begin{prop}
  For a fixed set $S$, there are monomorphisms $\Beven(S) \into \Phieven(S)$
  and $\Bodd(S) \into \Phiodd(S)$, and a residual bijection
  $$
  \Phieven(S)-\Beven(S)= \Phiodd(S) -\Bodd(S) .
  $$
  This is {\em not} natural in $S$, though, and hence does not constitute
  an isomorphism of species, only an equipotence of species \cite{Bergeron-Labelle-Leroux}.
\end{prop}

\begin{cor}
  For a fixed $S$ there is a bijection
  $$
  \mu(S) \simeq \Beven(S) - \Bodd(S)
  $$
  but it is {\em not} natural in $S$.
\end{cor}

\begin{proof*}{Proof of the Proposition.}
  The map $\Beven \to \B$ is a  monomorphism of groupoids (\ref{mono}), so for each set $S$ of even cardinality
  there is a single element to subtract from $\Phieven(S)$.  The groupoid
  $\Phieven$ has as objects finite sets $S$ equipped with a surjection $S \onto \un k$
  for some even $k$.  If $S$ is itself of even cardinality $n$, then among such
  partitions there are $n!$ possible partitions into $n$ parts.  If there were
  given a total order on $S$, among these $n!$ $n$-block partitions, there is
  one for which the order of $S$ agrees with the order of the $n$ parts.  We
  would like to subtract that one and then establish the required bijection.
  This can be done fibrewise: over a given $n$-element set $S$, we can establish
  the bijection by choosing first a bijection $S \simeq \un n =
  \{1,2,\ldots, n\}$, the totally ordered set with $n$ elements.

  {\em For each $n$, there is an explicit bijection
  $$
  \{ \text{surjections } p: \un n \onto \un k \mid k \text{ even}, p 
  \text{ not the identity map} \}
  $$
$$  \leftrightarrow$$
  $$
    \{ \text{surjections } p: \un n \onto \un k \mid k \text{ odd}, p 
  \text{ not the identity map} \} 
  $$
  }

  Indeed, define first the bijection on the subsets for which $p^{-1}(1)\neq \{1\}$,
  i.e.~the element $1$ is not alone in the first block.  In this case the 
  bijection goes as follows.  If the element $1$ is alone in a block,
  join this block with the previous block.  (There exists a previous block as
  we have excluded the case where $1$ is alone in block $1$.)  If $1$ is not
  alone in a block, separate out $1$ to a block on its own, coming just after
  the original block.  Example
  $$
  (34,1,26,5) \leftrightarrow (134,26,5)
  $$
  For the remaining case, where $1$ is alone in the first block, we just leave 
  it
  alone, and treat the remaining elements inductively, considering now the
  case where the element $2$ is not alone in the second block.  In the end,
  the only case not treated is the case where for each $j$, we have 
  $p^{-1}(j)=\{j\}$, that is, each element is alone in the block with the same
  number.  This is precisely the identity map excluded explicitly in the 
  bijection.  (Note that for each $n$, this case only appears on one of the 
  sides
  of the bijection, as either $n$ is even or $n$ is odd.)
\end{proof*}

In fact, already subtracting the groupoid $\Beven$ from $\Phieven$ 
is not possible naturally.  We would have first to find a monomorphism
$\Beven\into\Phieven$ over $\B$.  But the automorphism group of an
object $\un n \in \B$ is $\mathfrak S_n$, whereas the 
automorphism
group of any overlying object in $\Phieven$ is a proper subgroup of 
$\mathfrak S_n$.  In fact it is the subgroup of those permutations 
that
are compatible with the surjection $\un n \onto \un k$.
So locally the fibration $\Phieven \to \B$ is a group monomorphism,
and hence it cannot have a section.
So in conclusion, we cannot even realise $\Beven$ as a full 
subgroupoid in $\Phieven$, and hence it doesn't make sense to 
subtract it.

\bigskip

One may note that  it is not logically necessary to be able to subtract
the redundancies from $\Phieven$ and $\Phiodd$ in order to find the economical
formula.
It is enough to establish directly (by a separate proof) that the economical 
formula holds, by actually convolving it with the zeta functor.
At the object level the simplified \M function would be the groupoid
$$
\Beven - \Bodd .
$$
We might try to establish directly that
$$
\zeta * \Beven = \zeta * \Bodd + \epsilon .
$$
This should be a groupoid equivalence over $\B$.
But again we can only establish this
fibrewise.  This time, however, rather than exploiting a non-natural total
order, we can get away with a non-natural base-point.
On the left-hand side, the fibre over an $n$-element set $S$, consists
of an arbitrary set and an even set whose disjoint union is $S$.  In other 
words,
it suffices to give an even subset of $S$.  Analogously, on the right-hand
side, it amounts to giving an odd subset of $S$---or in the special case of
$S=\emptyset$, we also have the possibility of giving that set, thanks to the
summand $\epsilon$.  This is possible, non-naturally:

{\em 
For a fixed nonempty set $S$, there is an explicit bijection between even subsets of
  $S$ and odd subsets of $S$.}

  Indeed, fix an element $s\in S$.  The bijection consists of adding $s$ to the subset 
  $U$
  if it does not belong to $U$, and removing it if it already belongs to $U$.
  Clearly this changes the parity of the set.

Again, since the bijection involves the choice of a basepoint, it seems impossible to lift
it to a natural bijection.

\begin{blanko}{Restricting \M formulae along \culf functors.}
  Once a more economical \M function has been found for a decomposition
  space $X$, it can be exploited to yield more economical formulae for any
  decomposition space $Y$ with a \culf functor to $X$. This is the content
  of the following straightforward lemma:
\end{blanko}

\begin{lemma}\label{lem:inheritM}
  Suppose that for the complete decomposition space $X$ we have found
  a \M inversion formula $\mu_{X}=\Psi_{1}-\Psi_{0}$, that is
  $$
  \zeta_{X} * \Psi_0 = \zeta_{X} * \Psi_1 + \epsilon .
  $$
  Then for every decomposition space \culf over $X$, say $f:Y \to X$, we
  have the formula $\mu_{Y}=f\upperstar\Psi_{1}-f\upperstar\Psi_{0}$, that is
  $$
  \zeta_{Y} * f\upperstar \Psi_0 = \zeta_{Y} * f\upperstar \Psi_1 + \epsilon.
  $$
\end{lemma}

\begin{blanko}{Free decomposition spaces.}\label{ex:cULF/N}
  In most of the examples treated, the length filtration~\ref{length} is
  actually a grading. Recall from \cite[6.20]{GKT:DSIAMI-2} that this
  amounts to having a simplicial map $X\to B\N$ to the nerve of $(\N,+)$. In
  the rather special situation when this is \culf, the economical \M
  function formula
  $$
  \mu = \delta^0 - \delta^1
  $$
  for $B\N$
  induces the same formula for the \M functor of $X$.
  This is of course a rather restrictive condition; in fact, 
  for nerves of categories, this happens only for free categories on
  directed graphs (cf.~Street~\cite{Street:categorical-structures}).
  More generally, decomposition spaces admitting a \culf functor to 
  $B\N$ are precisely the {\em free decomposition spaces}~\cite{Hackney-Kock:2210.11192},
  meaning that they are simplicial groupoids obtained by left Kan 
  extension along the inclusion functor 
  $j:\simplexcategory_{\operatorname{inert}} \to \simplexcategory$.
  The formula for $X_1$ is $X_1=\sum_{k\in \N} A_k$.
  The \culf functor for $X=j\lowershriek(A)$ (for some 
  $A:\simplexcategory_{\operatorname{inert}}\op\to\Grpd$) is the simplicial 
  map $j\lowershriek(A) \to j\lowershriek(1)=B\N$ which is always 
  \culf~\cite{Hackney-Kock:2210.11192}.
  The economical \M functor of $B\N$ is now inherited by 
  $X=j\lowershriek(A)$, by Lemma~\ref{lem:inheritM}.
  In detail, there is for each $n\in \N$ a linear span $X_1 \leftarrow A_n 
  \to 1$ denoted
  $\delta^n$ consisting of all the arrows of length $n$, and the 
  economical \M functor is $\delta^0-\delta^1$, which is essentially 
  $A_0-A_1$, with reference to the original 
  $A:\simplexcategory_{\operatorname{inert}}\op\to\Grpd$.

  Many combinatorial coalgebras of deconcatenation type are incidence
  coalgebras of free decomposition spaces. The simplest example is the free
  monoid on a set $S$, i.e.~the monoid of words in the alphabet $S$. The
  economical \M function is then $\delta^0-\delta^1$, where $\delta^1 =
  \sum_{s\in S} \delta^s$. In the power series ring, with a variable $z_s$
  for each letter $s\in S$, it is the series $1-\sum_{s\in S} z_s$. A
  slightly more elaborate example is the decomposition space of
  quasi-symmetric functions (briefly mentioned in \ref{QSym}): it is free
  on $B\nondeg \N :\simplexcategory_{\operatorname{inert}}\op\to\Grpd$, 
  and it follows that we have
  $$
  \mu(p) = \begin{cases} 1  & \text{ for } p:\emptyset\onto \emptyset,\\
  -1 & \text{ for } p:n\onto 1,\\
  0 &  \text{ for $p:n\onto k$ with $k\geq 2$}.
  \end{cases}
  $$
  (Note that the length grading (by codomain of a surjection) is not the
  usual grading of quasi-symmetric functions, which is instead by the
  domain of a surjection.)
\end{blanko}

\begin{blanko}{Decomposition spaces over $\mathbf B$ (\ref{ex:I=DecB}).}
  Similarly, if a decomposition space $X$ admits a \culf functor $\ell : X \to 
  \mathbf B$ (which may be thought of as a `length function with symmetries')
  then at the numerical level  and at the objective level, locally for each 
  object $S\in X_1$, 
  we can pull back the economical \M `functor' $\mu(n) = (-1)^n$
  from $\mathbf B$ to $X$, yielding the numerical  \M function on $X$
  $$
  \mu(f) = (-1)^{\ell(f)} .
  $$
  An example of this is the coalgebra of graphs \ref{ex:graphs-coalg} of 
  Schmitt~\cite{Schmitt:hacs}: the functor from the decomposition space of
  graphs to $\mathbf B$ which to a graph associates its vertex set is \culf.
  Hence the \M function for this decomposition space is
  $$
  \mu(G) = (-1)^{\norm{V(G)} } .
  $$
  In fact this argument  works for any restriction species~\cite{GKT:restriction}.
\end{blanko}

\begin{blanko}{Finite vector spaces.}
  We calculate the \M function in the incidence algebra of the Waldhausen
  decomposition space of $\F_q$-vector spaces, cf.~\ref{ex:q}.
  In this case, $\Phi_r$ is the
  groupoid of strings of $r-1$ nontrivial injections.  The fibre over $V$ is the
  discrete groupoid of strings of $r-1$ nontrivial injections whose last space
  is $V$.  This is precisely the set of nontrivial $r$-flags in $V$, i.e.~flags
  for which the $r$ consecutive codimensions are nonzero.
  In conclusion,
$$
\mu(V) =  \sum_{r=0}^n (-1)^r  \{ \text{ nontrivial $r$-flags in $V$} \} .
$$
(That's in principle a groupoid, but since we have fixed $V$, it is just a 
discrete groupoid: a flag inside a fixed vector space has no automorphisms.)

The number of flags with codimension sequence $p$ is the $q$-multinomial 
coefficient
$$
{ n \choose p_1, p_2, \ldots, p_r }_q .
$$
In conclusion, at the numerical level we find
$$\mu(V) = \mu(n) = \sum_{r=0}^n (-1)^r 
\sum_{\substack {p_1+\cdots + p_r=n \\ p_i > 0}}
{ n \choose p_1, p_2, \ldots, p_r }_q .
$$

On the other hand, it is classical that from the power-series representation
(\ref{ex:q}) one gets the numerical \M function
$$\mu(n) = (-1)^n q^{n\choose 2}.$$
While the equality of these two expressions can easily be established at the
numerical level (for example via a zeta-polynomial argument, cf.~below), we do
not know of an objective interpretation of the expression $\mu(n) = (-1)^n
q^{n\choose 2}$.  Realising the cancellation on the objective level would
require first of all to being able to impose extra structure on $V$ in such a
way that among all nontrivial $r$-flags, there would be $q^{r\choose 2}$ special
ones!
\end{blanko}

\begin{blanko}{\fdb.}
  Recall (from \ref{sec:fdb}) that the incidence bialgebra of the fat nerve of
  the monoidal category of finite sets and surjections is the \fdb bialgebra.
  Since clearly $\zeta$ and $\epsilon$ are multiplicative, also $\mu$ is 
  multiplicative, i.e.~determined by its values on
  the connected surjections.
  The general formula gives
$$
\mu( \un n \onto \un 1) = \sum_{r=0}^n (-1)^n \kat{Tr}(n,r)
$$
where $\kat{Tr}(n,r)$ is the (discrete) groupoid of $n$-leaf $r$-level trees
with no trivial level (in fact, more precisely, strings of $r$ nontrivial
surjections composing to $n \onto 1$), and where the minus sign is abusive 
notation for splitting into even and odd.

On the other hand, classical theory (see Doubilet--Rota--Stanley~\cite{Doubilet-Rota-Stanley})
gives the following `connected 
\M function':
$$
\mu(n) = (-1)^{n-1}(n-1)! .
$$
In conjunction, the two expressions yield
the following combinatorial identity:
$$
(-1)^{n-1}(n-1)! \ = \ \sum_{r=0}^n (-1)^r \norm{ \kat{Tr}(n,r) }.
$$

We do not know how to realise the cancellation at the objective level.
This would require first developing a bit further the theory of
monoidal decomposition spaces and incidence bialgebras, a task
we plan to take up in the near future.
\end{blanko}

\begin{blanko}{Zeta polynomials.}
  For a complete decomposition space $X$, we can 
  classify the $r$-simplices according to their degeneracy type,
  writing
  $$
  X_r = 
  \sum_{k=0}^r {r \choose k} \nondeg X_k ,
  $$
  where the binomial coefficient is an abusive shorthand for that many copies of
  $\nondeg X_k$, embedded disjointly into $X_r$ by specific degeneracy maps (see
  \cite[2.6]{GKT:DSIAMI-2} for details).  Now we fibre over a fixed
  arrow $f\in X_1$, to obtain
  \begin{equation}\label{eq:Xrf}
  (X_r)_f =\sum_{k=0}^\infty {r \choose k} (\nondeg X_k)_f ,
  \end{equation}
  where we have now allowed ourselves to sum to infinity, but for fixed $f$ of 
  finite length it is still a finite sum.
  
  The {\em `zeta polynomial}' of a decomposition space $X$ is the function
  \begin{eqnarray*}
    \zeta^r(f) : X_1 \times \N & \longrightarrow & \Grpd  \\
    (f,r) & \longmapsto & (X_r)_f
  \end{eqnarray*}
  assigning to each arrow $f$ and  $r\in\N$ the $\infty$-groupoid of $r$-simplices with 
  long edge $f$.
  For fixed $f\in X_1$ of finite length $\ell$, this is a polynomial in $r$,
  as witnessed by the expression \eqref{eq:Xrf}.
  In this case, at the numerical level,
  we can substitute $r= -1$ into it to find:
  $$
  \zeta^{-1}(f) = \sum_{k=0}^\infty (-1)^k \Phi_k(f) 
  $$
  Hence $\zeta^{-1}(f) = \mu(f)$, as the notation suggests.
  
  In some cases there is a polynomial formula for $\zeta^r(f)$.
  For example, in the case $X= (\N,+)$ of \ref{ex:N&L} we find
  $\zeta^r(n) = {n+r-1\choose n}$,
  and therefore $\mu(n) = {n-2\choose n}$, in agreement with the other 
  calculations (of this trivial example).
  In the case $X= (\B,+)$ of \ref{ex:I=DecB}, we find $\zeta^r(n) = r^n$,
  and therefore $\mu(n) = (-1)^n$ again.
 
  Sometimes, even when a formula for $\zeta^r(n)$ cannot readily be found, the
  $(-1)$-value can be found by a power-series
  representation argument.  For example in the case of the Waldhausen $S_\bullet$ 
  construction of $\vect$ (\ref{ex:q}), we have that $\zeta^r(n)$ is the set of $r$-flags of $\F_q^n$
  (allowing trivial steps).  We have
  $$
\zeta^r(n) = \sum_{\substack{p_1+\cdots+p_r=n \\ p_i \geq 0}}
\frac{[n]!}{[p_1]!\cdots [p_r]!},
$$
and therefore 
$$
\sum_{n=0}^\infty \zeta^r(n) \frac{z^n}{[n]!} = \left( \sum_{n=0}^\infty 
\frac{z^n}{[n]!}\right)^r ,
$$
Now $\zeta^{-1}(n)$ can be read off as the $n$th coefficient in the inverted
series $\big( \sum_{n=0}^\infty 
\frac{z^n}{[n]!}\big)^{-1}$.  In the case at hand, these coefficients are $(-1)^n q^{n
\choose 2}$, as we already saw.
\end{blanko}

\subsection{Tools for calculation of \M functors}

We mention three high-level tools for calculating \M functions, and wrap 
up with a few open ends.

\begin{blanko}{Carlier's Rota formula for bicomodules.}
  A classical formula of Rota~\cite{Rota:Moebius} compares the \M functions of
  two posets related by a Galois adjunction.
  Carlier~\cite{Carlier:1801.07504} has generalised this to the setting of
  decomposition spaces, where it concerns a notion of adjunction, or more
  generally certain bicomodule configurations: a $X$--$Y$--bicomodule
  configuration is a bisimplicial groupoid $B$ equipped with an
  augmentation column $X$ and an augmentation row $Y$, and such that $B$ is
  Segal in every row and every column and $B$ is furthermore stable, which
  is a pullback condition on top vertical face against top horizontal
  faces, and on bottom vertical faces against bottom horizontal faces.
  Furthermore, the augmentation maps must be \culf. Such a bicomodule
  configuration has an incidence bicomodule over the incidence coalgebras
  of $X$ and $Y$.

  One application of Carlier's Rota formula formula concerns the
  relationship between restriction species and directed restriction
  species, briefly treated in \ref{directedrestriction}. The decomposition
  space $\I$ corresponding to the terminal restriction species has \M
  function $(-1)^n$ for a set with $n$ elements. Since every restriction
  species is \culf over $\I$, a similar formula exists for general
  restriction species. Carlier~\cite{Carlier:1812.09915} sets up a \M
  bicomodule interpolating between $\I$ and the decomposition space $\C$ of
  finite posets (corresponding to the terminal directed restriction
  species), and applies the generalised Rota
  formula~\cite{Carlier:1801.07504} to calculate the \M function of any
  directed restriction species to be $\mu(Q) = (-1)^n$ if the underlying
  poset of $Q$ is discrete with $n$ elements, and zero otherwise. By a \culf
  argument, this also leads to a formula for the \M function of the
  decomposition space of the free monad on a polynomial endofunctor $P$ as
  in \ref{ex:freemonad} (the bialgebra of $P$-trees), namely $\mu(T) =
  (-1)^n$ if the forest $T$ consists of $n$ corollas (and possibly some
  isolated edges), and zero otherwise.
\end{blanko}

\begin{blanko}{Antipodes.}
  The formula for the M\"obius function
  $$
  X_1 \leftarrow \sum_k (-1)^k \nondeg X_k \to 1
  $$
  admits an elegant variation in the case where $X$ is a \culf monoidal 
  decomposition space. In that case one can be more precise on the codomain 
  side of the span, writing instead
  $$
  X_1 \leftarrow \sum_k (-1)^k \nondeg X_k \to X_1
  $$
  where the right-hand map is $\sum_k (-1)^k \nondeg X_k \to \sum_k (-1)^k
  \prod_{i=1}^k \nondeg X_1 \stackrel{\otimes}\to X_1$, returning the
  monoidal product of all the principal edges of a $k$-simplex. This is in
  fact a decomposition-space version (see~\cite{Carlier-Kock:1807.11858})
  of the antipode formula of Takeuchi~\cite{Takeuchi:1971} and
  Schmitt~\cite{Schmitt:antipodes}. It is not a true antipode unless $X_0$
  is contractible (so that the incidence bialgebra becomes connected), but
  even in the case where $X_0$ is not contractible, this mock antipode $S$
  has some merit. For example, it can still be used to calculate the
  M\"obius function as $\zeta \circ S$.

\end{blanko}

\begin{blanko}{Crapo's complementation formula.}
  Another classical tool for calculating M\"obius functions is Crapo's
  complementation formula, originally established for lattices~\cite{Crapo}
  but generalised to general finite posets by Bj\"orner and
  Walker~\cite{Bjoerner-Walker}. It concerns the situation where a poset
  $X$ (resp.~a decomposition space) has a convex subposet
  (resp.~sub-decomposition space) $C$. It reads
  $$
  \mu^X = \mu^{X\shortsetminus C} + \mu^X*\zeta^C * \mu^X  
  $$
  (with self-explanatory notation). This has been established recently \cite{gkt:crapo} at
  the objective level for any M\"obius decomposition space $X$, and
  provides in particular a bijective proof for the Bj\"orner--Walker
  theorem. It should be mentioned that it is more difficult to find
  interesting applications of this formula for decomposition spaces: at the
  moment, the only known applications are the original poset applications
  of \cite{Bjoerner-Walker}.

  In the context of the present discussions, the formula is more
  interesting for the questions it raises. For $C=X$, the formula reads
  $\mu = \mu * \zeta * \mu$. Numerically this is immediate from the fact
  that $\mu$ is convolution inverse to $\zeta$. But objectively there are
  {\em two} different cancellations establishing that easy fact: one from
  $(\mu * \zeta) * \mu$ and one from $\mu * (\zeta * \mu)$. What is more
  interesting is that the Crapo formula gives yet another two different
  cancellations. Altogether it is pressing to get to learn more about the
  structure of cancellations in general.
\end{blanko}

\begin{blanko}{Product formula?}
  Classically, a very useful tool for calculating \M functions, is the
  product formula: it states, for two locally finite posets $P$ and $Q$,
  $$
  \mu_{P\times} = \mu_P \times \mu_Q .
  $$
  For example, it calculates the classical \M function from number theory,
  which is the \M function of the divisibility poset $(\N^\times, \, | \, 
  )$.
  Since this is the infinite (but finitely supported) product of copies of
  the poset $(\N,\leq)$, one for each prime, the classical formula
  $$
  \mu(n) = \begin{cases} 
  (-1)^r & \text{if  $n$ is the product of $r$ 
  distinct primes} \\
  0 & \text{else}
  \end{cases}
  $$
  Unfortunately, it is not easy to derive any such formula at the 
  objective level, since there is no easy description of the nondegenerate 
  simplices of a product of decomposition spaces in terms of the 
  nondegenerate simplices of the factors. Elaborate cancellations seem to
  be required, and at the moment it is not known how to handle this.
\end{blanko}

We finish with a kind of non-example which raises further interesting
questions.
\begin{eks}
  Consider the strict nerve of the category
  $$
  \xymatrix {
  x \ar@(lu,ld)[]_e  \ar@/_0.6pc/[r]_r  & y \ar@/_0.6pc/[l]_s
  }
  $$
  in which $r\circ s = \id_y$, $s\circ r = e$
  and $e\circ e = e$.
  This decomposition space $X$ is clearly locally finite, so it defines
  a vector-space coalgebra, in fact a finite-dimensional one.
  One can check by linear algebra (see 
  Leinster~\cite[Ex.6.2]{Leinster:1201.0413}),
  that this coalgebra has \M inversion.
  On the other hand, $X$ is not of locally finite length,
  because the identity arrow $\id_y$ can be written as an 
  arbitrary long string $\id_y = r\circ s \circ \cdots \circ r\circ s$.
  In particular $X$ is not a \M decomposition space.
  So we are in the following embarrassing situation: on the objective level,
  $X$ has \M inversion (as it is complete), but the formula does not
  have a cardinality.  At the same time, at the numerical level \M inversion 
  exists nevertheless.  Since inverses are unique if they exist, it is
  therefore likely that the infinite \M inversion formula of the objective
  level admits some drastic cancellation at this level, yielding a finite
  formula, whose cardinality is the numerical formula.  Unfortunately, so far
  we have not been able to pinpoint such a cancellation.
\end{eks}

\appendix

\section{Groupoids}

\label{sec:groupoids}

\subsection{Homotopy theory of groupoids}

We briefly recall the needed basic notions of groupoids and their homotopy
cardinalities. 

\begin{blanko}{Groupoids.}
  A {\em groupoid} is a small category in which all the arrows are invertible.
  A map of groupoids is just a functor.  Let $\Grpd$ denote the category of
  groupoids and maps.
  
  Intuitively we consider groupoids as sets with built-in symmetries.  While a
  group models symmetry automorphisms of one object, groupoids model
  automorphisms and isomorphisms between several objects.  
\end{blanko}

\begin{blanko}{Homotopy equivalences.}
  A homotopy between two maps of groupoids is just a natural transformation of
  functors.  A map of groupoids $f:X\to Y$ is called a {\em homotopy
  equivalence} when there exists a pseudo-inverse $g: Y \to X$, meaning that the
  two composites are homotopic to the identities: $g\circ f \simeq \id_X$ and
  $f\circ g \simeq \id_Y$.  Just as for categories, homotopy equivalences
  can also be characterised as functors that are essentially surjective and 
  fully faithful.

  Homotopy equivalence is the appropriate notion of sameness for groupoids, and
  it is important that all the notions involved be invariant under homotopy
  equivalence.

  We adopt the convention that all notions in the paper are the homotopy
  invariant ones: outside this appendix we will usually say {\em equivalence,
  finite, discrete, trivial, cartesian, pullback, fibre, sum, colimit} and {\em
  monomorphism} instead of `homotopy equivalence', `homotopy finite', `homotopy
  discrete', etc, for the notions defined below.  It is essential that the word
  `homotopy' is understood throughout.
\end{blanko}

\begin{blanko}{Connectedness, discreteness.}
  A groupoid $X$ is {\em connected} if $\operatorname{obj}(X)$ is non-empty and
  the set $\Hom_X(x,y)$ is non-empty for all $x,y\in X$.  A maximal connected subgroupoid of $X$
  is termed a {\em component} of $X$ and denoted $[x]$ or $X_{[x]}$, where $x$
  is some object in the component.  The set of components is denoted $\pi_0(X)$.
  We denote by $\pi_1(X,x)$ the automorphism group $\Aut_X(x)=\Hom_X(x,x)$.  A
  groupoid $X$ is {\em homotopy discrete} if $\pi_1(X,x)$ is trivial for all
  $x$, and {\em contractible} if it is homotopy discrete and also connected.
  This means homotopy equivalent to a point, i.e.~the terminal groupoid $1$.
\end{blanko}

\begin{blanko}{Finiteness.}\label{finite}
  A groupoid $X$ is {\em locally finite} if $\pi_1(X,x)$ is finite for every
  $x$, and is {\em (homotopy) finite} if in addition $\pi_0(X)$ is finite.  We
  denote by $\grpd$ the category of finite groupoids.
\end{blanko}

\begin{blanko}{Pullbacks.}\label{pbk}
  The {\em homotopy fibre product} of \homomorphisms $f:G\to B$ and $g:E\to B$ is 
  the groupoid $H=G\times_B E$ whose objects are triples $(x,y,\varphi)$
  consisting of $x\in G$, $y\in E$, and $\varphi:fx\to gy$ in $B$,
  and whose arrows $(x',y',\varphi')\to (x,y,\varphi)$ are
  pairs $(\beta,\epsilon)\in \Hom_G(x',x) \times \Hom_E(y',y)$ 
  such that $\varphi \circ f(\beta)=g(\epsilon)\circ \varphi'$.
  There are canonical projections $p,q$,
\begin{equation}\vcenter{
\xymatrix@R-1ex{
H\drpullback 
\rto^-{q}\dto_-{p}
&E\ar[d]^-{g}\\G\ar[r]_-{f}&B.
}}\label{pullbacksquare}
\end{equation}
The diagram does not commute on the nose, but the third components of objects
$a=(x,y,\varphi)$ provide a natural isomorphism $\{\varphi:fp(a)\cong gq(a)\}$.  We
say a square \eqref{pullbacksquare} is {\em homotopy cartesian} or a {\em
homotopy pullback} if $H$ is homotopy equivalent to the homotopy fibre product
$G\times_B E$
given explicitly above.  The \homomorphism $p$ is sometimes termed the {\em
pullback of $g$ along $f$} and denoted $f\upperstar (g)$.
\end{blanko}

\begin{blanko}{Fibres.}\label{fibres}
  The {\em homotopy fibre} $E_b$ of a \homomorphism $p:E\to B$ over an object
  $b$ of $B$ is the homotopy pullback of $p$ along the map $\name b:1 \to B$ 
  that picks out the element $b$:
  $$\xymatrix{
     E_b\drpullback \ar[r]\ar[d] & E \ar[d] \\
     1 \ar[r]_{\name b} & B
  }$$
\end{blanko}
\begin{blanko}{Loops.}\label{loops}
  The {\em loop groupoid} $\Omega_b B$ of a groupoid $B$ at an object $b$ is given by
  the homotopy pullback $1\times_B 1$ of the inclusion $\name b:1\to B$ along
  itself.  This is discrete: it has $\Aut_B(b)$ as its set of objects, and only
  the identity isomorphisms.
\end{blanko}

\subsection{Slices and the fundamental equivalence}

\begin{blanko}{Slices.}
  We shall need homotopy slices, sometimes called weak slices.  First recall the
  usual notion of slice category: If $\CC$ is a category, and $I\in \CC$, then
  the usual slice category $\CC{/I}$ is the category whose objects are morphisms
  $X \to I$ in $\CC$ and whose arrows are commutative triangles
  $$\xymatrix@R-3ex@C-1ex{
  X' \ar[rdd] \ar[rr]    && X \ar[ldd]   \\ \\
  &     I   .
  }$$

  We are concerned instead with groupoid-enriched categories $\CC$,
  i.e.~categories such that the arrows between each pair of objects $X,Y$ define
  a groupoid $\Map(X,Y)$ instead of just a set, and the composition law is given by
  groupoid maps instead of just functions.  Thus, between two parallel arrows $X
  \rightrightarrows Y$ there may be invertible $2$-cells.  The {\em homotopy slice category}
  $\CC_{/I}$ then has as objects the morphisms $X \to I$; its arrows are
  triangles with a 2-cell
  \begin{equation}\label{triangledigram}
  \vcenter{\xymatrix@R-3ex@C-1ex{
  X \ar[rdd] \ar[rr]    && X' \ar[ldd]   \\ &\Rightarrow\\
  &     I   .
  }}
  \end{equation}
The basic example is $\CC=\Grpd$ with
$2$-cells given by homotopies between maps (that is, the natural isomorphisms).
\end{blanko}

\begin{blanko}{Homotopy sums and Grothendieck construction.}
  For a map $p:E \to B$, each isomorphism $\beta:b'\to b$ in $B$ induces an
  equivalence of homotopy fibres $\beta_*:E_{b'}\to E_b$, sending an object
  $(1,e,\varphi\!:\!pe\cong b')$ to
  $(1,e,\beta\varphi\!:\!pe\cong b)$.  Thus the
  homotopy fibres of $p:E\to B$ form a $B$-indexed family of groupoids, that is,
  a functor $E_{(-)}$ from $B$ to the category $\Grpd$ of groupoids.

  The {\em homotopy sum} of any $B$-indexed family of groupoids $E:B\to \Grpd$
  is the groupoid given by the {\em homotopy colimit} of this functor, which may
  be defined by the {\em Grothendieck construction} and denoted $\int^{b\in
  B}E_b$.  Its objects are pairs $(b, e)$ with $b\in B$ and of $e\in E_b$, and
  isomorphisms $(b',e')\to(b,e)$ are pairs $(\beta,\epsilon)$ of
  isomorphisms $\beta:b'\to b$ in $B$ and $\epsilon : \beta_*e'\to e$ in
  $E_{b}$.

  The Grothendieck construction of any family $E:B\to\Grpd$ comes equipped with
  a canonical projection to $B$ whose homotopy fibres give back the original
  family $E$ up to homotopy equivalence.  Conversely, for any \homomorphism
  $E\to B$, the homotopy sum of its homotopy fibres $E_b$ is homotopy
  equivalent, over $B$, to $E$.  Thus we have
\end{blanko}

\begin{thm}[Fundamental Equivalence]
  There is a canonical equivalence between the categories of groupoids over a
  fixed groupoid $B$ and that of groupoid-valued functors from $B$,
  $$
  \Grpd_{/B}\;\simeq\;\Grpd^B
  $$
  given by taking homotopy fibres and the Grothendieck construction.
\end{thm}

\begin{blanko}{Monomorphisms.}\label{mono}
  A \homomorphism $E\to B$ is a {\em homotopy monomorphism} if each homotopy
  fibre $E_b$ is empty or contractible.  Up to homotopy
  equivalence, such a \homomorphism is the inclusion of some collection of
  connected components of $B$, that is, the Grothendieck construction of an
  indicator function $B\to\{\varnothing,1\}\subset\Grpd$.  Note that in general neither
  $\name b:1\to B$ nor the diagonal map $B\to B\times B$ are homotopy mono.
 \end{blanko}
  
\begin{blanko}{Finite maps.}
  A \homomorphism is {\em homotopy finite} if each homotopy fibre is homotopy
  finite.  A pullback of any homotopy monic or finite map is again homotopy
  mono or finite.
\end{blanko}

\begin{blanko}{Families.}
  The homotopy sum of an $I$-indexed family in $\Grpd^B$ is defined as the
  homotopy sum of the corresponding object of $\Grpd^{I\times B}$, composed with
  the projection,
  $$
  E\longrightarrow I\times B\longrightarrow B.
  $$ 
  Homotopy sums of $I$-indexed families in $\Grpd_{/B}$ are defined similarly.
  We regard the maps $\name b:1\to B$, for $[b]\in\pi_0B$ as a {\em basis} of
  $\Grpd_{/B}$, in analogy with vector spaces. {\em Scalar multiples} $A\,\name b$ of basis elements in $\Grpd_{/B}$ are given by $A\to 1\xrightarrow{\name b}B$. 
\end{blanko}

\begin{lemma}\label{lem:fashosum}
  Any $f:E\to B$ in $\Grpd_{/B}$ may be expressed as a linear combination
  of basis elements as follows
$$  f \;\;\simeq\;\; \int^{e\in E} \name{f(e)} 
 \;\;\simeq\;\; \int^{b\in B} E_b\,\name{b} .
  $$
\end{lemma}

\subsection{Linear functors}

\label{sec:LIN}

\begin{blanko}{Basic slice adjunction.}\label{adj}
Taking homotopy pullback along a morphism of groupoids 
$f:B'\to B$  defines a functor between the slice categories
$$
f^*
:\Grpd_{/B}\to \Grpd_{/B'} .
$$
This has a homotopy left adjoint, given by postcomposition,
$$
f\lowershriek  :\Grpd_{/B'}\to \Grpd_{/B}.
$$ 
The homotopy adjointness is expressed by 
natural equivalences of mapping groupoids
\begin{align}
\Map_{/B}(f\lowershriek E',E)&\simeq
\Map_{/B'}(E',f^*E).\label{slice-adj}
\end{align}
\end{blanko}

Moreover,

\begin{lemma}[Beck--Chevalley]\label{lem:beck-chevalley}
For any homotopy pullback square \eqref{pullbacksquare}, the functors 
$$
q\lowershriek \,p\upperstar 
,\;
g\upperstar f\lowershriek 
:\Grpd_{/G}\to\Grpd_{/E}
$$
are naturally homotopy equivalent.
\end{lemma}

\begin{blanko}{Spans and linear functors.}
  A pair of groupoid \homomorphisms $A\xleftarrow r G\xrightarrow f B$ is termed a {\em span}
between $A$ and  $B$, and induces a functor between the slice categories by 
pullback and postcomposition 
$$
f\lowershriek \,r\upperstar :\Grpd_{/A}\longrightarrow\Grpd_{/B}   .
$$
A functor $\Grpd_{/A}\longrightarrow\Grpd_{/B}$ is {\em linear} if it is 
homotopy equivalent to one arising from a span in this way. 
By the Beck--Chevalley Lemma \ref{lem:beck-chevalley},
 composites of linear functors are linear,
$$
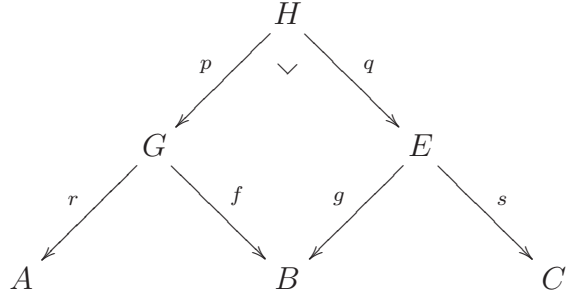
\vcenter{\xymatrix@!R=4ex@!C=5ex{
&&H\dpullback\dlto_p\drto^q\\&G\dlto_r\drto^f&&E\dlto_g\drto^s\\A&&B&&C\\
\makebox[0em][r]{$(sq)\lowershriek(rp)\upperstar:\;$}\Grpd_{/A}\rrto^{f\lowershriek r\upperstar}&&\Grpd_{/B}\rrto^{s\lowershriek g\upperstar}&&\Grpd_{/C}\;.\!\!\!
}}
$$
We write $\LIN$ for the monoidal $2$-category of all slice categories $\Grpd_{/B}$ and linear functors between them, with the tensor product induced from the cartesian product
$$
\Grpd_{/A}\otimes \Grpd_{/B} \;:=\;\Grpd_{/A\times B}.
$$
The neutral object is $\Grpd \simeq \Grpd_{/1}$, playing the role of the ground 
field.
\end{blanko}

\begin{blanko}{Duality.}\label{duality}
  The functor category $\Grpd^S$ is the {\em linear dual} of the slice category
  $\Grpd_{/S}$,
  since there is an equivalence (see~\cite[\S2.11]{GKT:HLA})
  $$
  \Grpd^S\;\;\simeq\;\;\LIN(\Grpd_{/S},\Grpd).
  $$
  A span $A\leftarrow G\to B$ defines both a linear functor
  $\Grpd_{/A}\to\Grpd_{/B}$ and the dual linear functor $\Grpd^B\to\Grpd^A$.
  In particular the span $1\leftarrow G\to S$ may be thought of as an element of
  $\Grpd_{/S}$, and its transpose $S\leftarrow G\to 1$ as an element of
  $\Grpd^S$.

  There is a canonical pairing
  \begin{align}
  \Grpd_{/S}\times{}&\Grpd^S\to\Grpd
  \nonumber
  \\ \label{objective-pairing}
  \langle\; \name t\;,\,\;&h^s\;\rangle
   = \Hom(s,t) = 
  \left\{\begin{array}{cr}\Omega_s(S)&(s\cong t)\\\varnothing&(s\not\cong t)\end{array}\right.
  \end{align}
  The maps $\name t:1\to S$ (or the spans $1\leftarrow 1\to S$) form the 
  canonical basis of the slice category, and the representable functors 
  $h^s=\Hom(s,-):S\to\Grpd$ (or the spans $S\leftarrow 1\to 1$) form the 
  canonical basis for the dual.
\end{blanko}

\subsection{Cardinality}

\begin{blanko}{Cardinality of groupoids.}
The cardinality of a finite groupoid $X$ is given by
$$
\norm{X}\; :=\; \sum_{[x]\in \pi_0(X) } \frac{1}{\norm{\pi_1(X,x)}} \;\in\;\mathbb Q.
$$
where $\pi_1(X,x)$ is the automorphism group $\Aut(x)$ of the 
object $x$ in the groupoid $X$, and the norm signs on the right refer to the order of the group.
Homotopy equivalent groupoids have the same cardinality.
For any component of a locally finite groupoid $B$ we have
\begin{equation}\label{normloop}
\norm{B_{[b]}}
\;=\;\norm{\pi_1(B,b)}^{-1}
\;=\;\norm{\Omega_b(B)}^{-1}.
\end{equation} 

For any function $q:\pi_0B\to \Q$, we use the notation
$$
\int^{b\in B}\!\!\!
q(b)\;:=\;\sum_{[b]\in \pi_0B}|B_{[b]}|\;
q(x)\;=\;\sum_{[b]\in \pi_0B}\frac{q(b)}{|\pi_1(B,b)|}  .
$$ 
This is chosen to resemble the Grothendieck construction notation since for any \homomorphism $E\to B$ from a finite groupoid we have, by \cite[Lemma~3.5]{GKT:HLA},
$$
|E|=\int^{b\in B}\!\!\!|E_b|.
$$
The case of the \homomorphism $\name b:1\to B$ is just equation \eqref{normloop}.
\end{blanko}

\begin{blanko}{Global cardinality.}
A span $A \xleftarrow r G \xrightarrow f B$, and the corresponding 
linear functor $\Grpd_{/A}\to\Grpd_{/B}$, are termed {\em finite} if
the \homomorphism $r$ is finite
(that is, the homotopy fibres of $r$ are finite).
We have \cite[Proposition~4.3]{GKT:HLA},
\begin{prop}\label{finitetypespan}
Let $A$, $B$, $G$ be locally finite groupoids and 
$A \xleftarrow{\;\;}  G \xrightarrow{\;\;}  B$ a finite span.
Then the induced finite linear functor $\Grpd_{/A}\to \Grpd_{/B}$ restricts to
\begin{equation}\label{finitelinear}
\grpd_{/A}\to\grpd_{/B}.
\end{equation}
\end{prop}

To a slice category $\grpd_{/A}$, with $A$ locally finite, we associate the vector space 
$\Q_{\pi_0A}$ 
with canonical basis $\{\delta_a\}_{[a]\in\pi_0A}$. 
To the finite linear functor \eqref{finitelinear}, we associate the linear map
\begin{eqnarray}\label{arrow-card}
\qquad\qquad\nonumber   \Q_{\pi_0 A} & \longrightarrow &\Q_{\pi_0 B}  \\
    \delta_a & \mapsto \quad& \!\!\!\!\!\!\!\!\!\!
\sum_{[b]\in\pi_0B}\!\!\! 
    \norm{B_{[b]}} \norm{G_{a,b}} \delta_b 
\;=\;\int^{b\in B} \!\!\!\norm{G_{a,b}} \delta_b 
\end{eqnarray}
where $G_{a,b}$ are the fibres of the map $G\to A\times B$  defined by the span.
This process is functorial \cite[Proposition~8.2]{GKT:HLA}, and
defines what we call {\em meta} or {\em global cardinality} 
$$\|\;\;\|:\ind\lin\to \Vect
$$
from the category $\ind\lin$ of slice categories $\grpd_{/A}$  ($A$ locally finite) and finite linear functors.
\end{blanko}
\begin{blanko}{Local cardinality.} 
To each object $p:G\to B
$ in $\grpd_{/B}$ ($B$ locally finite) we can associate a vector $\norm{p:G\to B}$ in $\mathbb Q_{\pi_0B}$, called the {\em relative} or {\em local cardinality},
$$
\norm p\; :=\; \sum_{[b]\in\pi_0B}\norm{B_{[b]}} \norm{G_{b}} \delta_b 
 \;=\; \int^{b\in B}\!\!\! \norm{G_{b}} \delta_b 
$$ 
Note that $p$ determines a finite linear functor via $1\leftarrow G\to B$, and the local cardinality $\norm p$ is just the  image of $1$ under the global cardinality $\Q\to \Q_{\pi_0B}$.
It follows that local cardinality respects the action of finite linear functors $L$,  
$$
|L(p)|\;=\;\|L\|(|p|).
$$
The local cardinality of the basis object $\name b:1\to B$ in $\grpd_{/B}$ is just the basis vector  $\delta_b$ in $\Q_{\pi_0B}$, by \eqref{normloop}.

To simplify notation we will write $\normnorm L$ for $\|L\|$
when the meaning is clear from the context, and say just cardinality rather than meta, global, relative or local cardinality.

\end{blanko}

\begin{blanko}{Cardinality of the dual.}\label{dual-card}
Dually we define cardinality of finite-groupoid valued functors (see \ref{duality}) as a map
$$\norm { \ \ } : \grpd^S \to\|\grpd^S\|=\Q^{\pi_0 S}$$
where $\Q^{\pi_0 S}$ is the function space,
the profinite dimensional vector space with pro-basis given by
the characteristic functions $\delta^s$.

Finite spans $A\leftarrow G\to B$ define linear maps $\grpd^B\to\grpd^A$, 
whose cardinality is defined using the same matrix as in \eqref{arrow-card} above:
\begin{eqnarray}\label{transpose-card}
\qquad\qquad\nonumber   \Q^{\pi_0 B} & \longrightarrow &\Q^{\pi_0 A}  \\
    \delta^b & \mapsto \quad& \!\!\!\!\!\!\!\!\!\!
\sum_{[a]\in\pi_0A}\!\!\! 
    \norm{B_{[b]}} \norm{G_{a,b}} \delta^a 
\end{eqnarray}
An element $g\in\grpd^S$ is represented by a finite span $S\leftarrow G\to 1$ (using the fundamental equivalence) and has cardinality
\begin{equation}\label{presheaf-card}
|g|\;=\;\|(S\leftarrow G\to 1)\|\left(\delta^1\right)\;=\;\sum_{[s]\in\pi_0S}\norm{g(s)}\delta^s.
\end{equation}
The cardinality of the representable functor $h^s$ in $\grpd^S$ is thus
\begin{equation}\label{representable-card}
|h^s|\;=\;\|(S\leftarrow 1\to 1)\|\left(\delta^1\right)\;=\;\norm{\Omega_s(S)}\delta^s
\end{equation}
and the `objective pairing' \eqref{objective-pairing} 
is consistent with the classical pairing
$$\langle\norm{\name t},\norm{h^s}\rangle\;\;=\;\;
\langle\delta_t,\norm{\Omega_s(S)}\delta^s\rangle\;\;=\;\;
\langle\delta_t,\delta^s\rangle\norm{\Omega_s(S)}\;\;=\;\;\norm{\langle\name t,h^s\rangle}.$$
\end{blanko}

\section{Simplicial groupoids and fat nerves}

\label{sec:simplicial}

In this appendix, we provide some background material on simplicial groupoids
and fat nerves.  The general notion of simplicial set (originally termed a {\em
complete semi-simplicial complex}) has been widely used in homotopy theory since
the work of Eilenberg, Kan and others in the 1950s, owing its utility on one
hand to the fact that simplicial sets are a model for topological spaces up to
homotopy by way of the singular functor, and on the other hand because it
receives a fully faithful functor from the category of small categories, namely
the nerve (see \ref{app:nerve} below).  The theory of $\infty$-categories, the
common generalisation of spaces up to homotopy and categories, exploits the
simplicial setting in a crucial way.
 
Any poset can naturally be regarded as a category, hence we may talk about
posets in terms of their nerves.  In combinatorics, however, it is common to
view posets as  simplicial {\em complexes} instead of simplicial sets, 
associating to a poset its order complex.
The simplicial complexes that arise in this way have a canonical order on each
simplex, and due to this they can be regarded as special kinds of simplicial
sets, characterised by the property that $n$-simplices are completely determined
by their vertex sets.  Although such simplicial sets are of a simple kind,
the subcategory they form is not as nice as the category of simplicial sets
(which is a presheaf topos).  For the purposes of the present undertakings,
it is crucial to work with simplicial sets.

In this short appendix we recall the basic definitions, contrasting with
simplicial complexes.

\subsection{Simplicial sets and nerves}

\begin{blanko}{The simplex category (the topologist's Delta).}
  Let $\simplexcategory$ be the \emph{simplex category}, whose objects are the finite
  nonempty standard ordinals $${[n]}=\{0<1<\cdots<n\},$$ and whose arrows are the
  order-preserving maps between them.  These maps are generated by the
  injections ${\partial^{i}}:[n-1]\to[n]$ that skip the value $i$, termed {\em
  coface} maps, and the surjections ${\sigma^{i}}:[n+1]\to[n]$ that repeat $i$,
  termed {\em codegeneracy} maps.  The obvious relations between these
  generators are called the {\em cosimplicial identities} (dual to the 
  simplicial identities below). 
\end{blanko}

\begin{blanko}{Simplicial sets.}
  A \emph{simplicial set} is a functor
  ${X:\simplexcategory\op\to\Set}$.
  One writes $X_n$ for the
  image of $[n]$, and $d_i,s_i$ for the images of $\partial^i,\sigma^i$.
  The elements of $X_n$ are called {\em $n$-simplices}.

  Explicitly, a simplicial set $X$ is thus
  a sequence of sets $X_n$ ($n\geq 0$) together with
  {\em face maps} $d_i:X_n\to X_{n-1}$ and {\em degeneracy maps} $s_i:X_n\to
  X_{n+1}$, $(0\leq i\leq n)$,
  $$
\xymatrix@C+2em@R-2.5mm{
X_0  
\ar[r]|(0.55){s_0} 
&
\ar[l]<+2mm>^{d_0}\ar[l]<-2mm>_{d_1} 
X_1  
\ar[r]<-2mm>|(0.6){s_0}\ar[r]<+2mm>|(0.6){s_1}  
&
\ar[l]<+4mm>^(0.6){d_0}\ar[l]|(0.6){d_1}\ar[l]<-4mm>_(0.6){d_2}
X_2 
\ar[r]<-4mm>|(0.6){s_0}\ar[r]|(0.6){s_1}\ar[r]<+4mm>|(0.6){s_2}  
&
\ar[l]<+6mm>|(0.6){d_0}\ar[l]<+2mm>|(0.6){d_1}\ar[l]<-2mm>|(0.6){d_2}\ar[l]<-6mm>|(0.6){d_3}
X_3 
\ar@{}|\cdots[r]&
}
$$
  subject to the {\em simplicial identities}:
  $d_is_i=d_{i+1}s_i=1$ and
  $$
  d_id_j=d_{j-1}d_i,\quad
  d_{j+1}s_i=s_id_j,\quad
  d_is_j=s_{j-1}d_i,\quad
  s_js_i=s_is_{j-1},\quad
  \qquad(i<j).
  $$ 
  
\end{blanko}

\begin{blanko}{Simplicial maps.}
  A \emph{simplicial \homomorphism} $F:X\to Y$ between simplicial sets is
  given by a sequence of maps $(F_n:X_n\to Y_n)_{n\geq0}$ commuting
  with face and degeneracy maps, that is, a natural transformation between the
  functors $X$ and $Y$.  A simplicial \homomorphism is {\em cartesian} with
  respect to an order-preserving map $[m]\to [n]$ in $\simplexcategory$ if the
  associated naturality square is a pullback (see \ref{culf-maps} for examples).
\end{blanko}

\begin{blanko}{Simplicial complexes.}
  A {\em simplicial complex} $K$ consists of a set $V$ of vertices together with
  a collection $S(K)$ of nonempty subsets of $V$, termed the {\em simplices} of $K$,
  satisfying
  \begin{itemize}
   \item All the one-element subsets of $V$ are simplices of $K$.
   \item Any non-empty subset of a simplex of $K$ is a simplex of $K$.
  \end{itemize}
A map between simplicial complexes is a function between their vertex 
sets such that the image of each simplex is a simplex.
\end{blanko}

\begin{blanko}{Locally ordered simplicial complexes.} \label{app:LOSC}
  Certain simplicial complexes can be regarded as simplicial sets, but some
  ordering is necessary so as to have well-defined face maps.  We call a
  simplicial complex {\em locally ordered} if there is specified a linear order
  on each simplex, in such a way that all inclusions preserve these orders.
  (The terminology {\em hierarchical simplicial complex} is used by
  Ehrenborg~\cite{Ehrenborg}.)  A map of locally ordered simplicial complexes is
  a map of simplicial complexes whose restriction to each simplex is
  order preserving.  This defines a category $\kat{LOSC}$.
  
  To each locally ordered simplicial complex $K$ there is associated canonically
  a simplicial set $X$ whose $n$-simplices are sequences $(v_0,v_1,\ldots,v_n)$
  of elements in $V$, permitting repetitions, which as a set is required to
  form a simplex $F\in S(K)$, and which as a sequence is required to be
  non-decreasing for the linear order in the simplex $F$.  This can be described more formally as follows.
  Each linear order $[n]\in \simplexcategory$ can be regarded as a locally ordered
  simplicial complex, defining in fact a functor $\simplexcategory 
  \to\kat{LOSC}$.  Now the simplicial set $X$ assigned to $K$ has
  $$
  X_n := \Hom_{\kat{LOSC}}([n], K).
  $$
  This automatically accounts for the face and degeneracy maps, simply
  induced by precomposition with the coface
  maps and codegeneracy maps $[m] \to [n]$ in $\simplexcategory$. 

  This assignment defines a fully faithful functor 
  from locally ordered
  simplicial complexes to simplicial sets.  (Note that allowing repetition in
  the sequences is necessary for the assignment to be functorial in maps of
  locally ordered simplicial complexes, because these are allowed to send a
  simplex to a simplex of lower dimension.)
\end{blanko}

\begin{blanko}{The order complex and the nerve of a poset.}
  The {\em order complex} of a poset $P$ is the simplicial complex whose
  vertices are the elements of $P$ and whose $n$-simplices are those subsets
  that form $n$-chains $v_0<\dots< v_n$ in the poset.  The order complex is
  naturally locally ordered since each simplex is a total order, and its
  associated simplicial set is usually termed the {\em nerve} of the poset.  The
  definition of the nerve extends to more general categories as follows.
\end{blanko}

\begin{blanko}{Strict nerve.}\label{app:nerve}
  The {\em nerve} of a category $\CC$ is the simplicial set
  $$
  N\CC : \simplexcategory\op\to\Set
  $$
  whose set of $n$-simplices is the set of sequences of $n$ composable arrows in
  $\CC$ (allowing identity arrows).  The face maps are given by composing arrows
  (for the inner face maps) and by discarding arrows at the beginning or the end
  of the sequence (outer face maps).  The degeneracy maps are given by inserting an
  identity map in the sequence.   By regarding the total order $[n]$
  as a category, we see that 
  a sequence of $n$ composable arrows in $\CC$ is the same thing
  as a functor $[n]\to \CC$, and more formally the $n$-simplices can be 
  described as
  $$
  (N\CC)_n = \Fun([n],\CC),
  $$
  and in particular we see that the face and degeneracy maps of $N\CC$ are given
  simply by precomposition with the coface and codegeneracy maps in 
  $\simplexcategory$.
\end{blanko}

\subsection{Simplicial groupoids, fat nerves, and Segal spaces}

\begin{blanko}{Simplicial groupoids.}
  For any category $\EE$,
  one can talk about simplicial objects $X:\simplexcategory\op\to\EE$.
  Thus, in the case of the category of groupoids, 
  a \emph{simplicial groupoid} is a sequence of groupoids
  $X_n$, $n\geq 0$, and face and degeneracy \homomorphisms
  $d_i:X_n\to X_{n-1}$, $s_i:X_n\to X_{n+1}$, $(0\leq i\leq n)$, subject to the
   simplicial identities above.
\end{blanko}

\begin{blanko}{Fat nerve of a small category.}\label{fatnerve}
  Important examples of simplicial groupoids are given by the {\em fat nerve} of
  a small category $\CC$.  Here $X_n$ is the groupoid of all composable
  sequences
  $a_0\xrightarrow{\alpha_1}a_1\xrightarrow{\alpha_2}\dots\xrightarrow{\alpha_n}a_n$
  of $n$ arrows in $\CC$, that is,
  $$
  X_n=\{ \text{functors }\alpha:[n]\to \CC\}.
  $$ 
  In the case of the classical {\em strict nerve} this is just a set,
  or a discrete groupoid; in the {\em
  fat nerve}, $X_n$ includes all natural isomorphisms $\alpha\to\alpha'$
  $$
  \xymatrix{
  \cdot \ar[r] \ar[d]^*-[@]=0+!L{\scriptstyle \sim} & \cdot \ar[d]^*-[@]=0+!L{\scriptstyle \sim}
  \ar[r]  & \cdot \ar[r] \ar[d]^*-[@]=0+!L{\scriptstyle \sim} &\cdots\ar[r]&
   \cdot \ar[d]^*-[@]=0+!L{\scriptstyle \sim} \\
  \cdot \ar[r] & \cdot\ar[r] & \cdot\ar[r] &\cdots\ar[r]& \cdot
  }$$
   
  This can be described succinctly in categorical terms, in terms of
  the functor category, but allowing only invertible natural transformations:

  $$
  (\fatnerve \CC)_n = \Fun([n],\CC)^\iso .
  $$
  As in the previous cases (\ref{app:LOSC}, \ref{app:nerve}), this automatically accounts for face and degeneracy 
  maps by precomposition.
  In particular, $d_0:X_1 \to X_0$ assigns to an arrow its codomain, and
  $d_1: X_1 \to X_0$ assigns to an arrow its domain.

  Since $X_2$ is by definition 
  the groupoid of composable pairs of arrows, we have $X_2 \simeq X_1 
  \times_{X_0} X_1$.  Here the fibre product is
  \begin{equation}\label{appendix-segal-pb}\vcenter{\xymatrix{
     X_2\drpullback \ar[r]^{d_0}\ar[d]_{d_2} & X_1 \ar[d]^{d_1} \\
     X_1 \ar[r]_{d_0} & X_0}
  }\end{equation}
  expressing the composability condition: only those pairs of arrows
  such that the target of the first matches the source of the second.

  In particular, $d_1: X_2 \to X_1$ is the composition map.  Also, $d_0:X_2 \to
  X_1$ assigns to a composable pair the second arrow, and $d_2 : X_2 \to X_1$
  assigns to a composable pair the first arrow.  (Here we are referring to the
  order of composition, as in $a$-followed-by-$b$, and not the order used when
  writing this as $b\circ a$.)
  
  Note that if $\CC$ is just a poset, then it has no invertible arrows
  except the identities.
  Therefore, the notions of strict and fat nerve coincide.
\end{blanko}

\begin{blanko}{Rezk complete Segal spaces.}\label{Rezk}\label{Segal}
  A simplicial groupoid is a {\em Segal space} if $X_2 \simeq X_1\times_{X_0}
  X_1$, as in \eqref{appendix-segal-pb}, and in general the canonical Segal map
  $$
  X_n \longrightarrow  X_1 \times_{X_0} X_1\times_{X_0} \dots \times_{X_0} X_1
  $$ 
  is an equivalence for each $n\geq 1$.

  Consider the contractible groupoid generated by one isomorphism $0\isopil 1$, and its
  strict nerve $J$.  A Segal space $X$ is {\em Rezk complete} if the map $J\to *$
  induces an equivalence of groupoids $\Map(*,X)\to\Map(J,X)$, which in 
  turn means
  that $s_0: X_0 \to X_1$ is fully faithful and has as its essential image the
  arrows that admit left and right quasi-inverses.  
  More intuitively, the condition expresses the idea that up to homotopy there are no other
  weakly invertible arrows than those coming from $X_0$ via the
  degeneracy map $s_0$.

  The Rezk complete Segal spaces are precisely those simplicial groupoids
  that are levelwise-equivalent to the fat nerve of a category.
\end{blanko}

\begin{blanko}{Monoidal groupoids.}\label{monoidalgroupoids}
  A monoidal groupoid is a monoidal category $(\CC,\tensor,I)$ which happens to
  be a groupoid.  For these, one can define the {\em monoidal nerve}, which is
  essentially a simplicial groupoid $X:\simplexcategory\op\to\Grpd$.  One takes
  $X_0$ to be a singleton, takes $X_1$ to be the groupoid itself, and more
  generally let $X_n$ be the $n$-fold cartesian product
  $$
  X_n = 
  \underset{n \text{ factors}}{\underbrace{\CC \times \cdots \times \CC}}   .
  $$
  The outer face maps just project away an outer factor. The inner face maps use
  the monoidal structure $\tensor : \CC \times \CC \to \CC$ on two adjacent
  factors. The degeneracy maps insert a unit object. All this is completely
  canonical, given the monoidal structure. The only problem is that the
  simplicial identities do not hold on the nose, due to the fact that the
  monoidal structure is not assumed to be strict. The diagram is therefore not
  literally speaking a simplicial groupoid, but only a pseudo-functor
  $\simplexcategory\op\to\Grpd$.
  
  While this may be a slight annoyance sometimes, it is not actually important
  for the purpose of this work: for the sake of defining a homotopy-coherently
  coassociative coalgebra structure on $\Grpd_{/X_1}$, a pseudo-functor is just
  as good as a strict functor.
  Another thing is that one can alternatively invoke strictification theorems
  (see Mac Lane~\cite[\S XI.3, Theorem 1]{MacLane:categories}): any monoidal 
  category is equivalent to a strict monoidal category.  The monoidal nerve of
  the strictification of a monoidal groupoid is then a simplicial
  groupoid on the nose, equivalent to the original monoidal nerve.
\end{blanko}


\begin{thebibliography}{100}

\bibitem{Aguiar-Bergeron-Sottile}
{\sc Marcelo Aguiar, Nantel Bergeron, {\rm and }Frank Sottile}.
\newblock {\em Combinatorial {H}opf algebras and generalized
  {D}ehn-{S}ommerville relations}.
\newblock Compos. Math. {\bf 142} (2006), 1--30.

\bibitem{Aguiar-Mahajan}
{\sc Marcelo Aguiar {\rm and }Swapneel Mahajan}.
\newblock {\em Monoidal functors, species and {H}opf algebras}, vol.~29 of CRM
  Monograph Series.
\newblock American Mathematical Society, Providence, RI, 2010.
\newblock With forewords by Kenneth Brown and Stephen Chase and Andr{\'e}
  Joyal.

\bibitem{Baez-Dolan:finset-feynman}
{\sc John~C. Baez {\rm and }James Dolan}.
\newblock {\em From finite sets to {F}eynman diagrams}.
\newblock In {\em Mathematics unlimited---2001 and beyond}, pp. 29--50.
  Springer, Berlin, 2001.

\bibitem{Baez-Dolan:zeta}
{\sc John~C. Baez {\rm and }James Dolan}.
\newblock {\em Zeta functions}.
\newblock \url{http://ncatlab.org/johnbaez/Zeta+functions} (2010).

\bibitem{Baez-Hoffnung-Walker:0908.4305}
{\sc John~C. Baez, Alexander~E. Hoffnung, {\rm and }Christopher~D. Walker}.
\newblock {\em Higher dimensional algebra {VII}: groupoidification}.
\newblock Theory Appl. Categ. {\bf 24} (2010), 489--553.
\newblock ArXiv:0908.4305.

\bibitem{Batanin-Markl:1404.3886}
{\sc Michael Batanin {\rm and }Martin Markl}.
\newblock {\em Operadic categories and duoidal {D}eligne's conjecture}.
\newblock Adv. Math. {\bf 285} (2015), 1630--1687.
\newblock ArXiv:1404.3886.

\bibitem{Baues:Hopf}
{\sc Hans-Joachim Baues}.
\newblock {\em The cobar construction as a {H}opf algebra}.
\newblock Invent. Math. {\bf 132} (1998), 467--489.

\bibitem{Bergeron-Labelle-Leroux}
{\sc Fran{\c{c}}ois Bergeron, Gilbert Labelle, {\rm and }Pierre Leroux}.
\newblock {\em Combinatorial species and tree-like structures}, vol.~67 of
  Encyclopedia of Mathematics and its Applications.
\newblock Cambridge University Press, Cambridge, 1998.
\newblock Translated from the 1994 French original by Margaret Readdy, With a
  foreword by Gian-Carlo Rota.

\bibitem{Bergner-et.al:1609.02853}
{\sc Julia~E. Bergner, Angélica~M. Osorno, Viktoriya Ozornova, Martina
  Rovelli, {\rm and }Claudia~I. Scheimbauer}.
\newblock {\em 2-{S}egal sets and the {W}aldhausen construction}.
\newblock Topology and its Applications {\bf 235} (2018), 445--484.
\newblock ArXiv:1609.02853.

\bibitem{Bergner-Osorno-Ozornova-Rovelli-Scheimbauer:1809.10924}
{\sc Julia~E. Bergner, Angélica~M. Osorno, Viktoriya Ozornova, Martina
  Rovelli, {\rm and }Claudia~I. Scheimbauer}.
\newblock {\em 2-{S}egal objects and the {W}aldhausen construction}.
\newblock Algebr. Geom. Topol. {\bf 21} (2021), 1267--1326.
\newblock ArXiv:1809.10924.

\bibitem{Bergner-Osorno-Ozornova-Rovelli-Scheimbauer:1901.03606}
{\sc Julia~E. Bergner, Angélica~M. Osorno, Viktoriya Ozornova, Martina
  Rovelli, {\rm and }Claudia~I. Scheimbauer}.
\newblock {\em Comparison of {W}aldhausen constructions}.
\newblock Ann. K-theory {\bf 6} (2021), 97--136.
\newblock ArXiv:1901.03606.

\bibitem{Bjoerner-Walker}
{\sc Anders Bj\"{o}rner {\rm and }James~W. Walker}.
\newblock {\em A homotopy complementation formula for partially ordered sets}.
\newblock European J. Combin. {\bf 4} (1983), 11--19.

\bibitem{Brion}
{\sc Michel Brion}.
\newblock {\em Representations of quivers}.
\newblock In {\em Geometric methods in representation theory. {I}}, vol.~24 of
  S\'emin. Congr., pp. 103--144. Soc. Math. France, Paris, 2012.

\bibitem{BFKAdv2006}
{\sc Christian Brouder, Alessandra Frabetti, {\rm and }Christian
  Krattenthaler}.
\newblock {\em Non-commutative {H}opf algebra of formal diffeomorphisms}.
\newblock Adv. Math. {\bf 200} (2006), 479--524.

\bibitem{Calaque-EbrahimiFard-Manchon:0806.2238}
{\sc Damien Calaque, Kurusch Ebrahimi-Fard, {\rm and }Dominique Manchon}.
\newblock {\em Two interacting {H}opf algebras of trees: a {H}opf-algebraic
  approach to composition and substitution of {B}-series}.
\newblock Adv. Appl. Math. {\bf 47} (2011), 282--308.
\newblock ArXiv:0806.2238.

\bibitem{Carlier:1903.07964}
{\sc Louis Carlier}.
\newblock {\em Hereditary species as monoidal decomposition spaces, comodule
  bialgebras, and operadic categories}.
\newblock Int. Math. Res. Notices (2020).
\newblock ArXiv:1903.07964.

\bibitem{Carlier:1801.07504}
{\sc Louis Carlier}.
\newblock {\em Incidence bicomodules, {M}\"{o}bius inversion and a {R}ota
  formula for infinity adjunctions}.
\newblock Algebr. Geom. Topol. {\bf 20} (2020), 169--213.
\newblock ArXiv:1801.07504.

\bibitem{Carlier:1812.09915}
{\sc Louis Carlier}.
\newblock {\em M{\"o}bius functions of directed restriction species and free
  operads, via the generalised {R}ota formula}.
\newblock Mediterr. J. Math. {\bf 18} (2021), 170.
\newblock ArXiv:1812.09915.

\bibitem{Carlier-Kock:1807.11858}
{\sc Louis Carlier {\rm and }Joachim Kock}.
\newblock {\em Antipodes of monoidal decomposition spaces}.
\newblock Commun. Contemp. Math. {\bf 22} (2020), 1850081, 15.
\newblock ArXiv:1807.11858.

\bibitem{Cartier-Foata}
{\sc Pierre Cartier {\rm and }Dominique Foata}.
\newblock {\em Probl{\`e}mes combinatoires de commutation et
  r{\'e}arrangements}.
\newblock No.~85 in Lecture Notes in Mathematics. Springer-Verlag, 1969.
\newblock Republished in the ``books'' section of the S{\'e}minaire
  Lotharingien de Combinatoire.

\bibitem{Cebrian-Forero:2211.07721}
{\sc Alex Cebrian {\rm and }Wilson Forero}.
\newblock {\em Directed hereditary species and decomposition spaces}.
\newblock Preprint, arXiv:2211.07721.

\bibitem{Connes-Kreimer:9808042}
{\sc Alain Connes {\rm and }Dirk Kreimer}.
\newblock {\em Hopf algebras, renormalization and noncommutative geometry}.
\newblock Comm. Math. Phys. {\bf 199} (1998), 203--242.
\newblock ArXiv:hep-th/9808042.

\bibitem{Content-Lemay-Leroux}
{\sc Mireille Content, Fran{\c{c}}ois Lemay, {\rm and }Pierre Leroux}.
\newblock {\em Cat\'egories de {M}\"obius et fonctorialit\'es: un cadre
  g\'en\'eral pour l'inversion de {M}\"obius}.
\newblock J. Combin. Theory Ser. A {\bf 28} (1980), 169--190.

\bibitem{Cooper-Young:thisvolume}
{\sc Benjamin Cooper {\rm and }Matthew Young}.
\newblock Hall algebras via 2-{S}egal spaces.
\newblock This volume.

\bibitem{Crapo}
{\sc Henry~H. Crapo}.
\newblock {\em The {M}\"obius function of a lattice}.
\newblock J. Combin. Theory {\bf 1} (1966), 126--131.

\bibitem{Doubilet:1972}
{\sc Peter Doubilet}.
\newblock {\em On the Foundations of Combinatorial Theory. VII: Symmetric
  Functions through the Theory of Distribution and Occupancy}.
\newblock Studies in Applied Mathematics {\bf 51} (1972), 377--396.

\bibitem{Doubilet:1974}
{\sc Peter Doubilet}.
\newblock {\em A {H}opf algebra arising from the lattice of partitions of a
  set}.
\newblock J. Algebra {\bf 28} (1974), 127--132.

\bibitem{Doubilet-Rota-Stanley}
{\sc Peter Doubilet, Gian-Carlo Rota, {\rm and }Richard Stanley}.
\newblock {\em On the foundations of combinatorial theory. {VI}. {T}he idea of
  generating function}.
\newblock In {\em Proceedings of the {S}ixth {B}erkeley {S}ymposium on
  {M}athematical {S}tatistics and {P}robability ({U}niv. {C}alifornia,
  {B}erkeley, {C}alif., 1970/1971), {V}ol. {II}: {P}robability theory}, pp.
  267--318. Univ. California Press, Berkeley, Calif., 1972.

\bibitem{Dur:1986}
{\sc Arne D{\"u}r}.
\newblock {\em M\"obius functions, incidence algebras and power series
  representations}, vol. 1202 of Lecture Notes in Mathematics.
\newblock Springer-Verlag, Berlin, 1986.

\bibitem{Dyckerhoff:thisvolume}
{\sc Tobias Dyckerhoff}.
\newblock Higher {S}egal spaces.
\newblock This volume.

\bibitem{Dyckerhoff:1505.06940}
{\sc Tobias Dyckerhoff}.
\newblock {\em Higher Categorical Aspects of {H}all Algebras}.
\newblock In Dolors Herbera, Wolfgang Pitsch, {\rm and }Santiago Zarzuela,
  editors, {\em Building Bridges Between Algebra and Topology}, pp. 1--61.
  Springer International Publishing, Cham, 2018.

\bibitem{Dyckerhoff-Kapranov:1403.5799}
{\sc Tobias Dyckerhoff {\rm and }Mikhail Kapranov}.
\newblock {\em Crossed simplicial groups and structured surfaces}.
\newblock In {\em Stacks and categories in geometry, topology, and algebra.
  CATS4 conference on higher categorical structures and their interactions with
  algebraic geometry, algebraic topology and algebra, CIRM, Luminy, France,
  July 2--7, 2012.}, pp. 37--110. American Mathematical Society, 2015.
\newblock ArXiv:1403.5799.

\bibitem{Dyckerhoff-Kapranov:1306.2545}
{\sc Tobias Dyckerhoff {\rm and }Mikhail Kapranov}.
\newblock {\em Triangulated surfaces in triangulated categories}.
\newblock Journal of the European Mathematical Society {\bf 20} (April 2018),
  1473--1524.

\bibitem{Dyckerhoff-Kapranov:1212.3563}
{\sc Tobias Dyckerhoff {\rm and }Mikhail Kapranov}.
\newblock {\em Higher {Segal} spaces}, vol. 2244 of Lect. Notes Math.
\newblock Cham: Springer, 2019.
\newblock ArXiv:1212.3563.

\bibitem{Ehrenborg}
{\sc Richard Ehrenborg}.
\newblock {\em On posets and {H}opf algebras}.
\newblock Adv. Math. {\bf 119} (1996), 1--25.

\bibitem{Fauvet-Foissy-Manchon:1503.03820}
{\sc Fr\'{e}d\'{e}ric Fauvet, Lo\"{\i}c Foissy, {\rm and }Dominique Manchon}.
\newblock {\em The {H}opf algebra of finite topologies and mould composition}.
\newblock Ann. Inst. Fourier (Grenoble) {\bf 67} (2017), 911--945.

\bibitem{Feller-Garner-Kock-Proulx-Weber:1905.09580}
{\sc Matthew Feller, Richard Garner, Joachim Kock, May~U. Proulx, {\rm and
  }Mark Weber}.
\newblock {\em Every 2-{S}egal space is unital}.
\newblock Commun. Contemp. Math. {\bf 23} (2021), 2050055.
\newblock ArXiv:1905.09580.

\bibitem{Figueroa-GraciaBondia:0408145}
{\sc H{\'e}ctor Figueroa {\rm and }Jos{\'e}~M. Gracia-Bond{\'{\i}}a}.
\newblock {\em Combinatorial {H}opf algebras in quantum field theory. {I}}.
\newblock Rev. Math. Phys. {\bf 17} (2005), 881--976.
\newblock ArXiv:hep-th/0408145.

\bibitem{Foissy:1702.05344}
{\sc Lo{\"i}c Foissy}.
\newblock {\em Algebraic structures associated to operads}.
\newblock Preprint, arXiv:1702.05344.

\bibitem{Foissy:2002I}
{\sc Lo{\"i}c Foissy}.
\newblock {\em Les alg\`ebres de {H}opf des arbres enracin\'es d\'ecor\'es.
  {I}}.
\newblock Bull. Sci. Math. {\bf 126} (2002), 193--239.

\bibitem{GaKaTo2016}
{\sc Imma G{\'a}lvez-Carrillo, Ralph~M. Kaufmann, {\rm and }Andrew Tonks}.
\newblock {\em Three {H}opf algebras from number theory, physics \& topology,
  and their common background I: operadic \& simplicial aspects}.
\newblock Communications in Number Theory and Physics {\bf 14} (2020), 1--90.

\bibitem{gkt:crapo}
{\sc Imma G\'{a}lvez-Carrillo, Joachim Kock, {\rm and }Andrew Tonks}.
\newblock {\em Convex decomposition spaces and {C}rapo complementation
  formula}.
\newblock Preprint, arXiv:2409.03742, Preprint.

\bibitem{GKT:1404.3202}
{\sc Imma G{\'a}lvez-Carrillo, Joachim Kock, {\rm and }Andrew Tonks}.
\newblock Decomposition spaces, incidence algebras and {M}{\"o}bius inversion.
\newblock (old omnibus version).

\bibitem{GKT:QSym}
{\sc Imma G\'{a}lvez-Carrillo, Joachim Kock, {\rm and }Andrew Tonks}.
\newblock Decomposition spaces of quasi-symmetric functions.
\newblock Unpublished/in preparation.

\bibitem{GKT:sym}
{\sc Imma G\'{a}lvez-Carrillo, Joachim Kock, {\rm and }Andrew Tonks}.
\newblock Decomposition spaces of symmetric functions.
\newblock Unpublished/in preparation.

\bibitem{GalvezCarrillo-Kock-Tonks:1207.6404}
{\sc Imma G{\'a}lvez-Carrillo, Joachim Kock, {\rm and }Andrew Tonks}.
\newblock {\em Groupoids and {F}a{\`a} di {B}runo formulae for {G}reen
  functions in bialgebras of trees}.
\newblock Adv. Math. {\bf 254} (2014), 79--117.
\newblock ArXiv:1207.6404.

\bibitem{GKT:DSIAMI-1}
{\sc Imma G{\'a}lvez-Carrillo, Joachim Kock, {\rm and }Andrew Tonks}.
\newblock {\em Decomposition spaces, incidence algebras and {M}\"{o}bius
  inversion {I}: {B}asic theory}.
\newblock Adv. Math. {\bf 331} (2018), 952--1015.
\newblock ArXiv:1512.07573.

\bibitem{GKT:DSIAMI-2}
{\sc Imma G{\'a}lvez-Carrillo, Joachim Kock, {\rm and }Andrew Tonks}.
\newblock {\em Decomposition spaces, incidence algebras and {M}\"{o}bius
  inversion {II}: {C}ompleteness, length filtration, and finiteness}.
\newblock Adv. Math. {\bf 333} (2018), 1242--1292.
\newblock ArXiv:1512.07577.

\bibitem{GKT:MI}
{\sc Imma G{\'a}lvez-Carrillo, Joachim Kock, {\rm and }Andrew Tonks}.
\newblock {\em Decomposition spaces, incidence algebras and {M}\"{o}bius
  inversion {III}: {T}he decomposition space of {M}\"{o}bius intervals}.
\newblock Adv. Math. {\bf 334} (2018), 544--584.
\newblock ArXiv:1512.07580.

\bibitem{GKT:HLA}
{\sc Imma G{\'a}lvez-Carrillo, Joachim Kock, {\rm and }Andrew Tonks}.
\newblock {\em Homotopy linear algebra}.
\newblock Proc. Royal Soc. Edinburgh A {\bf 148} (2018), 293--325.
\newblock ArXiv:1602.05082.

\bibitem{GKT:corrigendum}
{\sc Imma G{\'a}lvez-Carrillo, Joachim Kock, {\rm and }Andrew Tonks}.
\newblock {\em Corrigendum to ``Decomposition spaces, incidence algebras and
  {M}{\"o}bius inversion {II}: Completeness, length filtration, and
  finiteness'' [Adv. Math. 333 (2018) 1242--1292]}.
\newblock Adv. Math. {\bf 371} (2020), 107267.

\bibitem{GKT:restriction}
{\sc Imma G{\'a}lvez-Carrillo, Joachim Kock, {\rm and }Andrew Tonks}.
\newblock {\em Decomposition spaces and restriction species}.
\newblock Int. Math. Res. Not. (September 2020), 7558--7616.

\bibitem{Gambino-Kock:0906.4931}
{\sc Nicola Gambino {\rm and }Joachim Kock}.
\newblock {\em Polynomial functors and polynomial monads}.
\newblock Math. Proc. Cambridge Phil. Soc. {\bf 154} (2013), 153--192.
\newblock ArXiv:0906.4931.

\bibitem{Goldman-Rota:1970}
{\sc Jay Goldman {\rm and }Gian-Carlo Rota}.
\newblock {\em On the foundations of combinatorial theory. {IV}. {F}inite
  vector spaces and {E}ulerian generating functions}.
\newblock Stud. Appl. Math. {\bf 49} (1970), 239--258.

\bibitem{Goncharov:Hopf}
{\sc Alexander~B. Goncharov}.
\newblock {\em Galois symmetries of fundamental groupoids and noncommutative
  geometry}.
\newblock Duke Math. J. {\bf 128} (2005), 209--284.
\newblock ArXiv:math/0208144.

\bibitem{Hackney:thisvolume}
{\sc Philip Hackney}.
\newblock The decomposition space perspective.
\newblock This volume.

\bibitem{Hackney-Kock:2210.11192}
{\sc Philip Hackney {\rm and }Joachim Kock}.
\newblock {\em Free decomposition spaces}.
\newblock Collectanea Mathematica (2024).
\newblock ArXiv:2210.11192.

\bibitem{Hermida:repr-mult}
{\sc Claudio Hermida}.
\newblock {\em Representable multicategories}.
\newblock Adv. Math. {\bf 151} (2000), 164--225.

\bibitem{JoniRotaMR544721}
{\sc Saj-nicole~A. Joni {\rm and }Gian-Carlo Rota}.
\newblock {\em Coalgebras and bialgebras in combinatorics}.
\newblock Stud. Appl. Math. {\bf 61} (1979), 93--139.

\bibitem{JoyalMR633783}
{\sc Andr{\'e} Joyal}.
\newblock {\em Une th\'eorie combinatoire des s\'eries formelles}.
\newblock Adv. in Math. {\bf 42} (1981), 1--82.

\bibitem{Joyal-Street:tensor-calculus}
{\sc Andr{\'e} Joyal {\rm and }Ross Street}.
\newblock {\em The geometry of tensor calculus. {I}}.
\newblock Adv. Math. {\bf 88} (1991), 55--112.

\bibitem{Joyal-Street:GLn}
{\sc Andr{\'e} Joyal {\rm and }Ross Street}.
\newblock {\em The category of representations of the general linear groups
  over a finite field}.
\newblock J. Algebra {\bf 176} (1995), 908--946.

\bibitem{Kock:0807}
{\sc Joachim Kock}.
\newblock {\em Polynomial functors and trees}.
\newblock Internat. Math. Res. Notices {\bf 2011} (2011), 609--673.
\newblock ArXiv:0807.2874.

\bibitem{Kock:MFPS28}
{\sc Joachim Kock}.
\newblock {\em Data types with symmetries and polynomial functors over
  groupoids}.
\newblock In {\em Proceedings of the 28th Conference on the Mathematical
  Foundations of Programming Semantics (Bath, 2012)}, vol. 286 of Electr. Notes
  in Theoret. Comp. Sci., pp. 351--365, 2012.
\newblock ArXiv:1210.0828.

\bibitem{Kock:1109.5785}
{\sc Joachim Kock}.
\newblock {\em Categorification of {H}opf algebras of rooted trees}.
\newblock Cent. Eur. J. Math. {\bf 11} (2013), 401--422.
\newblock ArXiv:1109.5785.

\bibitem{Kock:1407.3744}
{\sc Joachim Kock}.
\newblock {\em Graphs, hypergraphs, and properads}.
\newblock Collect. Math. {\bf 67} (2016), 155--190.
\newblock ArXiv:1407.3744.

\bibitem{Kock:1512.03027}
{\sc Joachim Kock}.
\newblock {\em {Polynomial functors and combinatorial Dyson–Schwinger
  equations}}.
\newblock Journal of Mathematical Physics {\bf 58} (04 2017), 041703.

\bibitem{Kock:1912.11320}
{\sc Joachim Kock}.
\newblock {\em The incidence comodule bialgebra of the {B}aez--{D}olan
  construction}.
\newblock Adv. Math. {\bf 383} (2021), Paper No. 107693, 79.
\newblock ArXiv:1912.11320.

\bibitem{Kock:2005.05108}
{\sc Joachim Kock}.
\newblock {\em Whole-grain {P}etri nets and processes}.
\newblock J. ACM. {\bf 70} (2023), 1--58.
\newblock ArXiv:2005.05108.

\bibitem{Kock-Joyal-Batanin-Mascari:0706}
{\sc Joachim Kock, Andr{\'e} Joyal, Michael Batanin, {\rm and
  }Jean-Fran{\c{c}}ois Mascari}.
\newblock {\em Polynomial functors and opetopes}.
\newblock Adv. Math. {\bf 224} (2010), 2690--2737.
\newblock ArXiv:0706.1033.

\bibitem{Kock-Spivak:1807.06000}
{\sc Joachim Kock {\rm and }David~I. Spivak}.
\newblock {\em Decomposition-space slices are toposes}.
\newblock Proc. Amer. Math. Soc. {\bf 148} (2020), 2317--2329.
\newblock ArXiv:1807.06000.

\bibitem{Kock-Weber:1609.03276}
{\sc Joachim Kock {\rm and }Mark Weber}.
\newblock {\em {F}aà di {B}runo for operads and internal algebras}.
\newblock Journal of the London Mathematical Society {\bf 99} (2019), 919--944.

\bibitem{LawvereMenniMR2720184}
{\sc F.~William Lawvere {\rm and }Mat{\'i}as Menni}.
\newblock {\em The {H}opf algebra of {M}\"obius intervals}.
\newblock Theory Appl. Categ. {\bf 24} (2010), 221--265.

\bibitem{Leinster:1201.0413}
{\sc Tom Leinster}.
\newblock {\em Notions of {M}{\"o}bius inversion}.
\newblock Bull. Belg. Math. Soc. {\bf 19} (2012), 911--935.
\newblock ArXiv:1201.0413.

\bibitem{Leroux:1975}
{\sc Pierre Leroux}.
\newblock {\em Les cat{\'e}gories de {M}{\"o}bius}.
\newblock Cahiers Topologie G{\'e}om. Diff{\'e}rentielle {\bf 16} (1976),
  280--282.

\bibitem{Lothaire:MR675953}
{\sc M.~Lothaire}.
\newblock {\em Combinatorics on words}, vol.~17 of Encyclopedia of Mathematics
  and its Applications.
\newblock Addison-Wesley Publishing Co., Reading, Mass., 1983.
\newblock A collective work by Dominique Perrin, Jean Berstel, Christian
  Choffrut, Robert Cori, Dominique Foata, Jean Eric Pin, Guiseppe Pirillo,
  Christophe Reutenauer, Marcel-P. Sch{\"u}tzenberger, Jacques Sakarovitch and
  Imre Simon, With a foreword by Roger Lyndon, Edited and with a preface by
  Perrin.

\bibitem{MacLane:categories}
{\sc Saunders {Mac~Lane}}.
\newblock {\em Categories for the working mathematician, second edition}.
\newblock No.~5 in Graduate Texts in Mathematics. Springer-Verlag, New York,
  1998.

\bibitem{Maia-Mendez:0503436}
{\sc Manuel Maia {\rm and }Miguel M{\'e}ndez}.
\newblock {\em On the arithmetic product of combinatorial species}.
\newblock Discrete Math. {\bf 308} (2008), 5407--5427.
\newblock ArXiv:math/0503436.

\bibitem{Manchon:MR2921530}
{\sc Dominique Manchon}.
\newblock {\em On bialgebras and {H}opf algebras of oriented graphs}.
\newblock Confluentes Math. {\bf 4} (2012), 1240003, 10.

\bibitem{Manchon:Abelsymposium}
{\sc Dominique Manchon}.
\newblock {\em A review on comodule-bialgebras}.
\newblock In {\em Computation and Combinatorics in Dynamics, Stochastics and
  Control. Abel Symposium 2016.}, vol.~13 of Abel Symposia. Springer, Cham,
  2016.

\bibitem{Manin:MR2562767}
{\sc Yuri~I. Manin}.
\newblock {\em A course in mathematical logic for mathematicians}, vol.~53 of
  Graduate Texts in Mathematics.
\newblock Springer, New York, second edition, 2010.
\newblock Chapters I--VIII translated from the Russian by Neal Koblitz, With
  new chapters by Boris Zilber and the author.

\bibitem{Manin:0904.4921}
{\sc Yuri~I. Manin}.
\newblock {\em Renormalization and computation. {I}: {Motivation} and
  background}.
\newblock In {\em Operads 2009. Proceedings of the school and conference,
  Luminy, France, April 20--30, 2009}, pp. 181--222. Paris: Soci{\'e}t{\'e}
  Math{\'e}matique de France, 2011.
\newblock ArXiv:0904.4921.

\bibitem{MoravaFdB1993}
{\sc Jack Morava}.
\newblock {\em Some examples of {H}opf algebras and {T}annakian categories}.
\newblock In {\em Algebraic topology ({O}axtepec, 1991)}, vol. 146 of Contemp.
  Math., pp. 349--359. Amer. Math. Soc., Providence, RI, 1993.

\bibitem{Morrison:0512052}
{\sc Kent~E. Morrison}.
\newblock {\em An introduction to {$q$}-species}.
\newblock Electron. J. Combin. {\bf 12} (2005).
\newblock ArXiv:math/0512052, 15pp. (electronic).

\bibitem{MuntheKaas:BIT95}
{\sc Hans Munthe-Kaas}.
\newblock {\em Lie-{B}utcher theory for {R}unge-{K}utta methods}.
\newblock BIT {\bf 35} (1995), 572--587.

\bibitem{Ozornova:thisvolume}
{\sc Viktoriya Ozornova}.
\newblock 2-{S}egal spaces and the {$S$}-construction.
\newblock This volume.

\bibitem{Poguntke:1808.04165}
{\sc Thomas Poguntke}.
\newblock {\em Equivariant motivic Hall algebras}.
\newblock Preprint, arXiv:1808.04165.

\bibitem{Ringel:Hall}
{\sc Claus~Michael Ringel}.
\newblock {\em Hall algebras and quantum groups}.
\newblock Invent. Math. {\bf 101} (1990), 583--591.

\bibitem{Rota:Moebius}
{\sc Gian-Carlo Rota}.
\newblock {\em On the foundations of combinatorial theory. {I}. {T}heory of
  {M}\"obius functions}.
\newblock Z. Wahrscheinlichkeitstheorie und Verw. Gebiete {\bf 2} (1964),
  340--368.

\bibitem{Rovelli:thisvolume}
{\sc Martina Rovelli}.
\newblock The {W}aldhausen construction as an equivalence between stable
  augmented double {S}egal spaces and 2-{S}egal spaces.
\newblock This volume.

\bibitem{Schiffmann:0611617}
{\sc Olivier Schiffmann}.
\newblock {\em Lectures on {H}all algebras}.
\newblock In {\em Geometric methods in representation theory. {II}}, vol.~24 of
  S\'emin. Congr., pp. 1--141. Soc. Math. France, Paris, 2012.
\newblock ArXiv:math/0611617.

\bibitem{Schmitt:antipodes}
{\sc William~R. Schmitt}.
\newblock {\em Antipodes and incidence coalgebras}.
\newblock J. Combin. Theory Ser. A {\bf 46} (1987), 264--290.

\bibitem{Schmitt:hacs}
{\sc William~R. Schmitt}.
\newblock {\em Hopf algebras of combinatorial structures}.
\newblock Canad. J. Math. {\bf 45} (1993), 412--428.

\bibitem{Schmitt:1994}
{\sc William~R. Schmitt}.
\newblock {\em Incidence {H}opf algebras}.
\newblock J. Pure Appl. Algebra {\bf 96} (1994), 299--330.

\bibitem{Stanley:MR513004}
{\sc Richard~P. Stanley}.
\newblock {\em Generating functions}.
\newblock In {\em Studies in combinatorics}, vol.~17 of MAA Stud. Math., pp.
  100--141. Math. Assoc. America, Washington, D.C., 1978.

\bibitem{Stanley}
{\sc Richard~P. Stanley}.
\newblock {\em Enumerative combinatorics. {V}ol. {I}}.
\newblock The Wadsworth \& Brooks/Cole Mathematics Series. Wadsworth \&
  Brooks/Cole Advanced Books \& Software, Monterey, CA, 1986.
\newblock With a foreword by Gian-Carlo Rota.

\bibitem{Stanley:volII}
{\sc Richard~P. Stanley}.
\newblock {\em Enumerative combinatorics. {V}ol. 2}, vol.~62 of Cambridge
  Studies in Advanced Mathematics.
\newblock Cambridge University Press, Cambridge, 1999.

\bibitem{Street:categorical-structures}
{\sc Ross Street}.
\newblock {\em Categorical structures}.
\newblock In {\em Handbook of algebra, {V}ol.\ 1}, pp. 529--577. North-Holland,
  Amsterdam, 1996.

\bibitem{Sweedler}
{\sc Moss~E. Sweedler}.
\newblock {\em Hopf algebras}.
\newblock W.A. Benjamin, Inc., New York, 1969.

\bibitem{Takeuchi:1971}
{\sc Mitsuhiro Takeuchi}.
\newblock {\em Free {H}opf algebras generated by coalgebras}.
\newblock J. Math. Soc. Japan {\bf 23} (1971), 561--582.

\bibitem{Toen:0501343}
{\sc Bertrand To{\"e}n}.
\newblock {\em Derived {H}all algebras}.
\newblock Duke Math. J. {\bf 135} (2006), 587--615.
\newblock ArXiv:math/0501343.

\bibitem{Walde:1611.08241}
{\sc Tashi Walde}.
\newblock {\em Hall monoidal categories and categorical modules}.
\newblock Preprint, arXiv:1611.08241.

\bibitem{Waldhausen}
{\sc Friedhelm Waldhausen}.
\newblock {\em Algebraic {$K$}-theory of spaces}.
\newblock In {\em Algebraic and geometric topology ({N}ew {B}runswick,
  {N}.{J}., 1983)}, vol. 1126 of Lecture Notes in Mathematics, pp. 318--419.
  Springer, Berlin, 1985.

\bibitem{Webb}
{\sc Peter Webb}.
\newblock {\em An introduction to the representations and cohomology of
  categories}.
\newblock In {\em Group representation theory}, pp. 149--173. EPFL Press,
  Lausanne, 2007.

\bibitem{Weber:TAC13}
{\sc Mark Weber}.
\newblock {\em Generic morphisms, parametric representations and weakly
  {C}artesian monads}.
\newblock Theory Appl. Categ. {\bf 13} (2004), 191--234.
\newblock (electronic).

\bibitem{Weber:TAC18}
{\sc Mark Weber}.
\newblock {\em Familial 2-functors and parametric right adjoints}.
\newblock Theory Appl. Categ. {\bf 18} (2007), 665--732.
\newblock (electronic).

\bibitem{Weber:1412.7599}
{\sc Mark Weber}.
\newblock {\em Operads as polynomial 2-monads}.
\newblock Theory Appl. Categ. {\bf 30} (2015), 1659--1712.
\newblock ArXiv:1412.7599.

\bibitem{Young:1611.09234}
{\sc Matthew Young}.
\newblock {\em Relative 2–{S}egal spaces}.
\newblock Algebraic \& Geometric Topology {\bf 18} (March 2018), 975--1039.

\end{thebibliography}
\end{document}